\documentclass[]{amsart}

\usepackage{allneeded}
\usepackage{enumerate}
\usepackage{xypic}

\begin{document}

\title{On the Slice-Ribbon Conjecture for Montesinos knots}%

\author{Ana G. Lecuona}%

\email{lecuona@mail.dm.unipi.it}%

\address{Dipartimento di Matematica, Universit\`a di Pisa, 56127 Pisa, Italy}%

\subjclass[2000]{57M25}%
\keywords{Montesinos links, slice-ribbon conjecture, rational homology balls}%

\date{October 23, 2009}%
\begin{abstract}
We establish the slice-ribbon conjecture for a large family of Montesinos' knots by means of Donaldson's theorem  on the intersection forms of definite $4$-manifolds.
\end{abstract}
\maketitle

\begin{quote}
\begin{footnotesize}
\tableofcontents
\end{footnotesize}
\end{quote}

\section{Introduction}\label{s:int}

A celebrated open  question in knot theory is the slice-ribbon conjecture due to Fox \cite{b:Fo} in 1962. Recall that a knot $K\subseteq S^3$ is called (smoothly) \textbf{slice} if it bounds a smoothly and properly embedded disc $D^2\hookrightarrow D^4$ in the 4-ball $D^4$, and \textbf{ribbon} if it bounds an immersed disc $D^2\looparrowright S^3$ with only ribbon singularities (see \cite[p.\ 70]{b:Ka} for the definition). It is not difficult to prove that every ribbon knot is slice: simply push the ribbon singularities into the fourth dimension to obtain an embedded $D^2\hookrightarrow D^4$. The converse, whether every slice knot is ribbon, is the well known Fox's slice-ribbon conjecture.

In recent works \cite{b:Li,b:GJ} the conjecture has been proved for all $2$-bridge knots and for infinitely many $3$-stranded pretzel knots. Both families of knots are particular cases of the significantly broader family of Montesinos links (first constructed in \cite{b:Mo2}). In \cite{b:Wi} it is shown that no member of a five parameter family of Montesinos knots is slice. In the present work, in order to introduce Montesinos links we follow \cite[p.\ 18, Theorem~$(c)$]{b:Si}, where these links are defined as the boundary of $2$-di\-men\-sional plumbings with star-shaped plumbing graphs. A \textbf{star-shaped} graph  is a connected tree with a distinguished vertex $v_0$ (called the central vertex) such that the degree of any vertex other than the central one is $\leq 2$. In a weighted star-shaped graph $\Gamma$ each vertex represents a twisted band, that is a $D^1$-bundle over $S^1$, embedded in $S^3$, with the number of half-twists given by the weight of the vertex. Bands are plumbed together precisely when the corresponding vertices are adjacent (see Figure~\ref{f:bandas} for an example). The result of this plumbing construction is a surface $B_\Gamma\subset S^3$ whose boundary $\Ml_{\Gamma}$ is, by definition, a Montesinos link. Since $S^3=\partial D^4$, we can push the interior of $B_\Gamma$ into the interior of $D^4$. It follows that the double covering of $D^{4}$ branched over $B_{\Gamma}$ is the $4$-dimensional plumbing $M_\Gamma$, obtained by plumbing $D^{2}$-bundles over $S^{2}$ according to the graph $\Gamma$, which defined the Montesinos link. The boundary $Y_{\Gamma}:=\partial M_{\Gamma}$ is a Seifert space (see \cite{b:Ra} for a proof) with as many singular fibers as legs of the graph $\Gamma$. A \textbf{leg} of a star-shaped graph is any connected component of the graph obtained by removing the central vertex. The involution $u$ that defines the covering $M_{\Gamma}\rightarrow D^{4}\simeq M_{\Gamma}/u$, turns the Seifert space $Y_{\Gamma}$ into the double covering of $S^{3}$ branched along the Montesinos link $\Ml_{\Gamma}$. Restricting our attention to three-legged star-shaped graphs $\Gamma$, it is well known \cite[Theorem~12.29]{b:BZ} that the Seifert space $Y_{\Gamma}$ is the double covering of $S^{3}$ branched along exactly one Montesinos link (up to link isotopy).

In the present work, following in part the approach of \cite{b:Li}, we study the family $\wp$ of all three-legged connected plumbing graphs $\Gamma$ such that:
\begin{itemize}
\item $I(\Gamma):=\sum_{i=0}^n (a_i-3)<-1$, where by $-a_0,...,-a_n$ we denote the weights of the vertices of $\Gamma$; and
\item the central vertex has weight less or equal to $-3$ and every non central vertex has weight less or equal to $-2$.
\end{itemize} 
Our choice of the family $\wp$ is motivated by what follows. In \cite{b:Li}, the fact that a linear graph $\Gamma$ has a dual graph $\Gamma'$ such that $Y_\Gamma=-Y_{\Gamma'}$ and $M_\Gamma, M_{\Gamma'}$ are both negative definite, is strongly used. Indeed, this property allows one to assume, without loss of generality in the case of $2$-bridge links, that $I(\cdot)<0$. In our case, a \virg{dual} graph $\Gamma'$ still exists, but it is always the case that one of $M_\Gamma,M_{\Gamma'}$ is indefinite. This forces us to restrict to the case of plumbing graphs $\Gamma$ such that $I(\Gamma)<0$. In the present work, we deal with  the case $I(\Gamma)<-1$. For the case $I(\Gamma)=-1$ so far we have obtained partial results \cite{b:Le}. We hope to return to this case in a future paper. The second condition defining the family $\wp$ is due to technical reasons. More specifically, we will show that, for every $\Gamma\in\wp$ such that the Seifert space $Y_{\Gamma}$ bounds a rational homology ball, we have $I(\Gamma)\in\{-4,-3,-2\}$ and we will study separately the three possible cases. It turns out that, allowing the central vertex of the graph $\Gamma$ to have weight $-2$, the fact that $Y_{\Gamma}$ bounds a rational homology ball does not imply that $I(\Gamma)$ is bounded from below. 

We note that, for every Montesinos knot $\Ml_P$ with $P\in\wp$ we have that $\Ml_P$ is neither a $3$-stranded pretzel knot nor a member of the family studied in \cite{b:Wi}, since these two families of knots have both associated negative graphs with central vertex of weight $-2$. 

Our main result is the following.

\begin{thm}\label{t:int}
Consider $\Gamma\in\wp$. The Seifert space $Y_\Gamma$ is the boundary of a rational homology ball $W$ if and only if there exist a surface $\Sigma$ and a ribbon immersion $\Sigma\looparrowright S^{3}$ such that $\partial\Sigma=\Ml_{\Gamma}$ and $\chi (\Sigma)=1$.
\end{thm}

Our analysis gives a complete list, Theorems~\ref{l:strings1} and \ref{l:2}, of the Seifert spaces $Y_\Gamma$ with $\Gamma\in\wp$ which bound rational homology balls, providing a partial anser to a question of Andrew Casson \cite[Problem~4.5]{b:Ki}.

Theorem \ref{t:int} immediately implies, 

\begin{cor}\label{c:int}
The slice-ribbon conjecture holds true for all Montesinos knots $\Ml_{\Gamma}$ with $\Gamma\in\wp$.
\end{cor}
\begin{proof}
Let $\Gamma\in\wp$ be such that the knot $\Ml_{\Gamma}\subset S^{3}$ is slice. Let $D^{2}\hookrightarrow D^{4}$ be a smooth slicing disc for $\Ml_{\Gamma}$ and $W$ the $2$-fold cover of $D^4$ branched along $D^2$. It is well known \cite[Lemma~17.2]{b:Ka} that $W$ is a rational homology ball and that $\partial W=Y_{\Gamma}$. It follows immediately from Theorem~\ref{t:int} that the knot $\Ml_{\Gamma}$ is ribbon.
\end{proof}

In \cite{b:BS,b:CH,b:Fi,b:FS,b:St} other families of Seifert spaces bounding rational homology balls were constructed. A case by case comparison shows that the intersection between these families and the family studied in this work is essentially empty.

The strategy to approach Theorem~\ref{t:int} can be sketched as follows. All the plumbing graphs $\Gamma\in\wp$ give rise to negative definite $4$-manifolds $M_\Gamma$ with boundary $\partial M_\Gamma=Y_\Gamma$. Therefore, if we assume the existence of a rational homology ball $W$ with $\partial W=Y_\Gamma$, we can build a closed, oriented, negative definite, $4$-manifold $X_\Gamma$ as $M_\Gamma\cup_{Y_\Gamma} (-W)$. Donaldson's celebrated theorem \cite{b:Do} implies that the intersection lattice $(\Z^n,Q_{X_\Gamma})$ is isomorphic to the standard negative diagonal intersection lattice $(\Z^n,-\mathrm{Id})$, where $n:=b_2(X_\Gamma)=b_2(M_\Gamma)$. Therefore, the intersection lattice $(\Z^n,Q_{M_\Gamma})$ must embed in the standard negative definite intersection lattice of equal rank; that is, there must exist a monomorphism $\iota:\Z^n\rightarrow\Z^n$ such that $Q_{M_\Gamma}(\alpha,\beta)=-\mathrm{Id}(\iota(\alpha),\iota(\beta))$ for every $\alpha,\beta\in\Z^n\cong H_2(M_\Gamma;\Z)/\Tors$. The first step in the proof of Theorem~\ref{t:int} consists of a careful analysis of this obstruction, determining which among the intersection lattices $(\Z^n,Q_{M_\Gamma})$, with $\Gamma\in\wp$, admit an embedding into the standard negative diagonal lattice. This analysis leads to a list of candidates $P\in\wp$ (Theorems~\ref{l:strings1} and \ref{l:2}) such that the associated Seifert spaces $Y_P$ may bound rational homology balls. In the first part of this paper we use techniques and results from \cite{b:Li}.

In the second step of the proof, we find explicitly, for each $\Ml_P$ with $P$ in the list of candidates, the ribbon surface claimed in Theorem~\ref{t:int}. In \cite{b:Li,b:GJ}, once they arrive to their respective lists of candidates, the construction of the surface is not directly related to the analysis done in the first part of the proof. Our approach to the construction of the ribbon surfaces is different. In a first attempt we tried to find systematically the bands that describe the ribbon surfaces on the diagrams of our candidates, but we were soon discouraged after realizing how complicated and random they seemed to be in the standard projection of a Montesinos link (see Figure~\ref{f:band} for an example). 
\begin{figure}
\begin{center}
\includegraphics{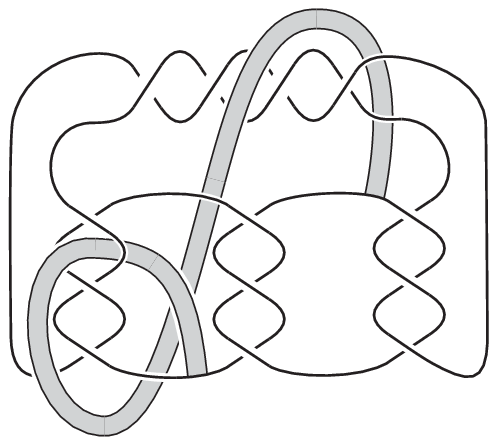}
\hcaption{The Montesinos knot $\Ml(-4;(-3,1)(-3,1)(-3,1))$, with the notation inherited from the classical notation for Seifert spaces, is ribbon. In fact, performing a ribbon move along the gray band we obtain two unlinked unknots.}
\label{f:band}
\end{center}
\end{figure}
To overcome this difficulty, instead of working with the link diagrams corresponding to the candidates $P$, we focus our attention on the corresponding $4$-dimensional plumbings $M_P$. The analysis done in the first step of the proof suggests how to modify $M_P$ with the addition of a $2$-handle, yielding a $4$-manifold $M'$ whose boundary is the double cover of $S^3$ branched over a link. It turns out that this link bounds a surface with Euler characteristic equal to $2$. A theorem due to Montesinos \cite{b:Mo} implies that the added $2$-handle corresponds to a ribbon move on the initial link $\Ml_P$ and this concludes the proof.

The paper is organized as follows: in Section~\ref{s:prel} we give a quick overview of some basic facts on Seifert spaces and Montesinos links and we introduce the necessary definitions in order to state Theorems~\ref{l:strings1} and \ref{l:2}. These give a list of candidates $P\subseteq\wp$ such that the knots $\Ml_P$ may be slice. We postpone the long and technical proof of Theorems~\ref{l:strings1} and \ref{l:2} to Sections~\ref{s:contractions} to \ref{s:ss}. A brief sketch of the proof can be found at the end of Section~\ref{s:prel}. In Section~\ref{c:bande} we construct a ribbon surface with Euler characteristic $1$ for every Montesinos' link stemming from the list of candidates and we prove Theorem~\ref{t:int}.
\\

\begin{quote}
\footnotesize
\textbf{Acknowledgments.} This work is part of my Ph.D.\ thesis at the University of Pisa. I wish to thank my advisor, Paolo Lisca, without whom this work would not have been completed. His expertise, guidance and encouragement have been precious throughout these years. I also thank Jose F. Fernando and Marco Mazzucchelli for their comments on an earlier version of the manuscript,  which led to a clearer exposition in the paper.
\end{quote}

\section{The candidates}\label{s:prel}
In this section we recall some terminology and well-known results concerning Seifert spaces and Montesinos links. Furthermore, we introduce the necessary concepts to state Theorems~\ref{l:strings1} and \ref{l:2}, which will be proved in Sections~\ref{s:contractions} to \ref{s:ss}.

\vspace{1.5mm}
\noindent\textbf{Seifert spaces and Montesinos links. }
Let $\Gamma$ be a \textbf{plumbing graph}, that is, a graph in which every vertex $v_i$ carries an integer weight $a_i$, $i=1,...,n$. Associated to each vertex $v_i$ is the $4$-dimensional disc bundle $X\rightarrow S^2$ with Euler number $a_i$. If the vertex $v_i$ has $d_i$ edges connected to it in the graph $\Gamma$, we choose $d_i$ disjoint discs in the base of $X\rightarrow S^2$ and call the disc bundle over the $j$th disc $B_{ij}=D^2\times D^2$. When two vertices are connected by an edge, we identify $B_{ij}$ with $B_{kl}$ by exchanging the base and fiber coordinates and smoothing the corners. This pasting operation is called \textbf{plumbing} (for a more general treatment we refer the reader to \cite{b:GS}), and the resulting smooth $4$-manifold $M_\Gamma$ is said to be obtained by plumbing according to $\Gamma$. The Kirby diagram of $M_{\Gamma}$ has an unknot for each vertex of the tree, and whenever two vertices are joined by an edge the corresponding unknots will be linked forming a Hopf link. Each framing will be the Euler number of the corresponding $D^{2}$-bundle. 

The group $H_2(M_\Gamma;\Z)$ has a natural basis represented by the zero-sections of the plumbed bundles. We note that all these sections are embedded $2$-spheres, and they can be oriented in such a way that the intersection form of $M_\Gamma$ will be given by the matrix $Q_\Gamma=(q_{ij})_{i,j=1,...,n}$ with the entries
$$q_{ij}
=
\left\{ 
\begin{array}{cl}
a_i & \mbox{if } i=j;\\
1 & \mbox{if } i \mbox{ is connected to } j \mbox{ by an edge;}\\
0 & \mbox{otherwise}.
\end{array}
\right.
$$ 
We will call $(\Z^n,Q_\Gamma)$ the intersection lattice associated to $\Gamma$. 

Notice that, since $M_{\Gamma}$ is a $2$-handlebody,  any matrix representing the intersection form of $M_{\Gamma}$ is also a presentation matrix for $H_{1}(\partial M_{\Gamma};\Z)$ (see e.g. Corollary 5.3.12 in \cite{b:GS}). In particular, $H_{1}(\partial M_{\Gamma};\Z)$ is finite if and only if $\det(Q_{\Gamma})\neq 0$, and in this case
\begin{equation}\label{e:dt}
\begin{split}
|\det(Q_\Gamma)|=|H_1(\partial M_\Gamma;\Z)|.
\end{split}
\end{equation}

The plumbing construction along a star-shaped plumbing graph $\Gamma$ yields a $4$-manifold $M_\Gamma$ whose boundary $Y_\Gamma:=\partial M_\Gamma$ is a \textbf{Seifert manifold} (see \cite{b:Ra} for a proof). Seifert manifolds are oriented, closed $3$-manifolds admitting a fixed point free action of $S^1$ and they are classified by their \virg{Seifert invariants} \cite{b:OR,b:Se}. The unnormalized Seifert invariant of the manifold $Y_\Gamma$ is the collection of numbers, $(b;(\alpha_1,\beta_1),...,(\alpha_r,\beta_r))$, where $b,\alpha_i,\beta_i\in\Z$, $\alpha_i\geq 1$ and $\mathrm{gcd}(\alpha_i,\beta_i)=1$ for all $i\in\{1,...,r\}$. This information can be read off from a star-shaped plumbing graph $\Gamma$ as follows. First of all, the number $r$ is precisely the number of legs of $\Gamma$. The number $b$ is the weight of the central vertex. Finally, if the weights on the $i$-th leg are $\{-a_{1,i},\cdots,-a_{k_i,i}\}$, then the irreducible fraction $\frac{\alpha_i}{\beta_i}$ is recovered from the continued fraction decomposition
$$ \frac{\alpha_i}{\beta_i}=[a_{1,i},...,a_{k,i}]:=a_{1,i}-\frac{1}{\displaystyle a_{2,i} - \frac{ \bigl. 1}{\displaystyle \ddots\ _{\displaystyle {a_{k_{i}-1,i}} -\frac {\bigl. 1}{a_{k_i,i}}}} }\ .$$
Notice that if $a_{j,i}\geq 2$ for every $j,i$, then we necessarily have $0< \beta_i< \alpha_i$. Among the properties of continued fractions we will need Riemenschneider's point rule \cite{b:Ri}, which we now briefly recall. Let $p>q>0$ be coprime integers, and suppose
$$\frac{p}{q}=[a_1,...a_\ell],\ a_i\geq 2,\s \frac{p}{p-q}=[b_1,...,b_k],\ b_j\geq 2.$$
Then, the coefficients $a_1,...,a_\ell$ and $b_1,...,b_k$ are related by a diagram of the form 
\begin{center}
\includegraphics[scale=0.8]{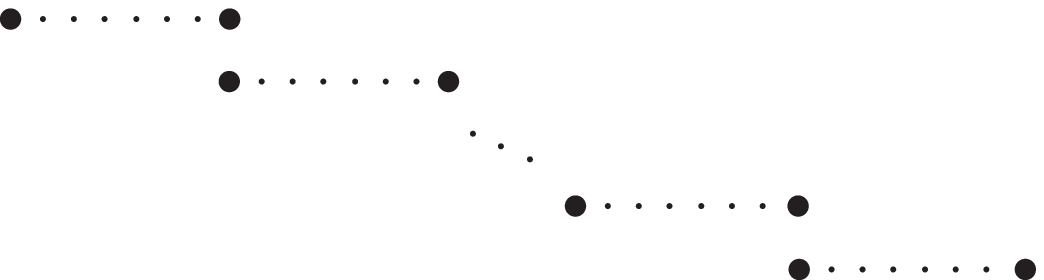}
\end{center}
where the $i$-th row contains $a_i-1$ \virg{points} for $i=1,...,l,$ and the first point of each row is vertically aligned with the last point of the previous row. The point rule says that there are $k$ columns, and the $j$-th column contains $b_j-1$ points for every $j=1,...,k$. For example for $\frac{17}{7}=[3,2,4]$ and $\frac{17}{10}=[2,4,2,2]$ the corresponding diagram is given by 
\begin{center}
\includegraphics{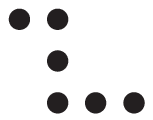}
\end{center}

\begin{rem}\label{r:riem}
Given two strings of integers $(a_1,...,a_n)$ and $(b_1,...,b_m)$, we consider the following operations:
\begin{itemize}
\item[$(1)$]$(a_1,...,a_n),(b_1,...,b_m)\longrightarrow (a_1,...,a_n+1),(b_1,...,b_m,2)$
\item[$(2)$]$(a_1,...,a_n),(b_1,...,b_m)\longrightarrow (a_1,...,a_n,2),(b_1,...,b_m+1)$
\end{itemize}
It is straightforward to check that if we start with $a_1=(2)$, $b_1=(2)$ the strings obtained using the above described operations are related to one another by Riemenschneider's point rule.
\end{rem}

Every Seifert manifold $Y_\Gamma(b;(\alpha_1,\beta_1),...,(\alpha_r,\beta_r))$ satisfies (see e.g.\ Lemma~4.2 in \cite{b:NR})
$$|H_1(Y_\Gamma;\Z)|=\abs{\alpha_1\cdots\alpha_r\cdot\cbra{\sum_{i=1}^r\frac{\beta_i}{\alpha_i}+b}}$$
and therefore, by \eqref{e:dt}, we have
\begin{equation}\label{e:det}
|\det(Q_\Gamma)|=\abs{\alpha_1\cdots\alpha_r\cdot\cbra{\sum_{i=1}^r\frac{\beta_i}{\alpha_i}+b}}.
\end{equation}

\begin{figure}[h!]
\begin{center}
\includegraphics[scale=0.7]{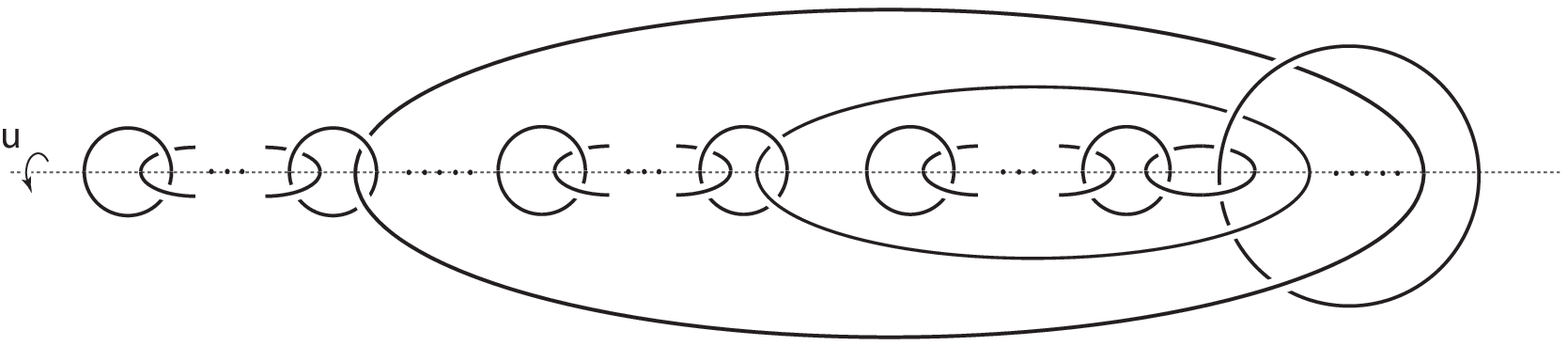}
\hcaption{Kirby diagram of a star-shaped plumbing graph (with arbitrary framings), as a strongly invertible link with respect to the involution $u$.}
\label{f:M_s.inv.}
\end{center}
\end{figure}

The Kirby diagram of a star-shaped plumbing graph $\Gamma$ consists of a link $L$ in $S^3$ which is \textbf{strongly invertible}, i.e.\ there exists a $\pi$-rotation on $S^3$ that induces on each connected component of $L$ an involution with two fixed points, see Figure~\ref{f:M_s.inv.}. Particularizing the statement to $4$-dimensional plumbings, we have the following result due to Montesinos~\cite[Theorem~3]{b:Mo}.

\begin{thm}[Montesinos]\label{t:Mo}
Consider the handle representation $M_\Gamma=H^0\cup nH^2$, where $n$ is the number of vertices in $\Gamma$. If the $n$ $2$-handles are attached along a strongly invertible link in $S^3$, then $M_\Gamma$ is a $2$-fold cyclic covering space of $D^4$ branched over a $2$-manifold. 
\end{thm}

The branching set in Theorem~\ref{t:Mo} is constructed as follows. Consider the strongly invertible link which represents the Kirby diagram of $M_\Gamma$, take the half of the Kirby diagram under the symmetry axis, substitute the half-circles with bands with as many half-twists as indicated by the framing of the corresponding circle and glue all these bands to a rectangle as shown in Figure~\ref{f:bandas}. This construction represents the Seifert space $Y_\Gamma=\partial M_\Gamma$ as a double cover of $S^3$ branched over the boundary $\Ml_\Gamma=\partial B_\Gamma$. The link $\Ml_\Gamma$ is by definition a \textbf{Montesinos link} (first defined by Montesinos in \cite{b:Mo2}).

\begin{figure}[h!]
\begin{center}
\frag[n]{s}{$=$}
\frag[n]{f}{$\cong$}
\frag{(a)}{\textbf{(a)}}
\frag{(b)}{\textbf{(b)}}
\frag{1}{$1$}
\frag{2}{$2$}
\frag{p}{$-2$}
\frag{3}{$3$}
\frag{1}{$-1$}
\frag{u}{$u$}
\includegraphics[scale=0.8]{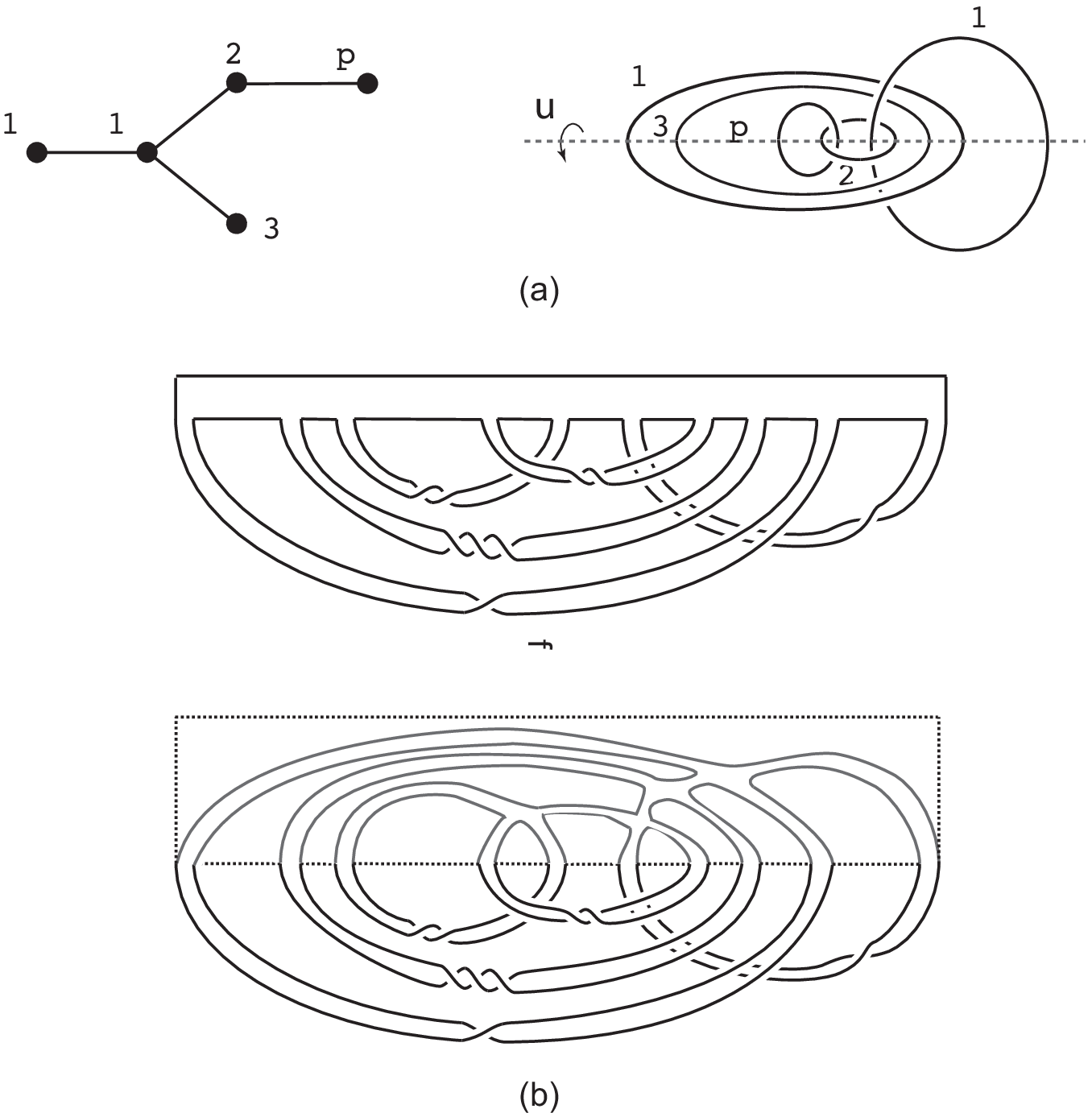}
\hcaption{Part \textbf{(a)} shows a star shaped plumbing graph and its associated Kirby diagram as a strongly invertible link. The bottom picture shows the branch surface in $D^{4}$ which is the image of Fix($u$). This surface is homeomorphic to the result of plumbing bands according to the initial graph: the gray lines retract onto the sides of the rectangle.}
\label{f:bandas}
\end{center}
\end{figure}

\begin{rem}
Note that, the surface $B_\Gamma$ is homeomorphic to the one obtained by plumbing twisted bands according to the plumbing graph $\Gamma$, which is how we introduced Montesinos links in Section~\ref{s:int}.
\end{rem}

The \textbf{$2$-bridge links} are the Montesinos links arising when we consider connected \textbf{linear} plumbings, that is when the plumbing graph $\Gamma$ has no distinguished central vertex because all vertices have valence $\leq 2$. In this case the associated Seifert spaces $Y_\Gamma$ are lens spaces. We will use the classical notation $L(p,q)$ for lens spaces and $K(p,q)$ for the corresponding $2$-bridge links. The numbers $p>q>0$ are coprime integers and in this case  the continued fraction expansion $\frac{p}{q}=[a_1,...,a_\ell]$ gives the string of integers $(a_1,...,a_\ell)$ which are the weights of the vertices of the linear graph with opposite signs.

\vspace{1.5mm}
\noindent\textbf{Definitions and notation concerning plumbing graphs. }
We  consider $\Z^n=\Z\oplus...\oplus\Z$ as an intersection lattice with respect to the product~$\cdot$ given in matrix form by $-\mathrm{Id}$, i.e.
$$ v\cdot w= - \langle v,w \rangle,\s\s \forall v,w\in\Z^n, $$
where $\langle\cdot,\cdot\rangle$ is the standard inner product of $\Z^n$. If we denote by $e_1,...,e_n$ the standard basis of $\Z^n$, we have 
$$e_i\cdot e_j=-\delta_{ij},\s\s \forall i,j=1,...,n. $$

We are interested in three-legged plumbing graphs $P$ whose associated intersection lattice admits an embedding into $\Z^n$, where $n=|P|$ is the number of vertices in the graph. The vertices of $P$, which from now on will be identified with their images in $\Z^n$ and will also be called vectors, are indexed by elements of the set $J:=\{(s,\alpha)|\, s\in\{0,1,...,n_\alpha\},\ \alpha\in\{1,2,3\}\}$. Here, $\alpha$ labels the legs of the graph and $L_\alpha:=\{v_{i,\alpha}\in P|\, i=1,2,...,n_\alpha\}$ is the set of vertices of the $\alpha$-leg. We will write $L_\alpha (P)$ and $n_\alpha (P)$ when we want to point out the graph $P$ to which these objects belong. The \textbf{string} associated to the leg $L_\alpha$ is the $n_\alpha$-tuple of integers $(a_{1,\alpha},...,a_{n_\alpha,\alpha})$, where $a_{i,\alpha}:=-v_{i,\alpha}\cdot v_{i,\alpha}\geq 0$. The three legs are connected to a common central vertex, which we denote indistinctly by $v_0=v_{0,1}=v_{0,2}=v_{0,3}$ (notice that, with our notation, $v_0$ does not belong to any leg). We say that the legs $L_\alpha, L_\beta\subseteq P\subseteq\Z^n$ are \textbf{complementary legs} when their associated strings are related to one another by Riemenschneider's point rule.

The main assumption on $P$ will be that the central vertex $v_0$ satisfies $v_0\cdot v_0\leq -3$, while each other vertex $v_{s,\alpha}$ satisfies $v_{s,\alpha}\cdot v_{s,\alpha}\leq -2$. Throughout the paper we will also deal with disconnected graphs with only one trivalent vertex. We will use the same notation for connected and disconnected graphs. 

Given $P\subseteq\Z^{n}$ we define the set of orthogonal matrices 
$$\Upsilon_{P}:=\{\Omega\in O(n;\Q)|\,\Omega v_{s,\alpha}\in\Z^{n}\ \mathrm{for\ every}\ v_{s,\alpha}\in P\}.$$ 
This set has no group structure, nevertheless since its elements are orthogonal matrices, for every $\Omega\in\Upsilon_{P}$ the intersection graph $\Omega P:=\{\Omega v_{s,\alpha}|\, v_{s,\alpha}\in P\}$ is the same as the intersection graph of $P$ and, by definition of $\Upsilon_P$, the set $\Omega P$ is again a subset of $\Z^n$. Notice that, for every $P$, the group $O(n;\Z)$ is a subset of $\Upsilon_{P}$ and it contains the reflections across each hyperplane orthogonal to an $e_i$, as well as all the transformations determined by the permutations of $\{e_1,...,e_n\}$. The introduction of the set $\Upsilon_P$ is due to the fact that we are interested in whether a plumbing graph admits an embedding, while the embedding itself is less relevant. Therefore, in the future we will usually identify $P$ with $\Omega P$ where $\Omega\in\Upsilon_P$.

The number of connected components of the plumbing graph $P$ will be denoted by $c(P)$ and we shall say that a vector $v_{s,\alpha}\in P$ is \textbf{isolated} [resp.\ \textbf{final}] if it is an isolated vertex [resp.\ a leaf] of the plumbing graph. A vertex that is neither isolated nor final  will be called \textbf{internal}. The only trivalent vertex in the graph is $v_0$. We will assume it is internal and call it the \textbf{central} vertex.

Two vectors $v,w\in\Z^n$ are \textbf{linked} if there exists $e\in\Z^n$ with $e\cdot e=-1$ such that $v\cdot e\neq 0$ and $w\cdot e\neq 0$. A set $P\subseteq\Z^n$ is \textbf{irreducible}\index{subset of $\Z^n$!irreducible} if, given two vectors $v,w\in P$, there exists a finite sequence $v=v_1,v_2,...,v_k=w$ of vectors of $P$ such that, for each $i=1,...,k-1$, $v_i$ and $v_{i+1}$ are linked. A set which is not irreducible is \textbf{reducible}. 


\vspace{1.5mm}
\noindent\textbf{Good and standard subsets of $\Z^n$. }
A  possibly disconnected plumbing graph $P\subseteq\Z^n$ is called \textbf{good}\index{subset of $\Z^n$!good} if it is irreducible and satisfies the following conditions.
\begin{itemize}
\item If it has one vertex of valence three, its  incidence matrix  has the form 
\frag{-aii}{\tiny$-a_{1\!,\!1}$}
\frag{gii}{\tiny$\gamma_{1\!,\!1}$}
\frag{gnni}{\tiny$\gamma_{n_{\!1}\!-\!1\!,\!1}$}
\frag{-ani}{\tiny$-a_{n{\!_1}\!,\!1}$}
\frag{-aim}{\tiny$-a_{1\!,\!2}$}
\frag{gim}{\tiny$\gamma_{1\!,\!2}$}
\frag{gnmi}{\tiny$\gamma_{n_{\!2}\!-\!1\!,\!2}$}
\frag{-anm}{\tiny$-a_{n{\!_2}\!,\!2}$}
\frag{-aip}{\tiny$-a_{1\!,\!3}$}
\frag{gip}{\tiny$\gamma_{1\!,\!3}$}
\frag{gnpi}{\tiny$\gamma_{n_{\!3}\!-\!1\!,\!3}$}
\frag{-anp}{\tiny$-a_{n{\!_3}\!,\!3}$}
\frag{1}{\tiny$1$}
\frag{App}{\tiny$-a_0$}
\begin{equation}\label{e:Q_P}
\begin{split}
Q_P=
\left(
\begin{array}{l|l|l|l}
\includegraphics{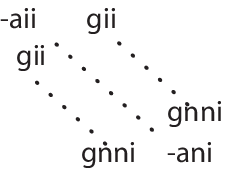}& & & \includegraphics{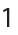} \\ \hline
& \includegraphics{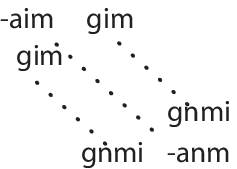} & & \includegraphics{1ver.eps}\\ \hline
& & \includegraphics{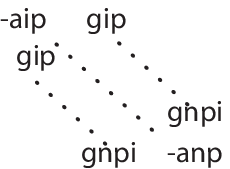} & \includegraphics{1ver.eps} \\ \hline
\includegraphics{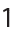} & \includegraphics{1hor.eps} & \includegraphics{1hor.eps} & \includegraphics{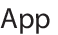}
\end{array}
\right)
\end{split}
\end{equation}
where $\gamma_{s,\alpha}\in\{1,0\}$, $a_{s,\alpha}=-v_{s,\alpha}\cdot v_{s,\alpha}\geq 2$ and $a_0=-v_0\cdot v_0\geq 3$.

\item if all its vertices have at most valence two (and in this case we call $P$ a \textbf{linear set}), its  incidence matrix  has the form
\frag[n]{-aii}{$-a_{1\!,\!1}$}
\frag[n]{gii}{$\gamma_{1\!,\!1}$}
\frag[n]{gnni}{$\gamma_{n\!-\!1\!,\!1}$}
\frag[n]{-ani}{$-a_{n\!,\!1}$} 
\begin{equation}\label{e:Q_P_linear}
\begin{split}
Q_P= \cbra{
\begin{array}{c}
\includegraphics[scale=1.5]{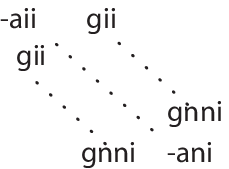}
\end{array}
}
\end{split}
\end{equation}
where again $\gamma_{s,\alpha}\in\{1,0\}$ and $a_{s,\alpha}=-v_{s,\alpha}\cdot v_{s,\alpha}\geq 2$.
\end{itemize}
Furthermore, if $P$ is a connected graph (i.e.\ all the $\gamma_{s,\alpha}$'s in its incidence matrix are equal to $1$) we will say that $P$ is \textbf{standard}. Standard and good linear sets were studied by Lisca in \cite{b:Li}.

\vspace{1.5mm}
\noindent\textbf{The quantity $I(\cdot)$ and the family $\wp$. }
Given a subset $P\subseteq\Z^n$, the key quantity in our discussion will be the number \index{$I(P)$}
$$I(P):=\sum_{(s,\alpha)\in J}(-v_{s,\alpha}\cdot v_{s,\alpha}-3).$$

Consider the family $\wp$ containing all three-legged connected plumbing graphs such that:
\begin{itemize}
\item The central vertex has weight less or equal to $-3$ and every non central vertex has weight less or equal to $-2$.
\item $I(\Gamma)<-1$ for every $\Gamma\in\wp$. 
\end{itemize} 

\begin{rem}\label{r:neg}
Notice that, for $\Gamma\in\wp$ the matrix $Q_\Gamma$ is negative definite, since $\Gamma$ is a canonical negative plumbing graph in the sense of \cite[Theorem~5.2]{b:NR}.
\end{rem}

As explained in Section~\ref{s:int} our first aim is to determine all the graphs $P\in\wp$ whose associated intersection lattice admits an embedding in the standard negative diagonal intersection lattice. In the terminology that we have just introduced our aim is to determine all possible standard subsets $P\subseteq\Z^n$ such that $I(P)<-1$. For the linear case, the answer is known and can be found in Remark~\ref{r:lisca}. For graphs with a trivalent vertex the complete list is given in Theorems~\ref{l:strings1} and \ref{l:2} below. The proof of these theorems is developed in Sections~\ref{s:contractions} to \ref{s:ss}.

\begin{thm}\label{l:strings1} 
Let $P\subseteq\Z^n$ be a standard subset with a trivalent vertex $v_0$, two complementary legs $L_2$ and $L_3$ and $I(P)<-1$. Then $I(P)\in\{-2,-3,-4\}$ and the numbers $\{a_0,a_{1,1},...,a_{n_3,3}\}$ satisfy:
\begin{itemize}
\item The strings associated to the complementary legs, namely $(a_{1,2},...,a_{n_2,2})$ and $(a_{1,3},...,a_{n_3,3}) $, are related to each other by Riemenschneider's point rule.
\item The linear set $L_1\cup\{v_0\}$ has an associated string $(a_{n_1,1},...,a_{1,1},a_0)$ that is obtained from the string associated to a linear standard set with $I\in\{-3,-2,-1\}$ by adding $1$ to the final vertex that plays the role of central vertex in $P$. The complete list of the possible $(a_{n_1,1},...,a_{1,1},a_0)$ is the following.\\
If $I(P)=-4$
\begin{itemize}
\item[]$(b_k,b_{k-1},...,b_1,2,c_1,...,c_{l-1},c_\ell+1),\s \hfill \forall k,\ell\geq 1.\s\s\s\s$
\end{itemize}
If $I(P)=-3$
\begin{itemize}
\item[] $(2^{[t]},3,2+s,2+t,3,2^{[s-1]},3),\s \hfill  \forall  s\geq 1,t\geq 0,\s\s\s\s$
\item[] $(2^{[t]},3,2,2+t,4),\s \hfill  \forall  t\geq 0,\s\s\s\s$
\item[]$(2^{[t]},3+s,2,2+t,3,2^{[s-1]},3),\s \hfill \forall  s\geq 1,t\geq 0,\s\s\s\s$
\item[]$(2^{[s]},3,2+t,2,3+s,2^{[t-1]},3),\s \hfill \forall  t\geq 1,s\geq 0,\s\s\s\s$
\item[]$(2^{[s]},3,2,2,4+s),\s \hfill \forall  s\geq 0,\s\s\s\s$
\item[]$(b_k,b_{k-1},...,b_1+1,2,2,1+c_1,...,c_{l-1},c_\ell+1),\s \hfill \forall  k,\ell\geq 1,\s\s\s\s$
\end{itemize}
where $b_1,...,b_k\geq 2$ are arbitrary integers, and $c_1,...,c_\ell$ are obtained from $b_1,...,b_k$ using Riemenschneider's point rule.\\
If $I(P)=-2$
\begin{itemize}
\item[]$(t+2,s+2,3,2^{[t]},4,2^{[s-1]},3),\s \hfill \forall  s\geq 1,t\geq 0,\s\s\s\s$
\item[]$(t+2,2,3,2^{[t]},5),\s \hfill  \forall  t\geq 0,\s\s\s\s$
\item[]$(2^{[s]},4,2^{[t]},3,s+2,t+3),\s \hfill  \forall  s,t\geq 0,\s\s\s\s$
\item[]$(t+2,2,3+s,2^{[t]},4,2^{[s-1]},3),\s \hfill  \forall  s\geq 1,t\geq 0,\s\s\s\s$
\item[]$(2^{[s]},4,2^{[t]},3+s,2,t+3),\s \hfill \forall  s,t\geq 0,\s\s\s\s$
\item[]$(t+3,2,3+s,3,2^{[t]},3,2^{[s-1]},3),\s\hfill  \forall  s\geq 1,t\geq 0,\s\s\s\s$
\item[]$(t+3,2,3,3,2^{[t]},4),\s \hfill \forall  t\geq 0,\s\s\s\s$
\item[]$(2^{[s]},3,2^{[t]},3,3+s,2,t+4),\s \hfill \forall  t,s\geq 0.\s\s\s\s$
\end{itemize}
\end{itemize}
\end{thm}

\begin{thm}\label{l:2}
Let $P_n\subseteq\Z^n$ be a standard subset with a trivalent vertex, no complementary legs and such that $I(P_n)<-1$. Then $I(P_n)=-2$ and the graph $P_n$ belongs to the list in Figure~\ref{f:lista}.
\end{thm}

\begin{figure}[h]
\begin{center}
\frag{t}{$t$}
\frag{s}{$s$}
\frag{r}{$r$}
\frag{1}{$-1$}
\frag{t>=0,s>0}{$t\geq 0,s>0$}
\frag{t>0,s>0}{$t>0,s>0$}
\frag{t>0}{$t>0$}
\frag{t>0,s>=0}{$s>0,t\geq 0$}
\frag{t>=0,r>0}{$r\geq 0,t,s>0$}
\frag{s-1}{$s-1$}
\frag{2}{$-2$}
\frag{3}{$-3$}
\frag{q}{$-t$}
\frag{p}{$-s$}
\frag{2-t}{$-\!2\!-\!t$}
\frag{2-r}{$-\!2\!-\!r$}
\frag{s-2}{$-\!2\!-\!s$}
\frag{t-3}{$-\!3\!-\!t$}
\frag{-b_1}{$-b_1$}
\frag{-b_K}{$-b_k$}
\frag{-c_1+1}{$-c_1\!-\!1$}
\frag{-c_1+2}{$-c_1\!-\!2$}
\frag{-c_l}{$-c_\ell$}
\frag{a}{$(a)$}
\frag{b}{$(b)$}
\frag{c}{$(c)$}
\frag{d}{$(d)$}
\frag{e}{$(e)$}
\includegraphics[scale=0.7]{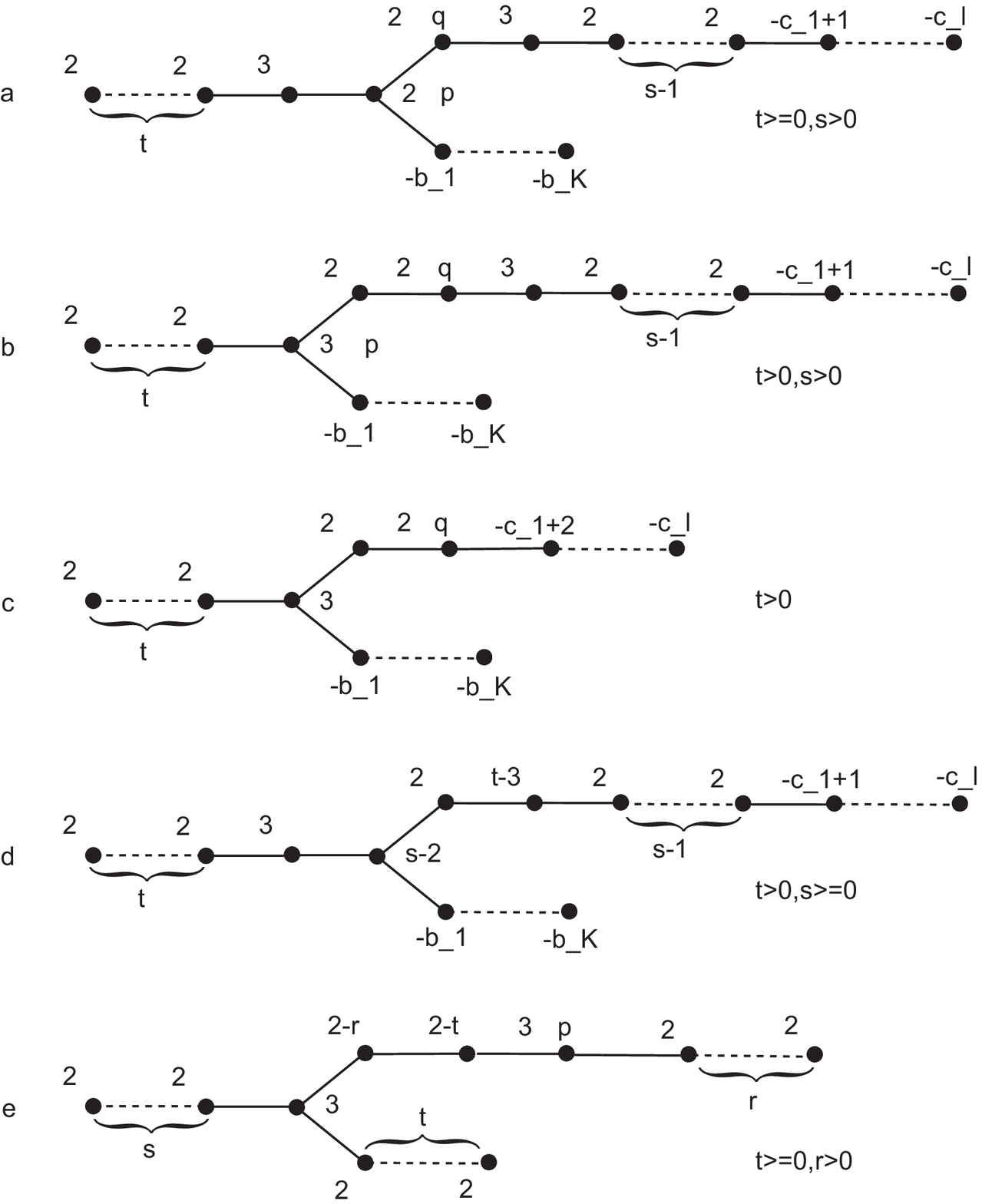}
\hcaption{Families $(a)-(e)$ are all possible graphs of standard subsets of $\Z^n$ with a trivalent vertex and no complementary legs. The integers $b_1,...,b_k\geq 2$ are arbitrary, while $c_1,...,c_\ell$ are obtained from $b_1,...,b_k$ using Riemenschneider's point rule.}
\label{f:lista}
\end{center}
\end{figure}

The proof of Theorems~\ref{l:strings1} and \ref{l:2}, which is carried out in Sections~\ref{s:contractions} to \ref{s:ss}, can be briefly sketched as follows. We start by defining an operation, the contraction, which given a good set $P\subseteq\Z^n$ returns a subset of $\Z^{n-1}$. Then we show that, under certain assumptions, the result of a contraction is again a good set $P'$ and moreover, $I(P')\leq I(P)$. Afterwards, we observe that complementary legs are preserved under contractions and that the contraction of a good set without complementary legs is again a set without complementary legs. Iterating contractions we prove that every good set with complementary legs can be contracted to one of the three good subsets of $\Z^5$ which have $I<-1$. In turn, every good set without complementary legs can be contracted to one of the two good subsets of $\Z^3$ with $I<-1$. Keeping track of the quantity $I$ and of the number of connected components of the sets involved in the sequence of contractions leads to the complete classification of standard subsets.

\section{Existence of Ribbon surfaces}\label{c:bande}
Theorems~\ref{l:strings1} and \ref{l:2} give the complete list of plumbing graphs $P\in\wp$ whose associated intersection lattice admits an embedding in the standard negative diagonal intersection lattice. As explained in the introduction, this is a necessary condition for the corresponding Montesinos \emph{knots}, $\Ml_P$, to bound a slicing disc. In this section we shall find, for each $\Ml_P$ with $P$ as in Theorem~\ref{l:strings1} or Theorem~\ref{l:2}, a surface with boundary $\Sigma$ such that $\chi (\Sigma)=1$ and a ribbon immersion $i:\Sigma\looparrowright S^3$ with $i(\partial\Sigma)=\Ml_P$. In this way we obtain that the slice-ribbon conjecture is true for all Montesinos \emph{knots} $\Ml_P$ with $P\in\wp$. This section concludes with the proof of the main result of this work, Theorem~\ref{t:int}.  

\subsection{Montesinos links with complementary legs}\label{s:Montesinos_compl}

In this section we deal with the Montesinos links associated to the graphs $P$ in Theorem~\ref{l:strings1}. We start by showing that for each $P$ there exists a cobordism between the Seifert space $Y_P$ and $L(p,q)\# (S^1\times S^2)$, where $L(p,q)$ is a lens space bounding a rational homology ball. A theorem by Montesinos, Theorem~\ref{t:Mo}, shows that this cobordism corresponds to a ribbon move (see \cite[p. 211]{b:GS} for the definition) on the link $\Ml_P$. 

Let $\Gamma$ be a connected three-legged plumbing graph with two complementary legs $L_2$ and $L_3$ with associated strings $(b_{1},...,b_{k})$ and $(c_{1},...,c_{l})$ respectively. Let $(a_{1,1},...,a_{n_1,1})$ be the string associated to the leg $L_{1}$ and $-a_{0}$ the weight of the central vertex. We recall that we write $M_\Gamma$ for the associated oriented $4$-manifold and $Y_\Gamma:=\partial M_\Gamma$ for its oriented three dimensional boundary. For these graphs $\Gamma$ we have the following result.

\begin{lem}\label{p:clb}
Attaching to $M_{\Gamma}$ a $4$-dimensional $2$-handle along the framed thick circle in the first diagram of Figure~\ref{f:cl}, we obtain a $4$-manifold whose boundary is $(S^{1}\times S^{2})\# Y_S$, where $S$ is a linear plumbing graph with associated string $(a_{n_1,1},a_{n_1-1,1},...,a_{1,1},a_0-1)$.
\end{lem}
\begin{proof}
The first link of Figure~\ref{f:cl} is the Kirby diagram of the graph $\Gamma$ with an added thick circle with framing $-1$. Blowing down this $(-1)$-circle we obtain a split link which consists of two linear chains. Since $L_{2}$ and $L_{3}$ are complementary legs, the strings $(b_{1},...,b_{k})$ and $(c_{1},...,c_{l})$ are related to one another by Riemenschneider's point rule. Notice that then, whenever $l+k>2$, we necessarily have either $b_{1}=2$ and $c_{1}>2$ or $b_{1}>2$ and $c_{1}=2$. By symmetry, let us suppose that the first case occurs. It is immediate to check, using the Riemenschneider's point diagram, that $(b_{2},...,b_{k})$ and $(c_{1}-1,...,c_{l})$ are again two strings related to one another by Riemenschneider's point rule. Moreover, if $l+k-1>2$ then either $b_{2}=2$ and $c_{1}-1>2$ or $b_{2}>2$ and $c_{1}-1=2$. Therefore,  in the third diagram of Figure~\ref{f:cl} we have $-b_{1}+1=-1$ and blowing down this circle produces a new diagram with a new $(-1)$-circle linked to the first circles of two complementary chains. After $k+l$ blow downs (starting with the first diagram), we arrive to a diagram with two components: an unknot with framing $0$ and the leg $L_{1}$ linked to a circle with framing $-a_{0}+1$. This diagram represents a $4$-manifold whose boundary is the connected sum of $S^2\times S^1$ and the lens space $L(p,q)$ where $\frac{p}{q}=[a_{n_1,1},a_{n_1-1,1},...,a_{1,1},a_0-1]$.
\end{proof}

\begin{figure}
\begin{center}
\frag[ss]{a0+1}{$-\!a_0\!+\!1$}
\frag[ss]{a0}{$-a_0$}
\frag[ss]{a1}{$-\!a_{1\!,\!1}$}
\frag[ss]{a2}{$-\!a_{2\!,\!1}$}
\frag[ss]{an}{$-\!a_{n_{1}\!,\!1}$}
\frag[ss]{b2}{$-b_{2}$}
\frag[ss]{bk}{$-b_{k}$}
\frag[ss]{c1}{$-c_{1}$}
\frag[ss]{c2}{$-c_{2}$}
\frag[ss]{1}{$-1$}
\frag[ss]{2}{$-2$}
\frag[ss]{cl}{$-c_{l}$}
\frag[ss]{isotopy}{isotopy}
\frag[ss]{b1+1}{$-b_{1}\!\!+\!\!1$}
\frag[ss]{b2+1}{$-b_{2}\!\!+\!\!1$}
\frag[ss]{c1+1}{$-c_{1}\!\!+\!\!1$}
\frag[ss]{c1+2}{$-c_{1}\!\!+\!\!2$}
\frag[ss]{0}{$0$}
\frag[ss]{u}{$u$}
\frag[ss]{7}{$\displaystyle\boundary^\partial$}
\frag[ss]{b1}{$-b_{1}$}
\frag[ss]{blabla}{$k+l-3$ blow downs}
\includegraphics[scale=0.7]{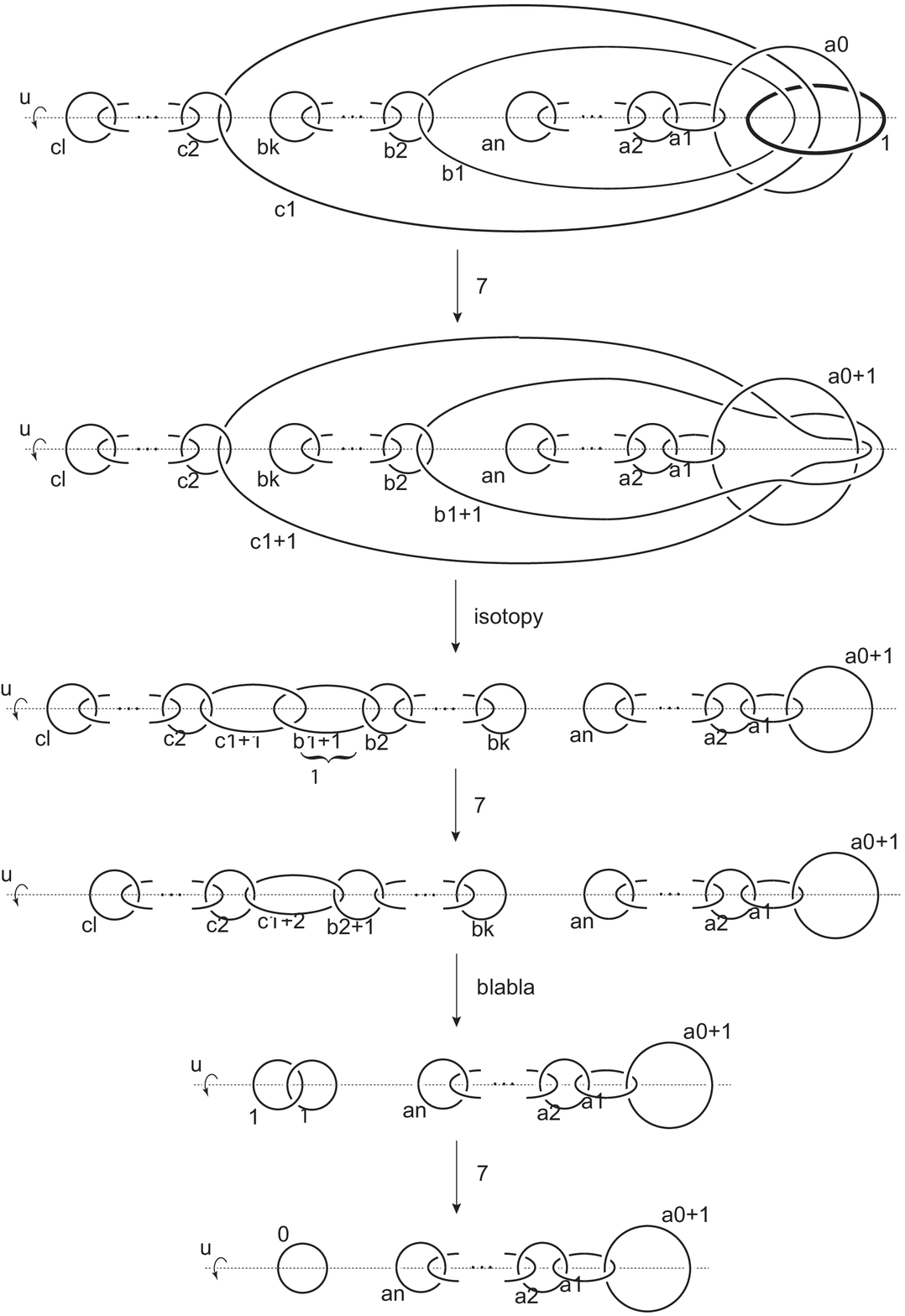}
\hcaption{Kirby diagram of Lemma~\ref{p:clb}. The arrows denote isotopies or a blow down operation, which preserves the boundary of the $4$-manifold. Notice that all these operations are done equivariantly with respect to the involution $u$.}
\label{f:cl}
\end{center}
\end{figure}

\begin{figure}[h!]
\begin{center}
\frag[ss]{a}{$-\!a_{n_{1}\!,\!1}$}
\frag[ss]{b}{$-\!a_{0}\!+\!1$}
\includegraphics[scale=0.85]{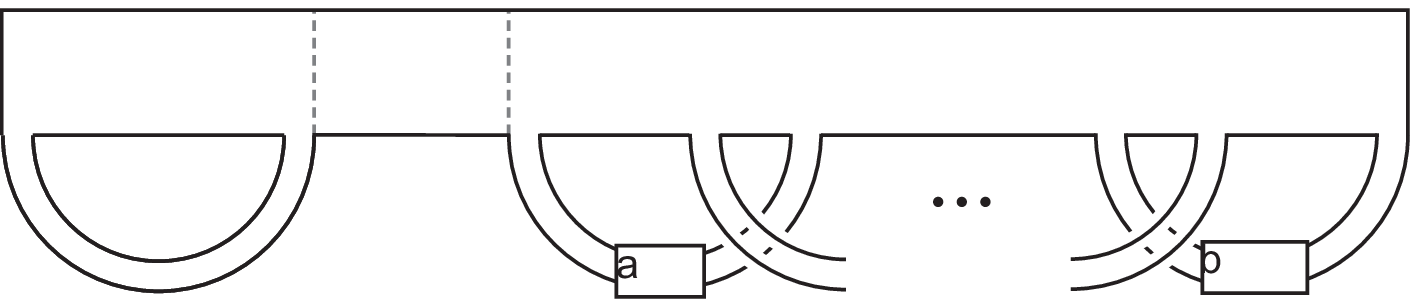}
\hcaption{Branch surface of the involution $u$ on the $4$-manifold defined by the last Kirby diagram in Figure~\ref{f:cl}. The discontinuous lines show where is performed a connected sum between the disjoint union of two unknots and the $2$-bridge link $K(p,q)$, where $\frac{p}{q}=[a_{n_{1},1},...,a_{1,1},a_{0}-1]$, yielding $U\cup K(p,q)$.}
\label{f:lR}
\end{center}
\end{figure}

Observe that the $2$-handle $h$, represented by the thick circle in the top link in Figure~\ref{f:cl}, is added to the Kirby diagram of a three-legged star-shaped plumbing graph $P$, in such a way that we obtain a strongly invertible link in $S^3$ (with respect to $u$). Before adding $h$, the involution $u$ turns $M_P$ into the double cover of $D^4$ branched over a surface $B_P$, which is obtained by plumbing bands according to $P$ and satisfies $\partial B_{P}=\Ml_{P}$. Let us call $M'$ the $4$-manifold in the first Kirby diagram of Figure~\ref{f:cl}, obtained after adding $h$ to $M_P$. By Montesinos' theorem, Theorem~\ref{t:Mo}, the involution $u$ turns $M'$ into the double cover of $D^4$ branched over a surface $B'$, which is obtained by adding a band with a half-twist to the surface $B_P$. The addition of this band is a ribbon move on $\Ml_P$. In order to understand the link obtained after this ribbon move we argue as follows. Given a tubular neighborhood of an unknotted component $C$ of a link in $S^3$, its complement is a solid torus $T$. As explained in \cite[Chap.\ 9 Sect.\ H]{b:Ro}, a blow down operation done along $C$ consists of applying a meridinal twist to $T$ and this alters the rest of the link as shown in Figure~\ref{f:cl}. The important point here is that this operation as well as the isotopies in Figure~\ref{f:cl} are done equivariantly with respect to $u$. Therefore, the surface $B'$ can be thought of as obtained by plumbing twisted bands according to the last diagram in Figure~\ref{f:cl} (instead of according to the first one). The result is ilustrated in Figure~\ref{f:lR}. Since the ribbon move turns $\Ml_P$ into the boundary of $B'$, we conclude that after this move we obtain a split link of the form $U\cup K(p,q)$, where $U$ is the unknot and $K(p,q)$ is the $2$-bridge link given by the fraction $\frac{p}{q}=[a_{n_{1},1},...,a_{1,1},a_{0}-1]$.   

Let $P$ be a plumbing graph as in Theorem~\ref{l:strings1}. Then, a case-by-case check shows that the graph $S\subseteq\Z^{n_1+1}$ in Lemma~\ref{p:clb} is one of the linear graphs listed in Remark~\ref{r:lisca}. 

\begin{rem}\label{r:lisca}
For the reader's convenience, we summarize Lem\-mas~7.1, 7.2 and 7.3 of \cite{b:Li}, which give all linear standard sets $P$ with $I(P)\in\{-3,-2,-1\}$. Lisca's results are expressed in terms of strings of integers $(a_{1,1},...,a_{n_1,1})$ and we will write, for any integer $t\geq 0$
$$(...,2^{[t]},...):=(...,\overbrace{2,...,2}^t,...).$$
Consider a linear standard set $P=\{v_{1,1},...,v_{n_1,1}\}\subseteq\Z^{n_1}$ and write, as usual, $v_{s,1}\cdot v_{s,1}=-a_{s,1}$. Then, 
\begin{enumerate}[(I)]
\item \label{p:-3} If $I(P)=-3$ the string $(a_{1,1},...,a_{n_1,1})$ is obtained from $(2,2,2)$ via a finite sequence of operations of the following types:
\begin{itemize}
\item[$(1)$]$(b_1,...,b_m)\longrightarrow (b_1+1,...,b_m,2)$
\item[$(2)$]$(b_1,...,b_m)\longrightarrow (2,b_1,...,b_m+1)$
\end{itemize}
An alternative description of the string $(a_{1,1},...,a_{n_1,1})$, which is obtained by a straightforward calculation, is given by the string
$$(b_k,b_{k-1},...,b_1,2,c_1,...,c_{l-1},c_\ell),\s\s k,l\geq 1,$$
where the $k$-tuple of integers $b_1,...,b_k\geq 2$ is arbitrary and the numbers $c_1,...,c_\ell\geq 2$ are obtained from $b_1,...,b_k$ using Riemenschneider's point rule \cite{b:Ri} (see also Remark~\ref{r:riem}).
\item\label{p:-2} If $I(P)=-2$ either $(a_{1,1},...,a_{n_1,1})$ or $(a_{n_1,1},...,a_{1,1})$ has one of the following forms\footnote{In \cite[Lemma~7.2]{b:Li} the family $(3)$ is missing, due to overlooking in the statement.}
	\begin{itemize}
	\item[$(1)$]$(2^{[t]},3,2+s,2+t,3,2^{[s]}),\s s,t\geq 0,$
	\item[$(2)$]$(2^{[t]},3+s,2,2+t,3,2^{[s]}),\s s,t\geq 0,$
	\item[$(3)$]$(b_k,b_{k-1},...,b_1+1,2,2,1+c_1,...,c_{l-1},c_\ell)$, for arbitrary integers $b_1,...,b_k\!\geq 2$ and 		for $c_1,...,c_\ell$ obtained from $b_1,...,b_k$ using Riemenschneider's point rule. 
	\end{itemize}
\item\label{p:-1} If $I(P)=-1$ either $(a_{1,1},...,a_{n_1,1})$ or $(a_{n_1,1},...,a_{1,1})$ has one of the following forms:
	\begin{itemize}
	\item[$(1)$]$(t+2,s+2,3,2^{[t]},4,2^{[s]}),\s s,t\geq 0,$
	\item[$(2)$]$(t+2,2,3+s,2^{[t]},4,2^{[s]}),\s s,t\geq 0,$
	\item[$(3)$]$(t+3,2,3+s,3,2^{[t]},3,2^{[s]}),\s s,t\geq 0.$
	\end{itemize}
\end{enumerate}
\end{rem}

The $3$-manifold $Y_S$ is a lens space $L(p,q)$, which determines a $2$-bridge link $K(p,q)$. In \cite[Section~8]{b:Li} we find, for every graph in Remark~\ref{r:lisca}, a ribbon surface $\Sigma_1$ in $S^{3}$  with boundary $\partial\Sigma_{1}=K(p,q)$, homeomorphic to a disc if $K(p,q)$ is a knot, and to the disjoint union of a disc and a M\"obius band if $K(p,q)$ is a $2$-component link. 

Summarizing the discussion of this section we have the following.
\begin{lem}\label{l:sigma_cl}
For a plumbing graph $P$ as in Theorem~\ref{l:strings1}, depending on the number of connected components of the links $\Ml_P$ and $K(p,q)$, we have the following possibilities, where the ribbon move is the attachment of the band corresponding to the $2$-handle added to $M_P$ in Lemma~\ref{p:clb}.
\begin{itemize}
\item If $\Ml_P$ is a knot, there exists a ribbon move that reduces it to a $2$-bridge ribbon knot and an unknot. It follows that $\Ml_P$ is ribbon.
\item If $\Ml_P$ is a $2$-component link, there exists a ribbon move that reduces it either to a $2$-bridge ribbon knot and an unknot or to a $2$-bridge ribbon link and an unknot. In any case it follows that $\Ml_P=\partial\Sigma$, where $\Sigma$ is a ribbon surface in $S^3$, which is the disjoint union of a disc and a M\"obius band.
\item If $\Ml_P$ is a $3$-component link, there exists a ribbon move that reduces it either to a $2$-bridge ribbon knot and an unknot or to a $2$-bridge ribbon link and an unknot. In both cases it follows that $\Ml_P=\partial\Sigma$, where $\Sigma$ is a ribbon surface in $S^3$. In the first case $\Sigma$ is the disjoint union of a disc and an annulus and in the second case $\Sigma$ is the disjoint union of either  two M\"obius bands and a disc or two discs and $K\setminus D^2$, where $K$ stands for the Klein bottle.
\end{itemize}
\end{lem}
\begin{proof}
We write the details for the second case, and we leave the remaining two as a straightforward exercise for the reader. The proof is also sketched in Figure~\ref{f:cob}. 

\begin{figure}
\begin{center}
\frag{1}{$1$}
\frag{m1}{$-1$}
\frag{0}{$0$}
\includegraphics[scale=0.7]{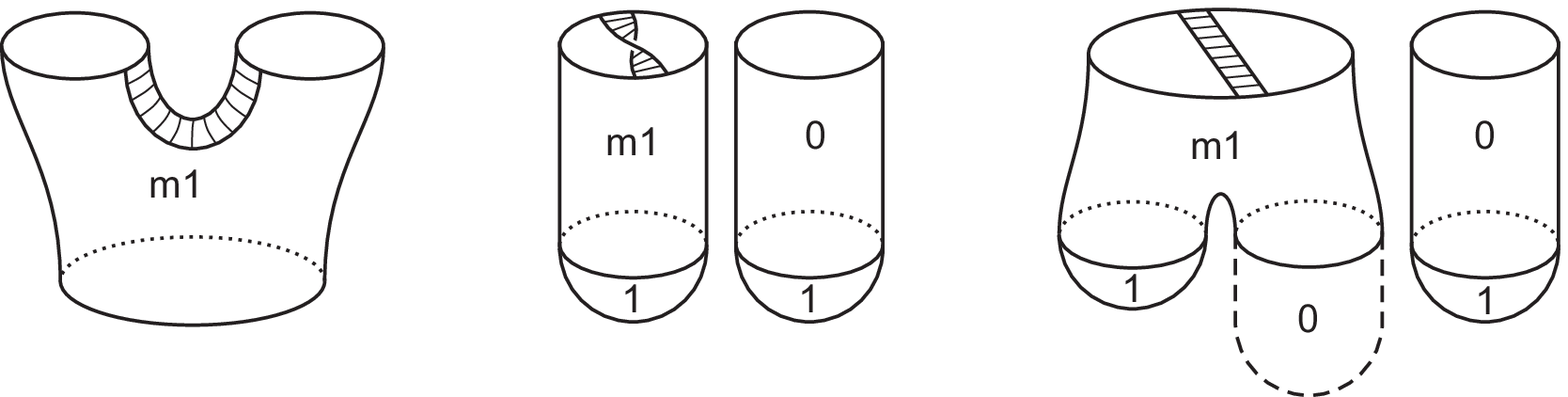}
\hcaption{These three figures represent schematically the possible ribbon moves (the addition of the striped bands) on the $2$-component link $\Ml_{P}$. The numbers are the contribution to the Euler characteristic of the surface between levels of the cobordism. The dashed line in the third figure denotes a M\"obius band.}
\label{f:cob}
\end{center}
\end{figure}

If $\Ml_P$ is a $2$-component link, there are $3$ possibilities for a ribbon move on it: 

\noindent $(i)\ $  If the band joins the two components, it yields a knot; 

\noindent $(ii)$ If the band joins two points of the same component without splitting it up, it yields a $2$-component link; 

\noindent $(iii)$ Finally, if the band joins two points of the same component and splits it up, it yields a $3$-component link. 

By Lemma~\ref{p:clb}, the result of this ribbon move is a split link of the form $U\cup K(p,q)$, where $U$ is the unknot and $K(p,q)$ is a $2$-bridge link. It follows that the first possibility for the ribbon move is excluded. The second possibility describes a cobordism, which is a surface with Euler characteristic $-1$, from $\Ml_P$ to $U\cup K(p,q)$, where $K(p,q)$ is a ribbon knot. Since $U\cup K(p,q)$ is the boundary of two disjoint discs, it follows that $\Ml_P$ is the boundary of the disjoint union of a M\"obius band and a disc. The third possibility describes a cobordism, which is again a surface with Euler characteristic $-1$, from $\Ml_P$ to $U\cup K(p,q)$, where this time $K(p,q)$ is a link, boundary of a ribbon surface consisting of the disjoint union of a M\"obius band and a disc. It follows that $U\cup K(p,q)$ is the boundary of two disjoint discs and a M\"obius band. Therefore, $\Ml_P$ is the boundary of the disjoint union of a M\"obius band and a disc.
\end{proof} 

\subsection{Montesinos links without complementary legs}\label{s:NoCl}
In this section we follow the same strategy as in the preceding one. We start by showing that adding one handle to the Kirby diagrams representing the $4$-manifolds $M_P$ with $P$ as in Theorem~\ref{l:2}, we obtain a $4$-manifold with boundary $S^1\times S^2$. We show that this gives a presentation of $S^1\times S^2$ via a strongly invertible framed link, so Montesinos' theorem, Theorem~\ref{t:Mo}, guarantees that there exists a surface with boundary $\Sigma$,  with $\chi (\Sigma)=1$ and a ribbon immersion $i:\Sigma\looparrowright S^3$ with $i(\partial\Sigma)=\Ml_P$. 

\begin{lem}\label{l:s1xs2}
All the $4$-manifolds represented by the plumbing graphs in Theorem~\ref{l:2} can be changed into a $4$-manifold with boundary $S^1\times S^2$ by adding a $2$-handle along a circle with framing $-1$. 
\end{lem}
\begin{proof}
The families $(a)$ with $t>0$, $(b),(c)$ and $(e)$ of Theorem~\ref{l:2}, are represented schematically in Figure~\ref{f:4Fam}, where a black square on a circle represents a possible linear plumbing linked to it. The thicker circle represents the added $2$-handle. In this way, we can perform Kirby calculus on this general figure and then specialize it to the different families, by substituting $x,y,z$ and the black squares with the corresponding framings and linear plumbings respectively. This is done in Figure~\ref{f:4Fam,2}, where $(i)$ [resp.\ $(ii)$ and $(iii)$] represents the last diagram in Figure~\ref{f:4Fam} with the data from family $(a)$ with $t>0$ or from family $(b)$ [resp.\ $(c)$ and $(e)$]. Since the last diagram in Figure~\ref{f:4Fam} is star-shaped, in Figure~\ref{f:4Fam,2} we have used the graph notation to improve clarity. In each family the graph has a vertex with weight $-1$. Blowing down this $-1$ we obtain a new graph with a new vertex with weight $-1$. For every family this blowing down operation can be iterated (the first case in Figure~\ref{f:4Fam,2} is done in full detail; the other two are approached analogously) until we are left with a graph with only one vertex of weight $0$, which represents the $4$-manifold $D^{2}\times S^{2}$ with boundary $S^{1}\times S^{2}$. Since the blow downs do not change the boundary, the statement is proved for these families.

We are left with the families $(a)$ when $t=0$ and $(d)$, which are represented schematically in Figure~\ref{f:2Fam}. The thicker circle represents, as before, the added $2$-handle. Family $(a)$ with $t=0$ has $x=-3$ and $\blacksquare=\emptyset$, while family $(d)$ has $x=-t-3$ and $\blacksquare$ represents a linear plumbing of $t$ circles, each of them with framing $-2$. Figure~\ref{f:2Fam} shows that the addition of the thick handle with framing $-1$ turns the original $4$-manifold into another one with boundary $S^{1}\times S^{2}$. We have done in detail the case $(a)$ when $t=0$, the study of family $(d)$ being analogous.    
\end{proof} 

\begin{figure}
\begin{center}
\frag[ss]{7}{$\displaystyle\boundary^\partial$}
\frag[ss]{t blow downs}{$t$ blow downs}
\frag[ss]{isot}{isotopy}
\frag[ss]{1}{$-1$}
\frag[ss]{t-1}{$t\!-\!1$}
\frag[ss]{-t-2}{$-\!t\!\!-\!\!2$}
\frag[ss]{-t-1}{$-\!t\!\!-\!\!1$}
\frag[ss]{2}{$-2$}
\frag[ss]{x}{$x$}
\frag[ss]{y}{$y$}
\frag[ss]{z}{$z$}
\frag[ss]{x-1}{$x+1$}
\frag[ss]{x-2}{$x+2$}
\frag[ss]{y-1}{$y+1$}
\frag[ss]{z-1}{$z+1$}
\includegraphics[scale=0.8]{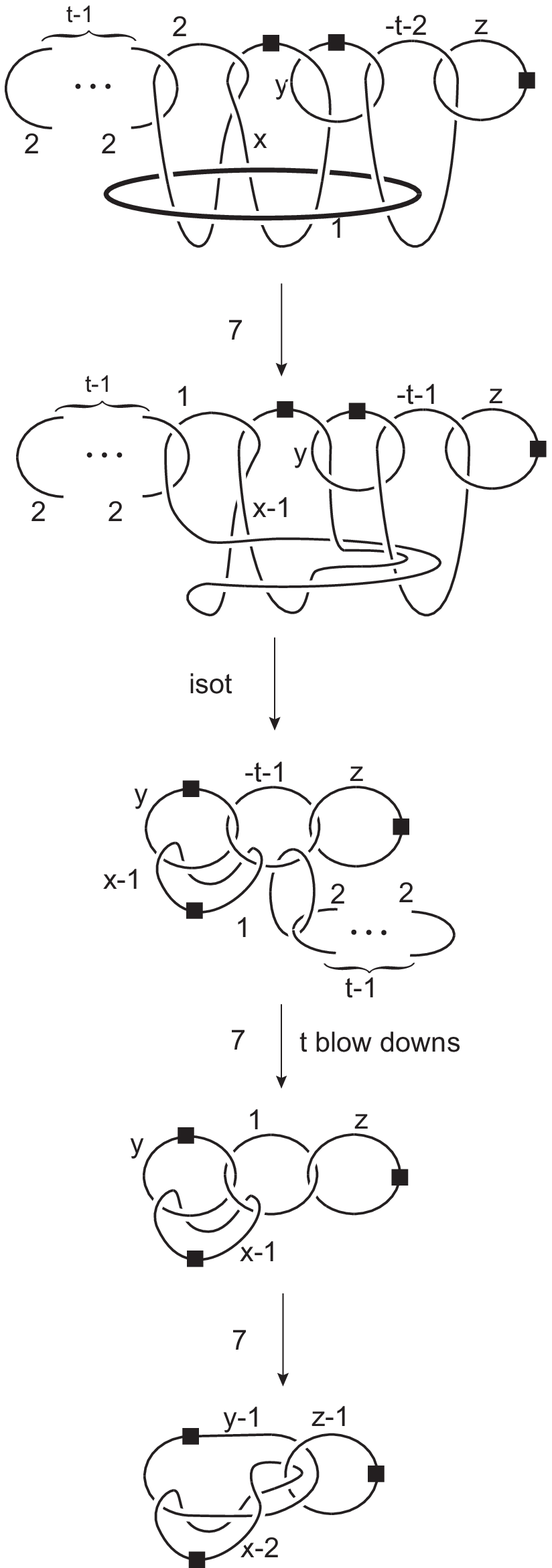}
\hcaption{This diagram shows how to add a $2$-handle with framing $-1$ to the families $(a)$ with $t>0$, $(b),(c)$ and $(e)$ of Theorem~\ref{l:2}, in order to obtain a $4$-manifold with boundary $S^{1}\times S^{2}$. The framings $x,y$ and $z$, and the linear plumbings represented by black squares differ in the four families. The arrows in the diagram represent either blow downs or isotopies.}
\label{f:4Fam}
\end{center}
\end{figure}

\begin{figure}
\begin{center}
\frag[ss]{7}{$\displaystyle\boundary^\partial$}
\frag[ss]{s-1 blow downs}{$s-1$ blow downs}
\frag[ss]{1}{$-1$}
\frag[ss]{s-1}{$s\!-\!1$}
\frag[ss]{-t-1}{$-\!t\!\!-\!\!1$}
\frag[ss]{-s-1}{$-\!s\!-\!1$}
\frag[ss]{-s-2}{$-\!s\!-\!2$}
\frag[ss]{2}{$-2$}
\frag[ss]{s}{$s$}
\frag{a}{$(i)$}
\frag{b}{$(ii)$}
\frag{c}{$(iii)$}
\frag[ss]{t}{$t$}
\frag[ss]{0}{$0$}
\frag[ss]{AAAA}{See Figure 6} 
\frag[ss]{-b_1}{$-b_1$}
\frag[ss]{-b_k}{$-b_k$}
\frag[ss]{-c_1-1}{$-\!c_1\!\!-\!\!1$}
\frag[ss]{-c_2}{$-c_2$}
\frag[ss]{-c_l}{$-c_\ell$}
\frag[ss]{-c_1}{$-c_1$}
\includegraphics[scale=0.8]{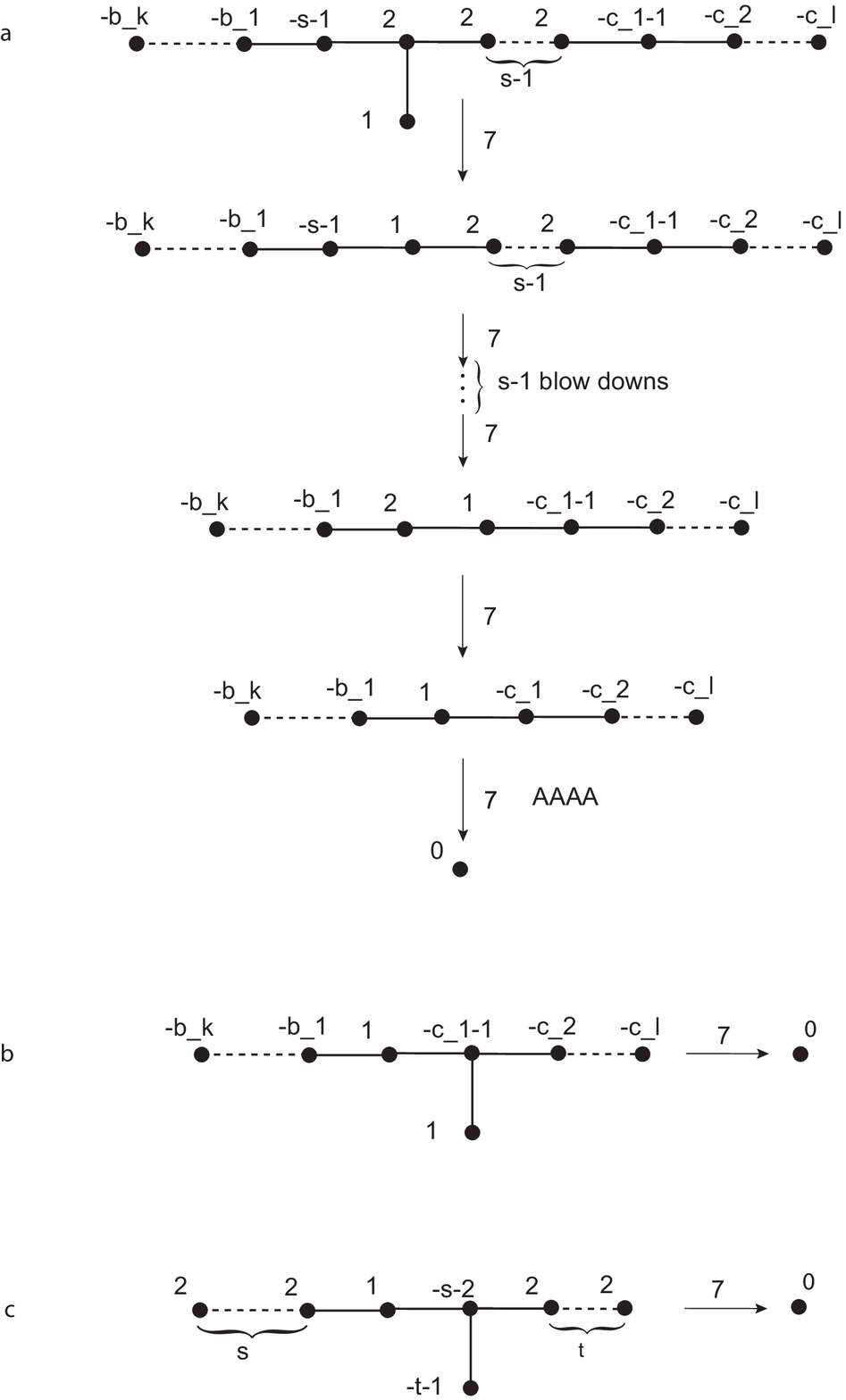}
\hcaption{Graph $(i)$ is obtained from the last diagram of Figure~\ref{f:4Fam} by substituting $x,y,z$ and the black squares with the data from the graph in family $(a)$ with $t>0$ or in family $(b)$. Graphs $(ii)$ and $(iii)$ correspond respectively to considering in the last diagram of Figure~\ref{f:4Fam} the data from families $(c)$ and $(e)$ respectively.} 
\label{f:4Fam,2}
\end{center}
\end{figure}

\begin{figure}
\begin{center}
\frag[ss]{7}{$\displaystyle\boundary^\partial$}
\frag[ss]{1}{$-1$}
\frag[ss]{s-1}{$s\!-\!1$}
\frag[ss]{-s-1}{$-\!s\!-\!1$}
\frag[ss]{-s-2}{$-\!s\!-\!2$}
\frag[ss]{2}{$-2$}
\frag[ss]{3}{$-3$}
\frag[ss]{x}{$x$}
\frag[ss]{x+1}{$x+1$}
\frag[ss]{x}{$x$}
\frag[ss]{isotopy}{isotopy}
\frag[ss]{0}{$0$}
\frag[ss]{blabla}{series of blow downs and Figure 6} 
\frag[ss]{exp}{$x=-3,\blacksquare=\emptyset$}
\frag[ss]{-b_1}{$-b_1$}
\frag[ss]{-b_k}{$-b_k$}
\frag[ss]{-c_1-1}{$-\!c_1\!\!-\!\!1$}
\frag[ss]{-c_2}{$-c_2$}
\frag[ss]{-c_l}{$-c_\ell$}
\frag[ss]{-c_1}{$-c_1$}
\includegraphics[scale=0.8]{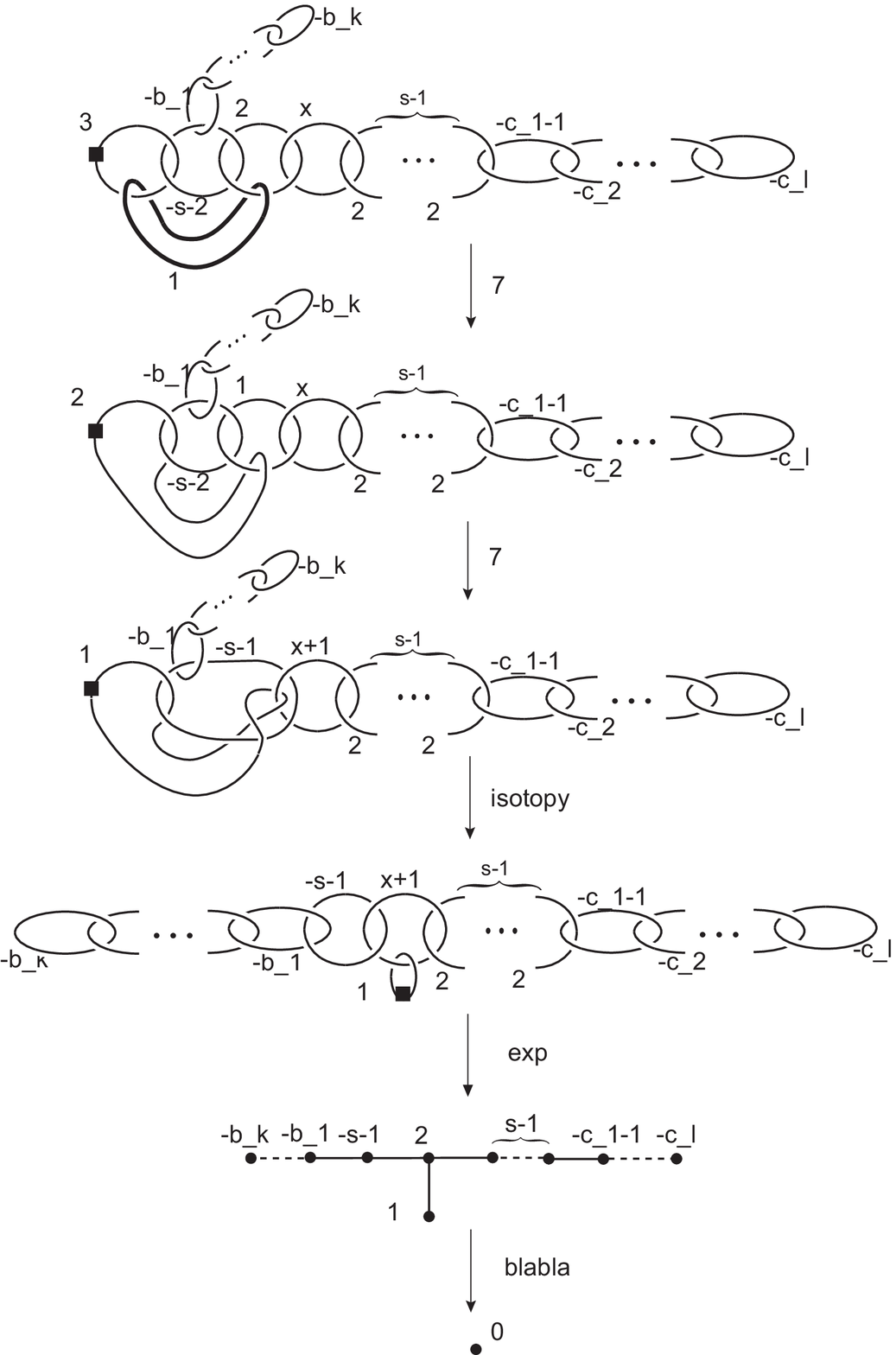}
\hcaption{This diagram shows how to add a $2$-handle with framing $-1$ (the thicker one) to the  families $(a)$ with $t=0$, and $(d)$ of Theorem~\ref{l:2}, in order to obtain a $4$-manifold with boundary $S^{1}\times S^{2}$. Family $(a)$ with $t=0$ satisfies $x=-3$ and $\blacksquare=\emptyset$, while family $(d)$ has $x=-t-3$ and $\blacksquare$ represents a $(-2)$-chain of length $t$. The vertical arrows in the diagram represent either blow downs or isotopies. The last two diagrams have been specialized to family $(a)$ with $t=0$.} 
\label{f:2Fam}
\end{center}
\end{figure}

\begin{rem}
The $2$-handle additions used in Lemmas~\ref{p:clb} and \ref{l:s1xs2} were suggested by the analysis of standard subsets of $\Z^n$ done in order to proof Theorems~\ref{l:strings1} and \ref{l:2}. In fact, consider for example the graph $P$ belonging to the family $(a)$ in Theorem~\ref{l:2} with $t=1, s=2, k=2$, and $b_{1}=b_{2}=2$. One can easily find the following embedding in $\Z^9$.
\begin{center}
\frag[ss]{e1-e2-e6+e7}{$e_1\!-\!e_{2}\!-\!e_{6}\!+\!e_{7}$}
\frag[ss]{e2-e3+e5}{$e_2\!-\!\bm{e_{3}}\!+\!e_{5}$}
\frag[ss]{-e3-e5}{$-\!\bm{e_3}\!-\!e_{5}$}
\frag[ss]{e1+e6}{$e_{1}\!+\!e_{6}$}
\frag[ss]{-e6-e7-e8-e9}{$-\!e_6\!-\!e_{7}\!-\!e_{8}\!-e_{9}$}
\frag[ss]{-e8+e9}{$-\!e_8\!+\!e_{9}$}
\frag[ss]{what}{$-\!e_7\!+\!e_{8}$}
\frag[ss]{-e2-e1+e4}{$-\!e_2\!-\!e_{1}\!+\!e_{4}$}
\frag[ss]{e2+e3+e4}{$e_2+\!\bm{e_{3}}\!+\!e_{4}$}
\frag[ss]{2}{$-2$}
\frag[ss]{3}{$-3$}
\frag[ss]{4}{$-4$}
\includegraphics[scale=0.9]{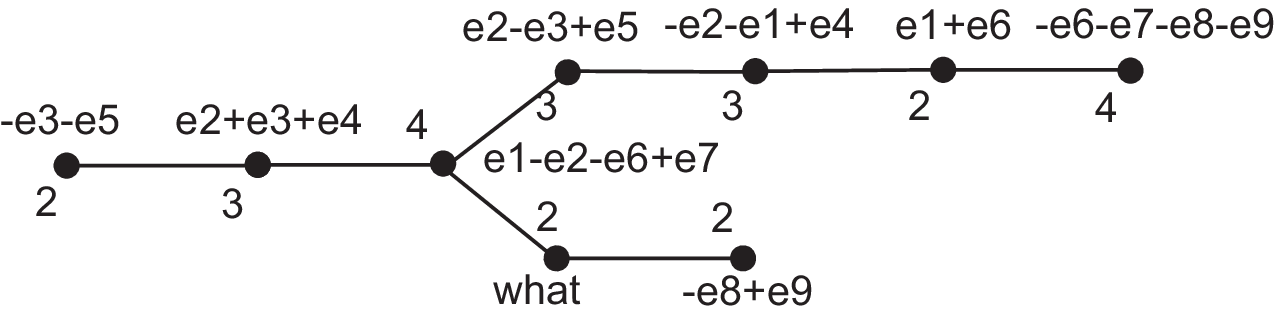}
\end{center}
Notice that, if we erase the basis vector $\bm{e_{3}}$ we are left with a graph having a vertex, the one on the extreme left, with weight $-1$ and with corresponding generator sent to $-e_5$ by the embedding. By erasing this time the $e_{5}$ vector we obtain again a vertex with weight $-1$ in the resulting graph. It is possible to continue in this way until we are left, in this example, with the graph
\begin{center}
\frag{1}{$-1$}
\frag{-e9}{$-e_{9}$}
\frag{a}{$e_{9}$}
\includegraphics[scale=0.9]{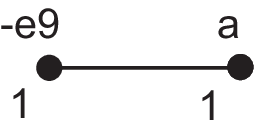}
\end{center}
This \virg{erasing} operation without a formal framework is meaningless, however it suggests a chain of blow downs, which turn the graph into a $4$-manifold with boundary $S^{1}\times S^{2}$. In fact, the thick circle added to the first diagram in Figure~\ref{f:4Fam} intersects the Kirby diagram of the family $(a)$ in Theorem~\ref{l:2} as suggested by the basis vector $e_{3}$ in this example. The same happens with all the plumbing graphs that we have studied.
\end{rem}

Let $M_P$ be the $4$-manifold corresponding to a plumbing graph $P$ as in Theorem~\ref{l:2}. We have seen in Section~\ref{s:prel} that $M_P$ admits a Kirby diagram consisting of a strongly invertible link in $S^3$ with respect to an involution $u$. Let us call $M'$ the $4$-manifold with boundary $S^1\times S^2$ obtained from $M_P$ by adding a $4$-dimensional $2$-handle as in Lemma~\ref{l:s1xs2}. As shown in Figure~\ref{f:sym_ncl}, this $2$-handle can be added in such a way that we obtain again a strongly invertible link with respect to the involution $u$.
\begin{figure}
\begin{center}
\includegraphics[scale=0.7]{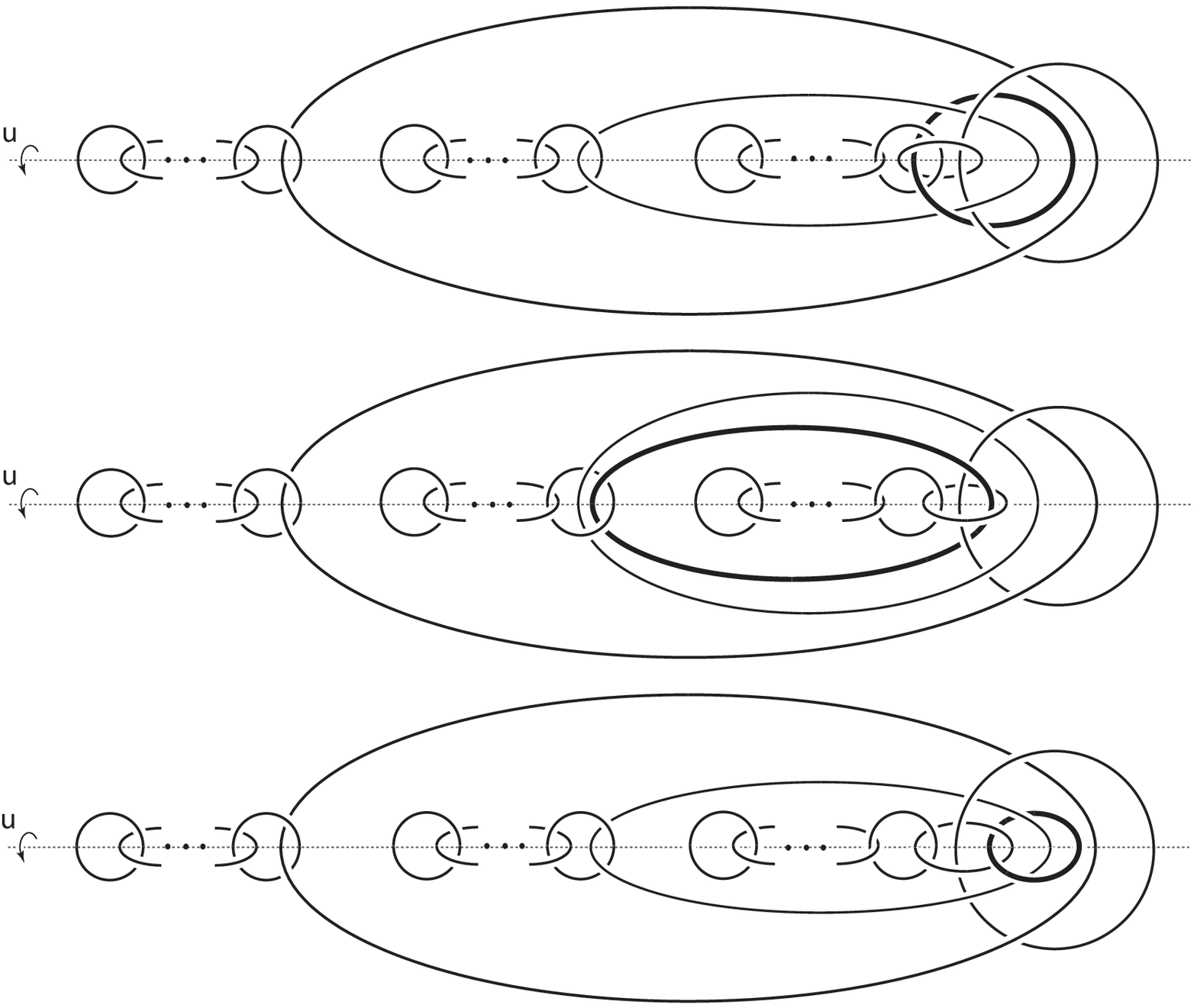}
\hcaption{These three strongly invertible links can be isotoped to the first diagrams in Figures~\ref{f:4Fam} and \ref{f:2Fam} as follows (for sake of clarity we have omitted the framings). The first invertible link can be isotoped to the first diagram in Figure~\ref{f:4Fam} when we consider on it the data from family $(a)$ with $t>0$ of Theorem~\ref{l:2}. The second invertible link is easily isotoped to the first diagram in Figure~\ref{f:4Fam} when we consider on it the data either from family $(b),(c)$ or $(e)$ of Theorem~\ref{l:2}. Finally, the last invertible link is easily isotoped to the first diagram in Figure~\ref{f:2Fam}, when we consider on it the data either from family $(a)$ with $t=0$ or $(d)$ from Theorem~\ref{l:2}.}
\label{f:sym_ncl}
\end{center}
\end{figure}
As in Section~\ref{s:Montesinos_compl}, it follows from Montesinos' theorem, Theorem~\ref{t:Mo}, that $M_P$ is the double cover of $D^4$ branched along a surface $B_P$, which consists of bands plumbed according to the graph $P$. In turn, $M'$ is the double cover of $D^4$ branched over the surface $B'$, which is $B_P$ with an additional band. The addition of this band corresponds to a ribbon move on the Montesinos link $\Ml_P=\partial B_P$. This ribbon move necessarily leads to two unlinked unknots, since $\partial M'=S^1\times S^2$ and, by \cite{b:KT}, whenever $S^1\times S^2$ double branch covers $S^3$, the branch set is the unlink of two unknotted components. The discussion of this section proves the following result.

\begin{lem}\label{l:sigma_ncl}
For a graph $P$ as in Theorem~\ref{l:2}, depending on the number of connected components of $\Ml_{P}$, we have the following possibilities:
\begin{itemize}
\item If $\Ml_{P}$ is a knot, then it bounds a ribbon disc in $S^{3}$. 
\item If $\Ml_{P}$ is a $2$-component link, then it bounds a ribbon surface in $S^3$, which is the disjoint union of a disc and a M\"obius band.
\item If $\Ml_{P}$ is a $3$-component link, then it bounds a ribbon surface in $S^3$, which is the disjoint union of a disc and an annulus. 
\end{itemize}
\end{lem}
\begin{proof}
The result follows easily from elementary facts on the classification of surfaces with boundary, as in the proof of Lemma~\ref{l:sigma_cl}.
\end{proof}

Having established the existence of the desired ribbon surfaces, we are now ready to  prove  the main result of this paper, modulo the technical proofs of Theorems~\ref{l:strings1} and \ref{l:2}, which will be carried out in the remaining sections.

\begin{proof}[Proof of Theorem \ref{t:int}]
We start assuming that $Y_\Gamma=\partial W$, where $W$ is a rational homology ball. Since $\Gamma\in\wp$, the intersection form of the $4$-dimensional plumbing $M_{\Gamma}$ is negative definite (see Remark~\ref{r:neg}) and hence, $X_{\Gamma}:=M_{\Gamma}\cup_{\partial}(-W)$ is a closed smooth negative $4$-manifold. By Donaldson's Theorem, the intersection lattice of $X_{\Gamma}$ is isomorphic to $(\Z^{n},-\mathrm{Id})$, where $n=b_2(X_\Gamma)$. Clearly, the group $H_{2}(M_{\Gamma};\Z)$ is isomorphic to $\Z^{n}$ and the intersection lattice $(\Z^{n},Q_{\Gamma})$ has a basis $\{v_{0},v_{1,1},...,v_{n_{1},1},v_{1,2},...,v_{n_{2},2},v_{1,3},...,v_{n_{3},3}\},$ where $n=n_{1}+n_{2}+n_{3}+1$, in which $Q_{\Gamma}$ has the form \eqref{e:Q_P}. Therefore, via the embedding $M_{\Gamma}\subset X_{\Gamma}$ we can view the above basis, and hence $\Gamma$, as a standard subset of $\Z^{n}$. If $\Gamma$ has two complementary legs [resp.\ no complementary legs] then it belongs to the list in Theorem~\ref{l:strings1} [resp.\ Theorem~\ref{l:2}] and the existence of the surface $\Sigma$ and of the ribbon immersion follows from Lemma~\ref{l:sigma_cl} [resp. Lemma~\ref{l:sigma_ncl}].

The arguments we use to prove the \virg{only if} part of the statement coincide with those in \cite[Proof of Theorem~1.2, (2) implies (1)]{b:Li}. For the reader's convenience we include them  here. Assume that there exist a surface $\Sigma$ and a ribbon immersion $\Sigma\looparrowright S^{3}$ such that $\partial\Sigma=\Ml_{\Gamma}$ and $\chi (\Sigma)=1$. Let $\Sigma'\subset D^4$ be a smoothly embedded surface obtained by pushing the interior of $\Sigma$ inside the $4$-ball. The $2$-fold covering $Y_\Gamma\rightarrow S^3$ branched over $\Ml_\Gamma$ extends to a $2$-fold covering $W\rightarrow D^4$ branched over $\Sigma'$ (see \cite[p.\ 277--279]{b:Ka}). We conclude by showing that $W$ is a rational homology ball, note that $Y_\Gamma=\partial W$. By definition of $\Sigma'$, we may assume that the function distance from the origin $D^4\rightarrow [0,1]$ restricted to $\Sigma'$ is a proper Morse function with only index-$0$ and index-$1$ critical points. This implies that $W$ has a handlebody decomposition with only $0$-, $1$- and $2$-handles (see \cite[lemma at p.\ 30--31]{b:CH}. Therefore, from 
$$b_0(W)-b_1(W)+b_2(W)=\chi (W)=2\chi (D^4)-\chi(\Sigma ')=1$$
we deduce $b_1(W)=b_2(W)$. Since $b_1(Y_\Gamma)=0$ and $H_1(W,Y_\Gamma;\Q)\cong H^3(W;\Q)=0$, the exact homology  sequence of the pair $(W,Y_\Gamma)$
$$\cdots\rightarrow\underbrace{H_1(Y_\Gamma;\Q)}_{=0}\rightarrow H_1(W;\Q)\rightarrow\underbrace{H_1(W,Y_\Gamma;\Q)}_{=0}\rightarrow\cdots$$
shows $b_1(W)=0$. It follows that $H_\ast(W;\Q)\cong H_\ast(D^4;\Q)$. 
\end{proof}

\section{Contractions of good sets}\label{s:contractions}
In the rest of the paper, Sections~\ref{s:contractions} to \ref{s:ss}, we carry out the proof of Theorems~\ref{l:strings1} and \ref{l:2}. In the current section we introduce some notation, define contractions of good sets and give some preliminary results.

Recall that, in Section~\ref{s:prel}, we fixed $e_1,...,e_n$, the standard basis of $(\Z^n,-\mathrm{Id})$ and we defined the set $J=\{(s,\alpha)|\, s\in\{0,1,...,n_\alpha\}, \alpha\in\{1,2,3\}\}$ indexing the vertices $v_{s,\alpha}$ of three-legged plumbing graphs $P\subseteq\Z^n$. We will adopt the following notation: for each $i\in\gbra{1,...,n}$, $(s,\alpha)\in J$ and  $S\subseteq P$ we put
\begin{align*}
E_i(P) & :=\{(s,\alpha)\in J\,|\ v_{s,\alpha}\cdot e_i\neq 0\}, \\
V_S & :=\{j\in\{1,...,n\}\,|\ e_j\cdot v_{s,\alpha}\neq 0\mbox{ for some }v_{s,\alpha}\in S\},\\
p_i(P) & :=\left| \gbra{ j\in\{1,...,n\}\,|\ |E_j(P)|=i } \right|.
\end{align*}
Given $e\in\Z^n$ with $e\cdot e=-1$, we denote by $\pi_e:\Z^n\to\Z^n$ the orthogonal projector onto the subspace orthogonal to $e$, i.e.
$$\pi_e(v):=v+(v\cdot e)e\in\Z^n,\s\s\forall v\in\Z^n.$$
Let $P\subseteq \Z^n$ be a good set and suppose that $E_h(P)=\{(s,\alpha),(t,\beta)\}$ for some $h\in\{1,...,n\}$ and $(s,\alpha),(t,\beta)\in J$. Then, we say that the subset $P'\subseteq\Z^{n-1}=\langle e_1,...,e_{h-1},e_{h+1},...,e_n\rangle$ defined by
$$P':=(P\setminus\{v_{s,\alpha},v_{t,\beta}\})\cup\{\pi_{e_h}(v_{t,\beta})\}$$
is obtained from $P$ by a \textbf{contraction}, and we write $P\searrow P'$. Moreover, we say that $P$ is obtained from $P'$ by an \textbf{expansion}, and we write $P'\nearrow P$.

For good subsets $P\subseteq\Z^n$ with two complementary legs, $L_2$ and $L_3$ (their associated strings are related to one another by Riemenschneider's point rule), we extend the definition of contraction to the following operation. Suppose that for some $i\in\{1,...,n\}$ and  for some $(s,1)\in J$, we have $E_i(P)=\{0,(s,1)\}$  and let $v_{t,1}$ be any final vector in $L_1$. Since  the leg $L_2$ is connected to the central vertex, we have $v_{1,2}\cdot v_0=1$, and therefore there exists $k\in\{1,...,n\}$ such that $k\in V_{v_0}\cap V_{v_{1,2}}$ and $v_0=\tilde v_0\pm e_k$. We say that the subset $P'\subseteq\Z^{n-1}=\langle e_1,...,e_{i-1},e_{i+1},...,e_n\rangle$ defined by
$$P':=(P\setminus\{v_0,v_{s,1},v_{t,1}\})\cup\{\pi_{e_i}(v_{s,1})\}\cup\{v_{t,1}\pm e_k\}$$
is obtained from $P$ by a \textbf{contraction}, and we write $P\searrow P'$. Moreover, we say that $P$ is obtained from $P'$ by an \textbf{expansion}, and we write $P'\nearrow P$. 

\begin{exm}\label{e:cl}
The two examples in Figure~\ref{f:ej} are extended contractions of sets with complementary legs. Following the above notation $P$ is the set on the left of the arrow and $P'$ the one on the right. Notice that the first example shows that the extended contraction of a standard set is again a standard set.
\hfill\qed
\end{exm}

\begin{figure}
\begin{flushleft}
\frag[ss]{v31}{$e_{2}\!+\!e_{3}\!+\!\bm{e_{4}}$}
\frag[ss]{8}{$e_{2}\!-\!e_{3}$}
\frag[ss]{v21}{$e_{1}\!-\!e_{2}$}
\frag[ss]{4}{$-\!e_{2}\!-\!e_{1}\!+\!\bm{e_{4}}\!+\!e_{5}$}
\frag[ss]{v12}{$-\!e_5\!-\!e_6$}
\frag[ss]{v13}{$-\!e_5\!+\!e_6$}
\frag[ss]{5}{$\!e_2\!-\!e_3\!+e_{5}$}
\frag[ss]{v32}{$\!e_2\!+\!e_3$}
\frag[ss]{v23}{$\!e_2\!-\!e_3\!+\!\bm{e_{4}}$}
\frag[ss]{A}{$e_1\!+\!e_2\!+e_{3}$}
\frag[ss]{B}{$-\!e_3\!-\!\bm{e_4}\!+e_{5}$}
\includegraphics[scale=0.6]{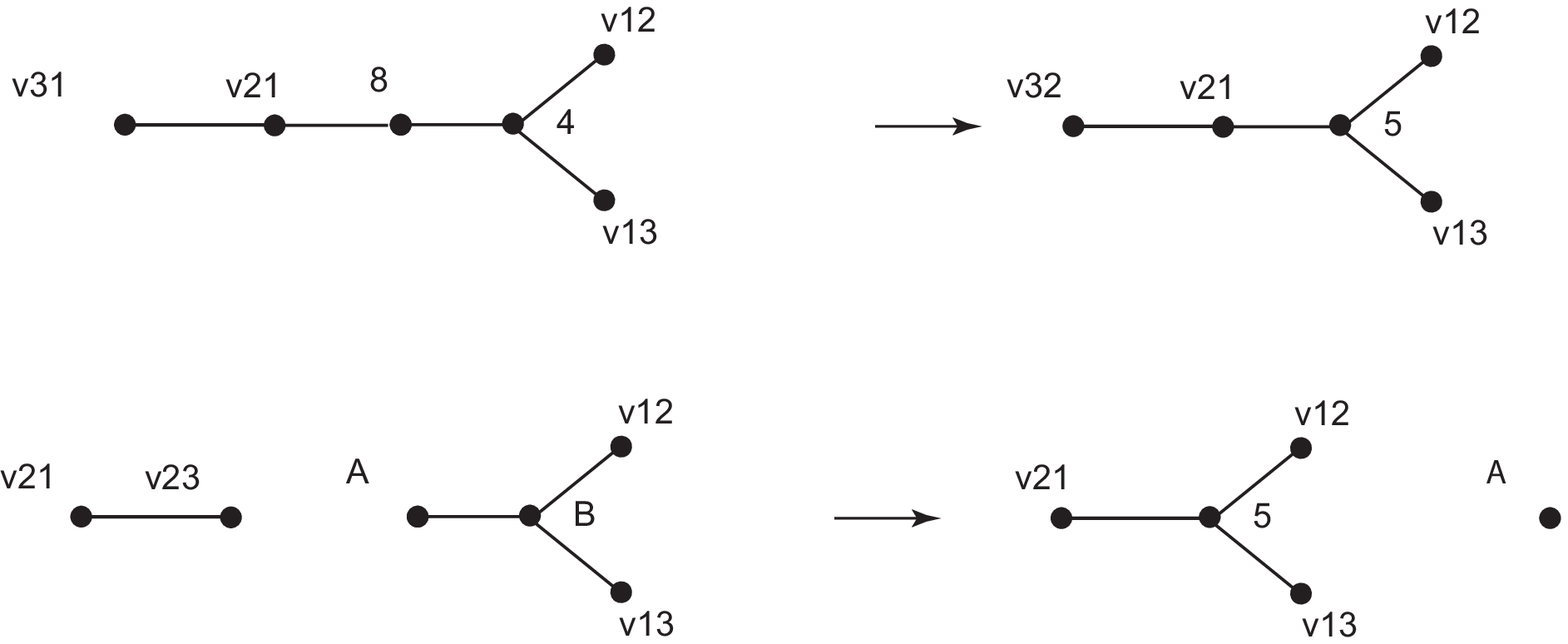}
\ccaption{Contractions of sets with complementary legs, which are the length-one legs.}
\label{f:ej}
\end{flushleft}
\end{figure}

Once that we have settled all the necessary definitions, we introduce here several preliminary results concerning good subsets.

\begin{lem}\label{l:ind}
The elements of a good subset $P\subseteq\Z^n$ are linearly independent over $\Z$. 
\begin{proof}
If $P$ is a linear set, the claim follows from \cite[Remark~2.1]{b:Li}. If, on the contrary $P$ has a trivalent vertex, let us further assume that  all the $\gamma_{s,\alpha}$'s in $Q_P$  are equal to $1$ (the general case is a straightforward extension of this one). Then, $Q_P$ turns out to be the intersection matrix associated to the Seifert space $Y_P=Y(-a_0;(\alpha_1,\beta_1)(\alpha_2,\beta_2)(\alpha_3,\beta_3))$, for some $\alpha_i,\beta_i\in\Z$ such that
$$ \frac{\alpha_i}{\beta_i}=a_{1,i}-\frac{1}{\displaystyle a_{2,i} - \frac{ \bigl. 1}{\displaystyle \ddots\ _{ \displaystyle{a_{n_{i}-1,i}} -\frac {\bigl. 1}{a_{n_i,i}}}} }\ .$$ 
By definition of Seifert invariants we have $\alpha_i>1$; by definition of good subset it holds $a_0\geq 3$ and $a_{k,i}\geq 2$ for every $k\in\{1,...,n_i\}$, which implies $\beta_i/\alpha_i<1$ (see Section~\ref{s:prel}). It follows that the determinant of $Q_P$ is non zero, since by \eqref{e:det} it holds
$$|\det(Q_P)|=\abs{\cbra{-a_0+\frac{\beta_1}{\alpha_1}+\frac{\beta_2}{\alpha_2}+\frac{\beta_3}{\alpha_3}}\alpha_1\alpha_2\alpha_3}\neq 0.$$
Now, consider the $n\times n$ matrix $M$ having as rows  the coordinates of the elements of $P$ with respect to the standard basis of $\Z^n$. This matrix satisfies $Q_P=-MM^t$ and therefore $|\det M|=|\det Q_P|^{1/2}\neq 0$, which readily gives the claim. 
\end{proof}
\end{lem}

\begin{rem}\label{r:li}
Notice that Lemma \ref{l:ind} remains valid for reducible sets $P\subseteq\Z^n$ whose incidence matrix $Q_P$ has the form \eqref{e:Q_P} or \eqref{e:Q_P_linear}. In fact, in this case the matrix $M$ (see the  proof of Lemma \ref{l:ind}) is a diagonal block matrix, i.e.
$$
M
=
\cbra{
\begin{array}{cccc}
M_1\\
& M_2\\
&&\ddots\\
&&&M_k
\end{array}
}.
$$ 
Let us denote by $P_i\subseteq P$ the irreducible subset having incidence matrix $Q_{P_i}=-M_iM_i^t$. Then, by  Lemma \ref{l:ind}, the matrices $M_i$ are non singular, hence $M$ and $Q_P=-MM^t$ are nonsingular as well. 
\end{rem}

For the reader's convenience we now include two Lemmas from \cite{b:Li}. More precisely, Lemma~\ref{l:lisca_1} corresponds to Lemma 2.5 in \cite{b:Li}, which gives important information on $p_1(P)$ and $p_2(P)$ coming from the assumption $I(P)<0$. On the other hand, Lemma \ref{l:n.e} points out the most relevant properties of  good linear sets with $p_1(P)>0$. Its proof follows from Section $3$ in \cite{b:Li}, although in Lemma \ref{l:n.e} we have dropped the assumption $I(P)<0$.  

\begin{lem}[Lisca]\label{l:lisca_1} 
Let $P\subseteq\Z^n$ be a subset of cardinality $n$ with $I(P)<0$. Then,
\begin{align}\label{e:p_i}
2p_1(P)+p_2(P)>\sum_{j=4}^n(j-3)p_j(P). 
\end{align}
\end{lem} 

\begin{lem}[Lisca]\label{l:n.e}
Let $n\geq 3$ and $P_n\subseteq\Z^n$ be a good linear set with $p_1(P_n)>0$, i.e.\ there exist $i\in\{1,...,n\}$ and $(s,1)\in J$ such that $E_i(P_n)=\{(s,1)\}$. Then $P_n$ is standard and there exist $j\in\{1,...,n\}$ and $\lambda\in\Z$ such that
\begin{itemize}
\item[$(1)$] $v_{s,1}=\lambda e_i\pm e_j$,
\item[$(2)$] $I(P)=\lambda^2-4$,
\item[$(3)$] If $n\geq 4$ there exist $h\in\{1,...,n\}$ and $(t,1),(r,1)\in\{(1,1),(n,1)\}$ such that $E_h(P)=\{(1,1),(n,1)\}$, $a_{t,1}=2$ and $a_{r,1}>2$,
\item[$(4)$] $P_n$ can be obtained by final $(-2)$-vector expansions from the set $P_3:=\{w_{1,1},w_{2,1},w_{3,1}\}\subseteq\Z^3$ where, up to replacing $P_3$  with $\Omega P_3$, $\Omega\in\Upsilon_{P_{3}}$, we have  $(w_{1,1},w_{2,1},w_{3,1})=(e_1+e_2,\lambda e_3-e_2,e_2-e_1)$.
\end{itemize}
\end{lem}
\begin{proof}
Throughout this proof we assume that the reader is familiar with the work done in \cite{b:Li}. Statement $(1)$ corresponds to Lemma 3.2(2) in \cite{b:Li}. If $n>3$, by \cite[Lemma 3.2(2) and (3)]{b:Li} we have that, for some $(t,1),(r,1)\in\{(s-1,1),(s+1,1)\}$, the set 
$$P_{n-1}:=(P_n\setminus\{v_{s,1},v_{t,1}\})\cup\{\pi_{e_j}(v_{t,1})\}$$ 
is good, $E_j(P_{n-1})=\{(r,1)\}$, $|e_j\cdot v_{r,1}|=1$ and $I(P_n)=I(P_{n-1})-2+a_{s,1}$. Combining the proofs of \cite[Lemma 3.2 and Proposition 3.3]{b:Li} we conclude that $I(P_{n-1})=-3$. In fact, the set $P_{n-1}$ satisfies the hypothesis of \cite[Lemma~3.2]{b:Li} and hence there exists a contraction $P_{n-1}\searrow P_{n-2}$. Applying the lemma $n-4$ times we obtain a sequence $P_{n-1},...,P_3$ of good sets with $p_1(P_{n-1}),...,p_1(P_3)>0$. Moreover, since $|e_j\cdot v_{r,1}|=1$ then $a_{r,1}=2$ and so $I(P_{n-1})=...=I(P_3)$. In \cite[Proposition~3.3]{b:Li} the claim follows from the fact that, since $I(P)<0$, there is only one possibility for $P_3$ (up to replacing $P_3$  with $\Omega P_3$, $\Omega\in\Upsilon_{P_{3}}$). However, in our case, by the construction of the sequence $P_{n-1},...,P_3$ and since $|e_j\cdot v_{r,1}|=1$, we have that $|v\cdot e_k|\leq 1$, for every $v\in P_3$ and every $k\in \{1,...,n\}$. This implies, just like the assumption $I(P)<0$, that there is only one possibility for $P_3$, again up to replacing $P_3$  with $\Omega P_3$, $\Omega\in\Upsilon_{P_{3}}$. This set $P_3$ satisfies $I(P_3)=-3$, therefore $I(P_{n-1})=-3$, and hence $I(P_n)=-3-2+\lambda^2+1=\lambda^2-4$, so $(2)$ holds.

In the proof of \cite[Proposition~3.3(3)]{b:Li} the assumption $I(P)<0$ is not used. Hence, we obtain $(3)$. The proof of \cite[Corollary~3.5]{b:Li} goes through taking into account that the base case in the induction ($n=3$) depends on the number $\lambda$, and it is immediate to check that the possible cases are the sets considered in $(4)$.
\end{proof}

\section{Bad Components}\label{s:bc}

In the forthcoming sections, in order to determine all possible standard subsets $P\subseteq\Z^n$ with $I(P)<-1$, we shall study acutely contractions of good sets. The idea is to choose these contractions in a suitable way, to obtain again good sets and this will be possible as long as the good sets have no \virg{bad components}. In this section we introduce this concept and establish some properties of bad components under set contractions.

Let $n\geq 3$ and $\widetilde{C}=\{v_{s-1,\alpha},v_{s,\alpha},v_{s+1,\alpha}\}\subseteq \Z^n$ be a connected graph such that $a_{s-1,\alpha}=a_{s+1,\alpha}=2$, $a_{s,\alpha}>2$ and $E_j(\widetilde{C})=\{(s-1,\alpha),(s,\alpha),(s+1,\alpha)\}$ for some $j\in\{1,,,.n\}$. Up to changing $e_j$ with $-e_j$, the graph $\widetilde C\subseteq\Z^n$ is of the form 
\begin{center}
{
\frag{-2}{$-2$}
\frag{-asa}{$-a_{s,\alpha}$}
\frag{ei-ej}{$e_i-e_j$}
\frag{ej+}{$e_j+...$}
\frag{eimej}{$-e_i-e_j$}
\includegraphics[scale=0.7]{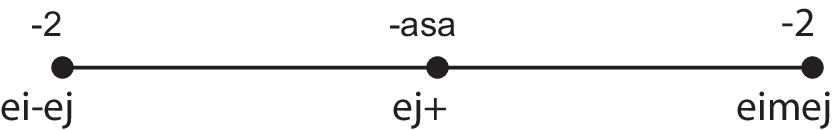}
}
\end{center}
We perform on $\widetilde C$, an arbitrary number of times and in any order, the following two expansions:
\begin{itemize} 
\item \textbf{Right expansion with final ($-2$)-vector}. This expansion, sketched below, can be performed on any connected linear set $C\subseteq\Z^n$, whenever there exists $i\in \{1,...,n\}$ such that $E_i(C)$ consists of the two final vertices in $C$. 
 \begin{flushleft}
\frag{ej+}{\color{grigio}{$e_i+...$}}
\frag{e+}{\color{grigio}{$-e_i+...$}}
\frag{ei+}{$e_k{\color{grigio}-e_i+...}$}
\frag{-eh}{$-e_i-e_k$}
\includegraphics[scale=0.7]{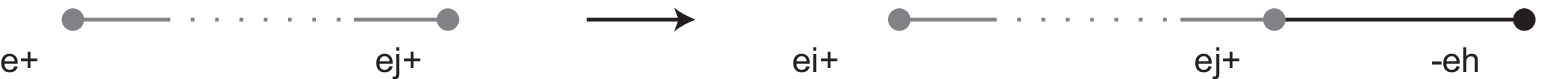}
\end{flushleft}
\item Analogously, we can perform a \textbf{left expansion with final ($-2$)-vector}, sketched below.
\begin{flushleft}
\frag{ej+}{\color{grigio}{$e_i+...$}}
\frag{e+}{\color{grigio}{$-e_i+...$}}
\frag{ei+}{$e_k{\color{grigio}+e_i+...}$}
\frag{-eh}{$e_i-e_k$}
\includegraphics[scale=0.7]{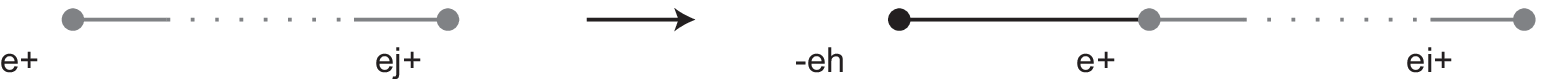}
\end{flushleft}
\end{itemize} 
A connected component $C\subseteq\Z^n$ obtained from $\widetilde C$ in this way will be called a \textbf{linear bad component}, and we will denote by $v_{\ast}$ the vector $v_{s,\alpha}\in C$. The number of linear bad components in a set $P$ will be denoted by $b(P)$.

\begin{exm}
A linear bad component after the sequence \virg{right expansion, left expansion, right expansion} is the following:
\begin{center}
\frag{-2}{$-2$}
\frag{-3}{$-3$}
\frag{-asa}{$-a_{s,\alpha}$}
\frag{AA}{$-e_h+e_k-e_\ell$}
\frag{BBB}{$e_i-e_j+e_h$}
\frag{CCC}{$\underbrace{e_j+...}_{\mbox{$=v_\ast$}}
$}
\frag{DD}{$-e_i-e_j$}
\frag{EEE}{$e_i-e_h-e_k$}
\frag{fff}{$e_k+e_\ell$}
\noindent\includegraphics[scale=0.6]{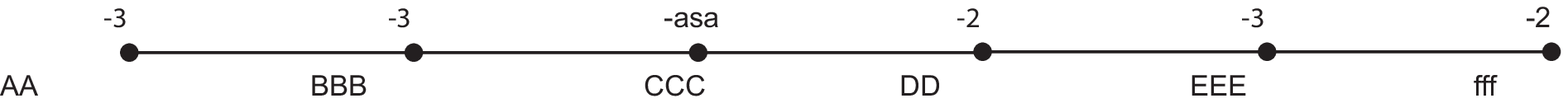}
\end{center}
\hfill\qed
\end{exm} 

In a similar fashion, we define bad components with a trivalent vertex. We start with an arbitrary linear bad component $C=\{v_{1,\alpha},...,v_{k,\alpha}\}\subseteq\Z^n$ and we attach two $(-2)$-vectors to $v_{1,\alpha}$ (or analogously to $v_{k,\alpha}$), as shown in the following graph: 
\begin{equation}\label{e:bad6}
\begin{split}
\frag{vka}{$\color{grigio}v_{k,\alpha}$}
\frag{v1a}{${\color{grigio}v_{1,\alpha}}+e_k$}
\frag{ek+}{$-e_k+e_h$}
\frag{ek-}{$-e_k-e_h$}
\includegraphics[scale=0.7]{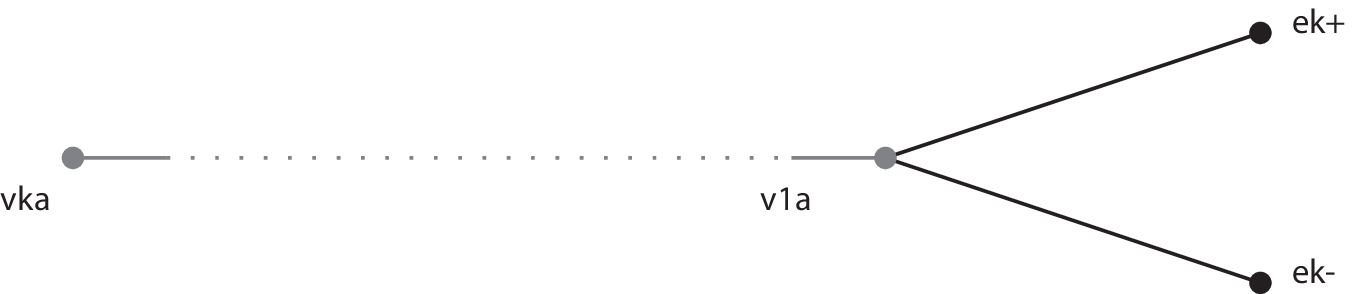}
\end{split}
\end{equation}
As before, we now perform an arbitrary number of final ($-2$)-vector expansions on the length-one legs of the graph in~\eqref{e:bad6}. We call any connected component $D\subseteq\Z^n$ obtained in this way a \textbf{three-legged bad component} and, as before, we denote by $v_{\ast}$ the vector $v_{s,\alpha}\in D$. Notice that every three-legged bad component can be thought of as obtained from the following  set $\widetilde D$,
\begin{equation}\label{e:bad5}
\begin{split}
\frag{-2}{$-2$}
\frag{-asa}{$-a_{\ast}$}
\frag{HHH}{$e_i-e_j$}
\frag{ZZZ}{$-e_k+e_h$}
\frag{WWW}{$-e_k-e_h$}
\frag{KKK}{$-e_i-e_j+e_k$}
\frag{m3}{$-3$}
\frag{m2}{$-2$}
\frag{ej+}{$\underbrace{e_j+...}_{\mbox{$=v_\ast$}}$}
\includegraphics[scale=0.7]{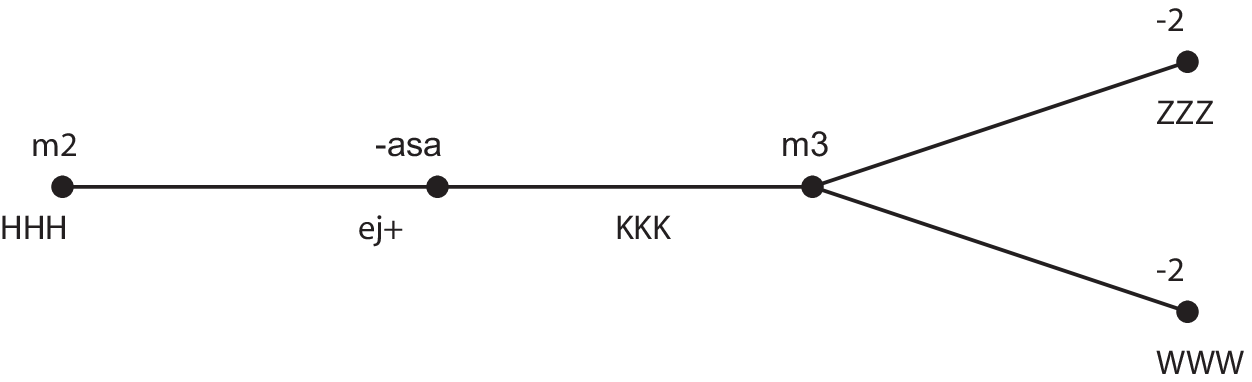}
\end{split}
\end{equation}
through a combination of the above described expansions: we erase $e_k$ from the central vertex and expand the horizontal leg in the graph $\widetilde D$ as if it was a linear bad component; afterwards, we add $+e_k$ to one of the final vectors and we proceed with the expansions of the two length-one legs.

We focus now our attention on the study of bad components. More precisely, in Lemma \ref{l:st.sin.bad} we prove that standard sets have no bad components of any type and in Lemmas \ref{l:b_3} and \ref{l:bad} we establish some properties of bad components under set contractions. 

\begin{lem}\label{l:st.sin.bad}
Every three-legged standard subset $P\subset\Z^n$ with $n\geq 5$ and $I(P)<-1$, has no bad components of any type.
\end{lem}
\begin{proof}
We will argue by contradiction. Suppose that $P$ has a bad component, then, since the graph $P$ is standard and it has a valence three vertex, its only connected component must be a three-legged bad component. By definition, $P$ is obtained from the set $\widetilde D$ in \eqref{e:bad5}
by ($-2$)-vector expansions and it is immediate to check that $I(P)=I(\widetilde D)$. The set $P$, being standard, is good and then $\widetilde D\subset\Z^5$ is good too (it is irreducible since $c(P)=c(\widetilde D)=1$ and, by construction, the incidence matrix $Q_{\widetilde D}$ is of the form \eqref{e:Q_P}). Therefore, with the notation of \eqref{e:bad5},
there must exist $e_\ell\in\Z^5$ and $\lambda\in\Z$ such that $v_\ast=e_j+\lambda e_\ell$. Since, by definition, $a_\ast>2$ we have $\lambda>1$ and hence we obtain $I(\widetilde D)\geq -1$, a contradiction.
\end{proof}

In the following result we study under which conditions a good set with no bad components can develop a bad component after a contraction. 

\begin{lem}\label{l:b_3}
Let $P\subseteq\Z^n$ be a good set with no bad components of any type and suppose that there is a contraction $P\searrow P'$, where $P'$ is a good set with bad components given by
$$P':=(P\setminus\{v_{s,\alpha},v_{t,\beta}\})\cup\{\pi_{e_h}(v_{t,\beta})\}$$
for some $h\in\{1,...,n\}$ and $(s,\alpha),(t,\beta)\in J$. Then,
\begin{itemize}
\item[$(1)$] $|V_{v_{s,\alpha}}|>2$ and $|V_{v_{t,\beta}}|>2$,
\item[$(2)$] $P'$ has either one linear bad component or one three-legged bad component, but not both.
\item[$(3)$] Consider the restricted projection $\pi_{e_h}:P\setminus v_{s,\alpha}\rightarrow P'$, let $D'\subseteq P'$ be the bad component in $P'$ and let $D:=\pi_{e_h}^{-1}(D')$. Then, $v_{t,\beta}\in D$ and $v_{s,\alpha}$ is not orthogonal to $D$.
\end{itemize}
\end{lem}
\begin{proof}
In $P$ there are only four possible configurations for $D$, $v_{s,\alpha}$ and $v_{t,\beta}$, which we analyze separately:

\textit{The vector $v_{s,\alpha}$ is orthogonal to $D$, i.e.\ for all $v\in D$, $v_{s,\alpha}\cdot v=0$, and $v_{t,\beta}\not\in D$}. In this case $D$ is a connected component and moreover $D=D'$, which contradicts the assumption that $P$ has no bad components of any type.

\textit{The vector $v_{s,\alpha}$ is orthogonal to $D$ and $v_{t,\beta}\in D$}. If $v_{\ast}=\pi_{e_h} (v_{t,\beta})$ (see the definition of linear and three-legged bad components above) then we would have that $D\subseteq P$ is a bad component and this contradicts the assumption of the lemma. Observe that $v_{s,\alpha}\cdot v_{t,\beta}=0$ implies that $|V_{v_{s,\alpha}}\cap V_{D'}|\geq 2$. In fact, since $v_{s,\alpha}$ and $v_{t,\beta}$ are orthogonal and $h\in V_{v_{s,\alpha}}\cap V_{v_{t,\beta}}$ there must be another index $k\neq h$ in the intersection $V_{v_{s,\alpha}}\cap V_{v_{t,\beta}}$. Since $v_{\ast}\neq\pi_{e_h} (v_{t,\beta})$, by  definition of bad component (either linear or three-legged), we have $\{(t,\beta)\}\varsubsetneq E_k(D')$. Hence, there exists $(r,\gamma)\in J$, $(r,\gamma)\neq (t,\beta),$ with $(r,\gamma)\in E_k(D')$. Note also that $v_{s,\alpha}\cdot v_{r,\gamma}=0$ implies $|V_{v_{s,\alpha}}\cap V_{D'}|\geq 2$. Let us now consider the vector 
$$v_{s,\alpha}':=-\sum_{j\in V_{D'}}(v_{s,\alpha}\cdot e_j)e_j,$$ 
which by the above discussion has square $v_{s,\alpha}'\cdot v_{s,\alpha}'\leq -2$. It follows that the set $(D'\setminus\{v_\ast\} )\cup\{v_{s,\alpha}'\}\subset\Z^{|D'|-1}$ is a good set and therefore, by Lemma~\ref{l:ind}, its $|D'|$ vectors are linearly independent, but this is not possible since they belong to the span of $|D'|-1$ vectors.

\textit{The vector $v_{s,\alpha}$ is not orthogonal to $D$ and $v_{t,\beta}\not\in D$}. We use the same argument of the preceding case. The vector
$$v_{s,\alpha}':=-\sum_{j\in V_{D'}}(v_{s,\alpha}\cdot e_j)e_j$$ 
has square $v_{s,\alpha}'\cdot v_{s,\alpha}'\leq -2$. In fact, since there exists $v\in D$ such that $v\cdot v_{s,\alpha}=1$, there must exist $j\in V_{v_{s,\alpha}}\cap V_v$ thus, taking into account that $v_\ast$ is internal in $D$ (so $v\neq v_\ast$), the claim follows from the definition of bad component (either linear or three-legged). Now,  considering as in the preceding case the good set $(D'\setminus\{v_\ast\})\cup\{v_{s,\alpha}'\}\subset\Z^{|D'|-1}$, we obtain $|D'|$ linearly independent vectors in the span of $|D'|-1$ vectors.

\textit{The vector $v_{s,\alpha}$ is not orthogonal to $D$ and $v_{t,\beta}\in D$}. Notice that this case is the only possibility left and that the three other cases lead to contradiction, implying $(3)$. As a consequence we obtain that a single contraction $P\searrow P'$ cannot produce neither two linear bad components, nor together a three-legged bad component and a linear bad component. Therefore, $(2)$ holds.

Since $v_{t,\beta}\in D$, the projection $\pi_{e_h}(v_{t,\beta})$ belongs to the bad component $D'$. By definition of bad component (linear or three-legged) we know that $|V_{\pi_{e_h}(v_{t,\beta})}|\geq 2$ and then $|V_{v_{t,\beta}}|>2$.
It only remains to prove that $|V_{v_{s,\alpha}}|>2$ and to this aim we further consider two subcases:

$(i)$ Suppose first that $v_{t,\beta}\cdot v_{s,\alpha}=1$ and let us put $t=s-1$ (the case $t=s+1$, possible only if $D'$ is a linear bad component, can be handled analogously). If $|V_{v_{s,\alpha}}|=2$ and $a_{s,\alpha}>2$ then $v_{s,\alpha}=\pm e_h+\lambda e_j$ where $j\in\{1,...,n\}$ and $\lambda\in\Z$, $|\lambda|>1$. Thus, replacing the vectors $v_{s,\alpha}$ and $v_{t,\beta}$ respectively with $\pi_{e_h}(v_{s,\alpha})$ and $\pi_{e_h}(v_{t,\beta})$, we obtain a good set of $n$ vectors whose associated incidence matrix is of the form \eqref{e:Q_P} or \eqref{e:Q_P_linear}. By Remark~\ref{r:li}, this $n$ vectors are linearly independent, but at the same time they belong to the span of the $n-1$ vectors $\{e_1,...,\hat{e}_h,...,e_n\}$. This contradiction implies that if $|V_{v_{s,\alpha}}|=2$, then necessarily $a_{s,\alpha}=2$. Let us consider the biggest $\ell\geq 0$ such that the set $S:=\{v_{s,\alpha},v_{s+1,\alpha},...,v_{s+\ell,\alpha}\}$ is connected. If $a_{s,\alpha}=a_{s+1,\alpha}=...=a_{s+\ell,\alpha}=2$ then the set $(P\setminus S)\cup\{\pi_{e_h}(v_{t,\beta})\}$ is a good set of $n-(l+1)$ linearly independent vectors (see Lemma \ref{l:ind}) in the span of $n-(l+2)$ vectors (notice that a length $d$ chain of ($-2$)-vectors, which is connected to some other vector of square smaller that $-2$, is contained in the span of $d+1$ basis vectors). This contradiction shows that there exists a smallest $r\in\{1,...,\ell\}$ such that $a_{s+r}>2$ and it is easy to check that for some $k\in\{1,...,n\}$ 
$$V_{v_{s+r-1,\alpha}}\cap V_{v_{s+r,\alpha}}=\{e_k\}\s\mbox{and}\s |v_{s+r,\alpha}\cdot e_k|=1.$$
Since $|\bigcup_{i=0}^{r-1}V_{v_{s+i,\alpha}}|=r+1$, it follows that the set $(P\setminus\{v_{s,\alpha},...,v_{s+r-1,\alpha}\})\cup\{\pi_{e_h}(v_{t,\beta})\}\cup\{\pi_{e_k}(v_{s+r,\alpha})\}$ is a good set of $n-r$ linearly independent vectors in the span of $n-(r+1)$ basis vectors. This contradiction yields $|V_{v_{s,\alpha}}|>2$ and therefore, if $v_{s,\alpha}\cdot v_{t,\beta}=1$, then $(1)$ holds.

$(ii)$ Suppose now that $v_{s,\alpha}\cdot v_{t,\beta}=0$. If $V_{v_{s,\alpha}}=\{h,j\}$ then, since $v_{s,\alpha}$ is not orthogonal to $D$, it holds $j\in V_{D'}$ and we necessarily have that $v_{s,\alpha}=\pm e_j\pm e_h$. Thus, $a_{s,\alpha}=2$ and $V_{v_{s,\alpha}}\cap V_{v_{t,\beta}}=\{h,j\}$. Now we have to distinguish if $D'$ is a three-legged or a linear bad component. Let us begin assuming that $D'$ is a three-legged bad component. Then, by its definition and since $a_{s,\alpha}=2$, there are only two possibilities: the vector $v_{s,\alpha}$ is attached to the final vector of one of the legs of $D'$ which grow by final $(-2)$-vector expansions and $v_{t,\beta}$ is the final vector in the other one, or the vector $v_{s,\alpha}$ is attached to the leg containing $\pi_{e_h}^{-1}(v_\ast)$ and $v_{t,\beta}$ is the central vertex. In both cases we have that $(D\cup \{v_{s,\alpha}\})\subseteq P$ is a three-legged bad component, contradicting the assumption of the lemma. In case $D'$ is a linear bad component, by its definition and since $a_{s,\alpha}=2$, there is only one possibility: the vector $v_{s,\alpha}$ is attached to a final vector of $D'$ and $v_{t,\beta}$ is the other final vector. Then, the set $(D\cup \{v_{s,\alpha}\})\subseteq P$ is a linear bad component, contradicting again the assumption of the lemma. Therefore, if $v_{s,\alpha}\cdot v_{t,\beta}=0$, we have that $|V_{v_{s,\alpha}}|>2$ as claimed.
\end{proof}

The last result in this section shows how to overcome the difficulties with the bad components: if after a contraction a good set with no bad components develops a bad component, there is always another possible contraction that yields a good set with no bad components. It is a key result to prove Lemma~\ref{l:red}.

\begin{lem}\label{l:bad}
Let $P\subseteq\Z^n$ be a good subset with no bad components of any type and $I(P)<-1$. Assume that, for some $i\in\{1,...,n\}$ and $(s,\alpha),(t,\beta)\in J$, the set obtained by the contraction
$$P':=(P\setminus\{v_{s,\alpha},v_{t,\beta}\})\cup\{\pi_{e_i}(v_{t,\beta})\}\subseteq\Z^{n-1}$$
is good with a bad component $D'\subseteq P'$, linear or three-legged. Then, interchanging the roles of $v_{s,\alpha}$ and $v_{t,\beta}$, the set
$$P'':=(P\setminus\{v_{t,\beta},v_{s,\alpha}\})\cup\{\pi_{e_i}(v_{s,\alpha})\}\subseteq\Z^{n-1}$$
is good with no bad components of any type.
\end{lem}
\begin{proof}
We start by showing that $v_{s,\alpha}\in P$ is not the central vertex. In fact, if $v_{s,\alpha}=v_0$, then $P'$ would be a linear good set with $b(P')=1$  (see Lemma~\ref{l:b_3}) and $c(P')\geq 3$. However, since $I(P)<-1$ and, by Lemma~\ref{l:b_3}, $a_{s,\alpha}>2$, we have $I(P')+b(P')<0$. Therefore, by \cite[Proposition~4.10]{b:Li2}, we obtain the contradiction $c(P')\leq 2$. This forces $v_{s,\alpha}\neq v_0$ and hence, the incidence matrix $Q_{P''}$ has the form \eqref{e:Q_P} or \eqref{e:Q_P_linear} since, by Lemma~\ref{l:b_3}, $|V_{v_{s,\alpha}}|\geq 3$.

Next, since the contraction $P\searrow P'$ produces a bad component $D'$, we know by Lemma~\ref{l:b_3} that $\pi_{e_i}(v_{t,\beta})\in D'$. We claim that $\pi_{e_i}(v_{t,\beta})\neq v_\ast\in D'$ (see definition of bad component above). In fact, by Lemma~\ref{l:b_3}\,$(3)$, $v_{s,\alpha}\cdot v=1$, for some $v\in D$, where $D:=\pi_{e_i}^{-1}(D')$, and hence there exists $j\in\{1,...,n\}$ such that $j\in V_{v_{s,\alpha}}\cap V_v$. If $\pi_{e_i}(v_{t,\beta})=v_\ast$, then $v_{t,\beta}$ is internal in $D$ and we conclude $j\neq i$. By definition of bad component (linear or three-legged) we have that $|E_j(D')|\geq 2$, and so there exists $w\in D'$, $w\neq v,v_\ast$, such that $j\in V_w$ and $w\cdot v_{s,\alpha}=0$. Thus, it follows that the vector
$$v_{s,\alpha}':=-\sum_{j\in V_{D'}}(v_{s,\alpha}\cdot e_j)e_j$$ 
has square $v_{s,\alpha}'\cdot v_{s,\alpha}'\leq -2$. Then, the set $(D'\setminus\{v_\ast\})\cup \{v_{s,\alpha}'\}\subseteq\Z^{|D'|-1}$ has an associated incidence matrix ot the form \eqref{e:Q_P} or \eqref{e:Q_P_linear}, and by Remark~\ref{r:li} its $|D'|$ vectors are linearly independent, contradicting the fact that they belong to the span of $|D'|-1$ vectors. Therefore $\pi_{e_i}(v_{t,\beta})\neq v_\ast$, and consequently, by definiton of bad component, every vector $v\in P$ linked to $v_{t,\beta}$ is also linked to $v_{t-1,\beta}$ or to $v_{t+1,\beta}$. Thus, $P''$ is irreducible.

Once we have shown that $P''\subseteq\Z^{n-1}$ is good, it remains to prove that it has no bad components of any type. Suppose by contradiction that there exists a bad component $D''\subseteq P''$, and consider first the case $v_{s,\alpha}\cdot v_{t,\beta}=0$. Applying Lemma~\ref{l:b_3} to the contraction $P\searrow P''$ we obtain that $\pi_{e_i}(v_{s,\alpha})\in D''$ and that $v_{t,\beta}$ is not orthogonal to $\pi_{e_i}^{-1}(D'')$; hence, $D'\cap D''\neq\emptyset$. Notice that since $\pi_{e_i}(v_{s,\alpha})\in D''$ and $v_{s,\alpha}\not\in D'$ it holds that $D''\setminus D'\neq\emptyset$. Since in $P$ there is only one trivalent vertex, at least one between $D'$ and $D''$ is a linear bad component, say it is $D''$. In $D''$ there are two final vectors: we denote by $v$ be the one belonging to $D''\cap D'$ and by $w$ be the one in $D''\setminus D'$. By definition of linear bad component, we know that there are two elements $j,h\in V_v\cap V_w$ and so the vector 
$$\tilde w:=-\sum_{j\in V_{D'}}(w\cdot e_j)e_j$$
has square $\tilde w\cdot\tilde w\leq -2$. Consider the vector $v_\ast\in D'$ and observe that, by definition of bad component (either linear or three-legged), we have that the set $(D'\setminus\{v_\ast\})\cup\{\tilde w\}\subseteq\Z^{|D'|-1}$ is good. Thus, by Lemma~\ref{l:ind}, its $|D'|$ vectors are linearly independent, which contradicts the fact that they belong to the span of $|D'|-1$ basis vectors.

We are left with the case $v_{s,\alpha}\cdot v_{t,\beta}=1$ and let us suppose, without loss of generality as explained before, that $D''$ is a linear bad component. If $\{i\}=V_{v_{s,\alpha}}\cap V_{v_{t,\beta}}$, then replacing $v_{s,\alpha}$ and $v_{t,\beta}$ respectively with $\pi_{e_i}(v_{s,\alpha})$ and $\pi_{e_i}(v_{t,\beta})$, we would obtain a set of $n$ linearly independent vectors in the span of $n-1$ basis vectors (recall that by Lemma~\ref{l:b_3}, we have $|V_{v_{s,\alpha}}|>2$ and $|V_{v_{t,\beta}}|>2$). Therefore $\{i\}\varsubsetneq V_{v_{s,\alpha}}\cap V_{v_{t,\beta}}$. Since the contraction $P\searrow P'$ produces a bad component $D'$, we have, by Lemma~\ref{l:b_3}, $\pi_{e_i}(v_{t,\beta})\in D'$. Analogously, applying Lemma~\ref{l:b_3} to the contraction $P\searrow P''$ we obtain $\pi_{e_i}(v_{s,\alpha})\in D''$. Then, the definition of bad component guarantees that for every $j\in\{1,...,n\}$ we have $|e_j\cdot v_{s,\alpha}|,|e_j\cdot v_{t,\beta}|\leq 1$. Since $v_{s,\alpha}\cdot v_{t,\beta}=1$, it follows that  $| V_{v_{s,\alpha}}\cap V_{v_{t,\beta}}|\geq 3$ and then, since $\pi_{e_i}(v_{t,\beta})\in D'$ we have that $|V_{D'}\cap V_{v_{s,\alpha}}|\geq 2$. Applying Lemma~\ref{l:b_3} to the contraction $P\searrow P''$ we obtain that $\pi_{e_i}(v_{s,\alpha})\in D''$ is a final vector (since $v_{s,\alpha}\cdot v_{t,\beta}=1$). Therefore, in this case the vector $v_{s,\alpha}$ has the same properties as the vector $v$ of the preceding case. Calling $w$ the other final vector in $D''$ we can repeat, word for word, the argument of the preceding case arriving to a contradiction. Therefore, the good set $P''$ has no bad components of any type and the lemma is proved.
\end{proof}

\begin{rem}
The author of the present paper discovered that, although the statement of \cite[Proposition~5.3]{b:Li} is correct, the arguments used to prove the claim \virg{$S'$ has no bad components} (p.\ 450, line 8) are incorrect. Lemma~\ref{l:bad} can be used to prove such claim.
\end{rem}

\section{Determination of good sets}\label{s:clas}
The aim of this section is to determine all good sets $P$ with $I(P)<-1$. We start by reformulating the definition of complementary legs and by showing how to construct from a linear good set a three-legged good set with two complementary legs. Note that, by \eqref{e:p_i}, good sets $P$ with $I(P)<-1$ satisfy that $p_1(P)$ or $p_2(P)$ is greater than zero. We consider in the first subsection the easier case of good sets satisfying $p_1(P)>0$, while in the second one we analyze good sets with $p_1(P)=0$. Finally in the last part we consider the general case and conclude with the main result of this section, which is Proposition~\ref{p:clave}. 

In Section~\ref{s:prel}, the definition of complementary legs is given in terms of their associated strings. For our purposes, it is now more convenient to bear in mind that $L_\alpha,L_\beta\subseteq P\subseteq\Z^n$ are complementary legs if they can be obtained as a sequence of final ($-2$)-vector expansions of the length-one legs $\widetilde L_\alpha:=\{v_{1,\alpha}=-e_k+e_h\}$ and $\widetilde L_\beta:=\{v_{1,\beta}={-e_k-e_h}\}$, where $\widetilde L_\alpha$ and $\widetilde L_\beta$ are defined up to the action of an element of $O(n;\Z)$. Notice that $\widetilde L_\alpha$ and $\widetilde L_\beta$ have associated strings $(2)$ and $(2)$ respectively. The final ($-2$)-vector expansions change the strings $(2)$ and $(2)$ as the operations described in Remark~\ref{r:riem}. Therefore, the strings associated to the legs $L_\alpha$ and $L_\beta$ are related to one another by Riemenschneider's point rule. Thus, they are complementary legs according to the definition given in Section~\ref{s:prel}. Notice that, by definition, in every three-legged bad component there are two complementary legs. Moreover, the expansion of the set \eqref{e:bad5} along its horizontal leg corresponds to the expansion of a set with complementary legs defined in Section~\ref{s:prel}.

Given a good linear  set $P'$, we can construct a good set $P$ with a trivalent vertex by adding  two complementary legs to $P'$. This construction produces a large family of good sets with a central vertex and, if we further require that $P$ has no bad components of any type and that $I(P)<-1$, we have a complete description of this family, as shown in the following lemma.

\begin{lem}\label{l:cl}
Let $n=n_1+n_2+n_3+1$ and $P_n\subseteq\Z^n$ be a good set without bad components of any type, with two complementary legs, $L_2$ and $L_3$, and such that $I(P_n)<-1$. Then,
\begin{itemize}
\item[$(1)$] There exists $j\in\{1,...,n\}$ such that the central vertex $v_0$ is equal to $\tilde v_0\pm e_j$, where $j\in V_{L_2\cup L_3}$ and $V_{\tilde v_0}\subseteq V_{L_1}$. Moreover, $V_{L_1}\cap V_{L_2\cup L_3}=\emptyset$.
\item[$(2)$] The set $S_{n_1+1}:=L_1\cup\{\tilde v_0\}$ is a linear good subset of $\Z^{n_1+1}$, $I(S_{n_1+1})<0$, $n_1\geq 2$ and there exists a sequence of contractions $S_{n_1+1}\searrow S_{n_1}\searrow\cdots\searrow S_3$ such that, for each $k=3,...,n_1$, the set $S_k$ is good without bad components.
\item[$(3)$] For every $i\in\{1,...,n\}$ and every $(s,\alpha)\in J$, we have $|v_{s,\alpha}\cdot e_i|\leq 1$.
\item[$(4)$] There exists a sequence of good sets $P_n, P_{n-1},...,P_5$ such that, for every $i=5,...,n$, $P_i$ is a good set without bad components of any type, and moreover we have either $(I(P_{i+1}),c(P_{i+1}))=(I(P_i),c(P_i))$ or
$$I(P_{i+1})-1\geq I(P_i)\ \ \mbox{and}\ \ c(P_{i+1})+1\geq c(P_i).$$
\item[$(5)$] If $P_n$ is a standard set, then every set in the sequence $P_n, P_{n-1},...,P_5$, built in $(4)$ is standard too.
\end{itemize} 
\end{lem}
\begin{proof}
Since $L_2$ and $L_3$ are complementary legs, there exists a sequence of contractions to the legs $\widetilde L_2=\{v_{1,2}=e_j+e_k\}$ and $\widetilde L_3=\{v_{1,3}=e_j-e_k\}$ for some $j,k\in\{1,...,n\}$. Since $v_0\cdot v_{1,2}=v_0\cdot v_{1,3}=1$, then $j\in V_{v_0}$, $|e_j\cdot v_0|=1$ and $k\not\in V_{v_0}$. If $v\in L_1$ was such that $V_v\cap\{e_j,e_k\}\neq\emptyset$, then we could not have $v\cdot v_{1,2}=0$ and $v\cdot v_{1,3}=0$, therefore $V_{L_1}\cap V_{L_2\cup L_3}=\emptyset$. Since $v_{1,1}\cdot v_0=1$, we have $V_{v_0}\cap V_{L_1}\neq\emptyset$ and hence $v_0=\tilde v_0\pm e_j$ with $V_{\tilde v_0}\subseteq v_{L_1}$ and $(1)$ holds.

By definiton of complementary legs, a simple calculation yields $|V_{L_2\cup L_3}|=n_2+n_3$ and $I(L_2\cup L_3)=-2$. From the first equality it follows that $S_{n_1+1}=L_1\cup\{\tilde v_0\}\subseteq\Z^{n_1+1}$ is a good set. Since $P_n=L_1\cup \{v_0\}\cup L_2\cup L_3$ and $I(\cdot)$ is additive under set union, we have $I(P_n)=I(L_1\cup \{v_0\})-2<-1$ and a straightforward computation gives then $I(S_{n_1+1})<0$. Since $P_n$ is a good set, we have $n_1\geq 1$. If $n_1=1$, then $S_2$ should be a standard set of $\Z^2$ satisfying $I(S_2)<0$, which is easily seen to be impossible. Since $P_n$ has no bad components of any type, the set $S_{n_1+1}$ has no bad components and hence it fulfills the hypothesis of \cite[Corollary~5.4]{b:Li} and $(2)$ holds. 

Applying \cite[Proposition~5.2]{b:Li}, we obtain that $|v\cdot e_i|\leq 1$, for every $v\in S_{n_1+1}$ and for every $i\in\{1,...,n\}$. We have shown before that $|v_0\cdot e_j|=1$ and,  by definition of complementary legs, for every $v\in L_2\cup L_3$ and for every $i\in\{1,...,n\}$ we have $|v\cdot e_i|\leq 1$. This proves $(3)$.

Assertions $(4)$ and $(5)$ follow from the above arguments. In fact, the sets $P_n,...,P_{n_1+3}$ are the contractions of the two complementary legs, which satisfy $(I(P_{i+1}),c(P_{i+1}))=(I(P_i),c(P_i))$ for $i\in\{n_1+3,...,n-1\}$, and the set $P_{n_1+3}$ is $L_1\cup \{v_0\}\cup\widetilde L_2\cup\widetilde L_3$. All these sets are good without bad components. The rest of the sets in the sequence, namely $P_{n_1-1+3},...,P_5$, are obtained using:
\begin{itemize} 
\item \cite[Theorem~6.4]{b:Li} applied to $S_{n_1+1}$, if this set is standard, which implies that $P_n$ is standard and $(5)$ follows;
\item \cite[Corollary~5.4]{b:Li} applied to $S_{n_1+1}$, if this set is only good, and in this case $(4)$ follows.
\end{itemize}
In both cases we must take into account the following consideration: if in the contraction given by \cite[Theorem~6.4]{b:Li} or by \cite[Corollary~5.4]{b:Li} we discard the vector $\tilde v_0$, then we must add $\pm e_j=v_0-\tilde v_0$ to any final vector in the contracted set, that will then play the role of central vertex with two complementary legs attached (see Example~\ref{e:cl}). In this way we obtain the desired sequence of good sets, with no bad components of any type. 
\end{proof}

\subsection{Good sets with $p_1>0$}\label{s:p1>0}
Along this subsection we will say that a subset $P\subseteq\Z^n$ satisfies the \textbf{working assumptions} when
\begin{itemize}\index{assumptions!\textbf{(W1)}--\textbf{(W4)}} 
\item[\textbf{(w1)}] $n=n_1+n_2+n_3+1\geq 5$,
\item[\textbf{(w2)}] $P=\{v_0,v_{1,1},...,v_{n_3,3}\}$ is a good set with a trivalent vertex and no bad components of any type,
\item[\textbf{(w3)}] $I(P)<-1$, 
\item[\textbf{(w4)}] $p_1(P)>0$. 
\end{itemize}

The aim of this subsection is to prove Propositions \ref{p:set_5} and \ref{p:p1>0_final}, which imply that a set $P$ fulfilling the working assumptions is necessarily standard, and can be obtained from the standard subset of $\Z^5$ given in Proposition \ref{p:set_5} by a finite sequence of expansions. 
\begin{prop}\label{p:set_5}
Let $n=5$ and  $P\subseteq\Z^n$ be a good subset with $p_1(P)\geq 1$ and $I(P)<-1$. Then, $I(P)=-4$ and, up to replacing $P$ with $\Omega P$ where $\Omega\in\Upsilon_{P}$, the plumbing graph $P$ with its embedding in the standard diagonal lattice is:
\begin{center}
\frag[s]{a-b}{$e_2-e_3$}
\frag[s]{c-a}{$e_1-e_2$}
\frag[s]{a+c-d}{$e_2+e_3+e_4$}
\frag[s]{d+e}{$-e_4+e_5$}
\frag[s]{d-e}{$-e_4-e_5$}
\includegraphics[scale=0.7]{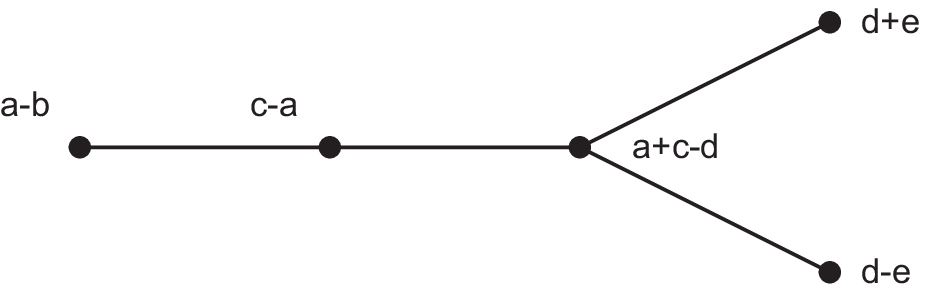}$\s\s\s$
\end{center}
\end{prop}

\begin{prop}\label{p:p1>0_final}
Let $P_n$ be a set satisfying the working assumptions. Then, $P_n$ is standard and there is a sequence of contractions
$$P_n\searrow\ P_{n-1}\searrow\cdots\searrow\ P_6\searrow P_5$$
such that $P_k$ is standard and $I(P_k)=-4$, for every $k=5,...,n$.
\end{prop}

Throughout the section we will use the following notation: let $P\subseteq\Z^n$ be a subset satisfying the working assumptions. Then, by \textbf{(w4)}, there exist $i\in\gbra{1,...,n}$ and $(s,\alpha)\in J$ such that $E_i(P)=\{(s,\alpha)\}$.

\begin{lem}\label{l:p1>0_1} 
Consider $P\subseteq\Z^n$ satisfying the working assumptions. Then,
\begin{itemize}
\item[$(1)$] The vertex $v_{s,\alpha}$ is internal and not central.
\item[$(2)$] There exists $j\in\{1,...,n\}$ such that 
\begin{align*}
&V_{v_{s,\alpha}}=\{i,j\},\\
&E_j(P)=\{(s-1,\alpha),(s,\alpha),(s+1,\alpha)\},\\ 
&|v_{s-1,\alpha}\cdot e_j|=|v_{s,\alpha}\cdot e_j|=|v_{s+1,\alpha}\cdot e_j|=1.
\end{align*}
\end{itemize}
\end{lem}

\begin{proof}
The set $P$, being good, is irreducible and therefore we have $|V_{v_{s,\alpha}}|\geq 2$. We claim that $|V_{v_{s,\alpha}}|= 2$. In fact, assume by contradiction that  $|V_{v_{s,\alpha}}|>2$, and consider the set $P'$ obtained from $P$ by replacing $v_{s,\alpha}$ with $\pi_{e_i}(v_{s,\alpha})$. Thus, $P'$ is contained in the span of the $n-1$ vectors $e_1,...,e_{i-1},e_{i+1},...,e_n$. However, since $\pi_{e_i}(v_{s,\alpha})\cdot\pi_{e_i}(v_{s,\alpha})\leq-2$, the matrix $Q_{P'}$  still has the form \eqref{e:Q_P} and, by  Remark~\ref{r:li}, $P'$  consists of $n$ linearly independent vectors, which gives a contradiction. 

Next, by assumption $i\in V_{v_{s,\alpha}}$, and therefore there exists $j\neq i$ such that  $V_{v_{s,\alpha}}=\{i,j\}$. Notice that $|v_{s,\alpha}\cdot e_j|= 1$. In fact, we have $|v_{s,\alpha}\cdot e_j|\geq 1$, and if the last inequality is strict, replacing $v_{s,\alpha}$ with $\pi_{e_i}(v_{s,\alpha})$ in $P$ we obtain the same contradiction as before.

On the other hand, since $P$ is irreducible and $E_i(P)=\{(s,\alpha)\}$, the vector $v_{s,\alpha}$ is not isolated. Let us assume that $v_{s,\alpha}$ is final, and in particular  $v_{s,\alpha}\cdot v_{s+1,\alpha}=0$ (the case $v_{s-1,\alpha}\cdot v_{s,\alpha}=0$ can be handled analogously). Let $\ell\in\gbra{1,...,s}$ be the largest number such that the set $S:=\{v_{s-1,\alpha},...,v_{s-\ell,\alpha}\}$ has a connected intersection graph. Then, there exists some $h\in\gbra{1,..., l}$  such that $a_{s-h,\alpha}>2$. In fact: 
\begin{itemize}
\item if $\ell=s$, we can take $h=s=\ell$ since, by \eqref{e:Q_P}, we know $a_0=a_{s-h,\alpha}\geq 3$;  
\item if $\ell<s$ and there is no $h\in\gbra{1,...,\ell}$  such that $a_{s-h,\alpha}>2$, we would have $a_{s-1,\alpha}=...=a_{s-\ell,\alpha}=2$. Then, being $P$  irreducible, there would  exist an element of $V_S$ also belonging to $V_{P\setminus S}$. Since the graph associated to $S$ is a chain of ($-2$)-vectors disconnected from the rest of $P$, we actually have $V_S\subseteq V_{P\setminus S}$, and hence $V_{v_{s,\alpha}}\subseteq V_{P\setminus S}$ too, contradicting $|E_i(P)|=1$. 
\end{itemize}
Now, fix the smallest $h\in\gbra{1,...,\ell}$ such that $a_{s-h,\alpha}>2$. It is easy to check that, for some $k\in\{1,...,n\}$,
$$V_{v_{s-h+1,\alpha}}\cap V_{v_{s-h,\alpha}}=\{e_k\}\s\mbox{and}\s |v_{s-h,\alpha}\cdot e_k|=1.$$
Since $|\bigcup_{i=0}^{h-1}V_{v_{s-i,\alpha}}|=h+1$,  by eliminating the vectors $v_{s,\alpha},v_{s-1,\alpha},...,v_{s-h+1,\alpha}$ and replacing $v_{s-h,\alpha}$ with $\pi_{e_k}(v_{s-h,\alpha})$, we obtain a set of $n-h$ linearly independent vectors contained in the span of $n-(h+1)$ vectors. This contradiction shows that $v_{s,\alpha}$ cannot be final, that is, it must be internal.

The fact that $v_{s,\alpha}$ is not central is guaranteed by the inequality $I(P)<0$, as explained in the following. Indeed, if $v_{s,\alpha}$ is the central vertex, then $\widetilde{P}:=P\setminus \{v_{s,\alpha}\}$ is a set with at least three connected components and moreover, it is a good linear set with $I(\widetilde{P})<-1$, (recall that, by \textbf{(w3)}, $I(P)<-1$ and that the central vertex has square at least $-3$). We claim that $\widetilde{P}$ has no linear bad components. In fact, if $\widetilde{P}$ has a linear bad component $C=\{v_{c_1},...,v_{c_k}\}$, then, since, by \textbf{(w2)}, $P$ has no bad components, there  must be an element of $C$ connected to the central vertex $v_0\in P$, let it be $v_{c_1}$. Notice that $j$ is the only possible element in the intersection $V_{v_{c_1}}\cap V_{v_0}$. By the definition of $C$, the index $j$ belongs to  $V_{v_{c_k}}$ too, and we obtain the contradiction $v_0\cdot v_{c_k}\neq 0$. The fact that $\widetilde{P}$ has no linear bad components, in turn, contradicts the fact that $c(\widetilde{P})\geq 3$. In fact, the hypothesis of \cite[Lemma~4.9]{b:Li2} are fulfilled and we obtain that $c(\widetilde{P})\leq 2$. This proves that $v_{s,\alpha}$ is not central and implies that $j$ belongs to $V_{v_{s-1,\alpha}}\cap V_{v_{s+1,\alpha}}$ and, since $|v_{s,\alpha}\cdot e_j|=1$, we conclude  $|v_{s-1,\alpha}\cdot e_j|=|v_{s+1,\alpha}\cdot e_j|=1$. 
\end{proof}

Given a set $P_n\subseteq\Z^n$ satisfying the working assumptions we show in Lemma \ref{l:p1>0_2} how to contract the leg $L_\alpha$, which contains the vector $v_{s,\alpha}$ satisfying $E_{i}(P)=\{(s,\alpha)\}$, in order to obtain a new good set, $P_{n-1}\subseteq\Z^{n-1}$, that still satisfies the working assumptions. The notation used in Lemma \ref{l:p1>0_2} is the same as the one of Lemma \ref{l:p1>0_1}.

\begin{lem}\label{l:p1>0_2} 
Consider a subset $P_{n}\subseteq\Z^n$ satisfying the working assumptions and put $\lambda := |e_i\cdot v_{s,\alpha}|$. Then we have
$$
\left\{
\begin{array}{c}
a_{s-1,\alpha}=2,\\
a_{s+1,\alpha}>2,
\end{array}
\right.
\s\s\mbox{or}\s\s
\left\{
\begin{array}{c}
a_{s-1,\alpha}>2,\\ 
a_{s+1,\alpha}=2
\end{array}
\right.
$$
Let us assign
$$
v_{p,\alpha}
:=
\left\{
\begin{array}{cl}
v_{s-1,\alpha},& \mbox{ if }a_{s-1,\alpha}=2,\\
v_{s+1,\alpha},& \mbox{ if }a_{s+1,\alpha}=2,
\end{array}
\right.
\s\s
v_{q,\alpha}
:=
\left\{
\begin{array}{cl}
v_{s-1,\alpha},& \mbox{ if }a_{s-1,\alpha}>2,\\
v_{s+1,\alpha},& \mbox{ if }a_{s+1,\alpha}>2.
\end{array}
\right.
$$
Then, if $n_\alpha>2$, the subset of $\Z^{n-1}=\langle e_1,...,e_{i-1},e_{i+1},...,e_n\rangle$
$$P_{n-1}:= \cbra{P_n\setminus\{v_{p,\alpha},v_{s,\alpha},v_{q,\alpha}\}}\cup\{\pi_{e_j}(v_{q,\alpha})\}\cup \{v_{p,\alpha}+(\lambda-v_{s,\alpha}\cdot e_j)e_j\}\subseteq\Z^{n-1},$$ 
where $j$ is the index given by Lemma~\ref{l:p1>0_1}\,$(2)$, satisfies the working assumptions.
Moreover,  we have $$I(P_n)=I(P_{n-1}),\s\s V_{L_\alpha}\cap V_{L_\beta\cup L_\gamma}=\varnothing,\s\s |V_{L_\alpha}|=n_\alpha +1,$$
where $\gbra{\alpha,\beta,\gamma}=\gbra{1,2,3}$.
\end{lem}
\begin{proof}
To begin with, notice that precisely one between $a_{s-1,\alpha}$ and $a_{s+1,\alpha}$ is equal to $2$ (so that the other one is strictly greater than $2$). In fact, if both  $a_{s-1,\alpha}$ and $a_{s+1,\alpha}$ are equal to $2$,  then we have, since $v_{s-1,\alpha}\cdot v_{s+1,\alpha}=0$, that $V_{v_{s-1,\alpha}}=V_{v_{s+1,\alpha}}$ and therefore that $P_n$ is  reducible for $n>3$, contradicting our working assumptions. On the other hand, if $a_{s-1,\alpha},a_{s+1,\alpha}>2$,  erasing  $v_{s,\alpha}$ from $P$ and replacing $v_{s-1,\alpha}$ and $v_{s+1,\alpha}$ respectively with $\pi_{e_j}(v_{s-1,\alpha})$ and $\pi_{e_j}(v_{s+1,\alpha})$, we get a set whose associated incidence matrix is of the form \eqref{e:Q_P}. Therefore, by Remark~\ref{r:li}, its $n-1$ elements are linearly independent, but at the same time they should belong to the span of  $n-2$ vectors, which gives a contradiction. 

Now, notice that the set
$$P_{n-1}:=(P_n\setminus\{v_{p,\alpha},v_{s,\alpha},v_{q,\alpha}\})\cup\{\pi_{e_j}(v_{q,\alpha})\}\cup \{v_{p,\alpha}+(\lambda-v_{s,\alpha}\cdot e_j)e_j\}$$
has an associated intersection matrix $Q_{P_{n-1}}$ of the form \eqref{e:Q_P}. Since $E_j(P_n)=\{(p,\alpha),(s,\alpha),(q,\alpha)\},$ we have that $E_j(P_{n-1})=\{(p,\alpha)\}$ and, since $\pi_{e_j}(v_{q,\alpha})$ and $v_{p,\alpha}+(\lambda-v_{s,\alpha}\cdot e_j)e_j$ are linked to each other, the set $P_{n-1}$ is irreducible. The verification that  $I(P_n)=I(P_{n-1})$ is straightforward and therefore, $P_{n-1}$ satisfies the working assumptions. 

We now show that $V_{L_\alpha}\cap V_{L_\beta\cup L_\gamma}=\varnothing$. Since $P_{n-1}$ still satisfies the working assumptions, by the above procedure we can construct $P_{n-2}$. We iterate this process until we obtain a set $\widetilde{P}$ whose leg $L_\alpha (\widetilde{P})$ has length  $2$. At this point $|V_{L_\alpha (\widetilde{P})}|=3$, and in fact, up to replacing $\widetilde P$ with $\Omega\widetilde P$ where $\Omega\in\Upsilon_{\widetilde P}$, we have $L_\alpha(\widetilde{P})=\{v_{2,\alpha}=e_k+e_j, v_{1,\alpha}=\lambda e_i-e_j\}$. Since $E_i(\widetilde{P})=\{(1,
\alpha)\}$ it is clear that the three elements in $V_{L_\alpha (\widetilde{P})}$ do not belong to the other two legs. In order to conclude, we just need to remark that, along the reduction process, we have always discarded vectors that only appeared in the leg $L_\alpha$. The last assertion, namely $|V_{L_\alpha (P_n)}|=n_\alpha (P_n)+1$ is also immediate from the reduction procedure: at each stage we have $n_\alpha (P_n)=n_\alpha (P_{n-1})+1$ and $|V_{L_\alpha}(P_n)|=|V_{L_\alpha}(P_{n-1})|+1$. Since $|V_{L_\alpha (\widetilde{P})}|=3$ and $n_\alpha (\widetilde{P})=2$ the claim follows.
\end{proof}

In the previous result we have shown how to contract the leg $L_\alpha$ with the vector $e_i$. The following one, in turn, explains how to contract the other two legs of a set satisfying the working assumptions. 

\begin{lem}\label{l:p1>0_3}
Let $P_n\subseteq\Z^n$ be a subset that satisfies the working assumptions and suppose, without loss of generality,  $\alpha=1$. If $n_2+n_3>2$ then,  for some $k\in\{1,...,n\}$ and  $i\in\gbra{2,3}$, the  set $P_{n-1}:=(P_n\setminus\{v_{n_2,2},v_{n_3,3}\})\cup\pi_{e_k}(v_{n_i,i})$
fulfills the working assumptions. Moreover, $I(P_n)=I(P_{n-1})$.
\begin{proof}
Consider the vectors 
$$\hat{v}_0:=-\sum_{j\in V_{L_2\cup L_3}}(v_0\cdot e_j)e_j,\s\s\s \tilde v_0:=v_0-\hat{v}_0.$$
We claim that $|V_{\hat{v}_0}|=1$. In fact, since the legs $L_2$ and $L_3$ are connected to the center of the graph we have $v_0\cdot v_{1,2}=v_0\cdot v_{1,3}=1$ and therefore $|V_{\hat{v}_0}|\geq 1$. If the last inequality were strict, the linear set $S:=L_2\cup L_3\cup\{\hat v_0\}$  would be good and it would consist of $n_2+n_3+1$ (linearly independent) vectors lying in the span of $n_2+n_3$ vectors\footnote{Recall that by Lemma~\ref{l:p1>0_2} we know $V_{L_1}\cap (V_{L_2\cup L_3})=\emptyset$ and $|V_{L_1}|=n_1 +1$.}, which gives a contradiction. Denoting by $j$ the only element in $V_{\hat{v}_0}$, since $\hat{v}_0\cdot v_{1,2}=1$ we have $|\hat{v}_0\cdot e_j|=1$.

The set $S$ defined above fails to be good because $a_{\hat v_0}=1$, and so we consider the following change: we take an auxiliary vector $e_{\aux}$ with $e_{\aux}\cdot e_{\aux}=-1$, we put $\hat{v}_0':=\hat{v}_0+e_{\aux}$ and define the set $S':=(S\setminus\hat{v}_0)\cup\{\hat{v}_0'\}$. This set is a connected good linear subset of $\Z^{n_2+n_3+1}$ and by construction $p_1(S')\geq 1$. Therefore, Lemma~\ref{l:n.e} applies and we obtain $k\in\{1,...,n\}$ and $i\in\{2,3\}$ such that $E_k(S')=\{(n_2,2),(n_3,3)\}$ and the 
set $S'':=(S'\setminus\{v_{n_2,2},v_{n_3,3}\})\cup\pi_{e_k}(v_{n_i,i})$, obtained from $S'$ by contraction,
is a connected good subset of $\Z^{n_2+n_3}$, with $I(S'')=I(S')$.

In order to conclude we just have to notice that the contraction performed on the set $S'$ can be done on $P_n$, obtaining the set $P_{n-1}$ in the statement that fulfills the working assumptions: we have not changed the leg $L_\alpha (P_n)$ and therefore $p_1(P_{n-1})\geq 1$ and $I(P_{n-1})=I(P_n)<-1$.
\end{proof}
\end{lem}

\begin{rem}\label{r:p1>0_coef.}
Notice that, the set $S'$ in the proof of Lemma~\ref{l:p1>0_3} is obtained (see also Lemma~\ref{l:n.e}\,$(4)$) by final $(-2)$-vector expansions. It follows that the legs $L_2$ and $L_3$ of a set $P\subset\Z^n$ satisfying the working assumptions are complementary legs and therefore, Lemma~\ref{l:cl} holds. In particular we know that, 
\begin{itemize}
\item[$(1)$]for every $i\in\{1,...,n\}$ and every $(t,\beta)\in J$, we have $|v_{t,\beta}\cdot e_i|\leq 1$; 
\item[$(2)$]the set $S_{n_1+1}:=L_1\cup\{\tilde v_0\}\subseteq\Z^{n_1+1}$ is a good linear set. Furthermore, by \textbf{(w4)} we have that $p_1(S_{n_1+1})>0$ and hence, $S_{n_1+1}$ fulfills the hypothesis of Lemma~\ref{l:n.e}.
\end{itemize}
\end{rem}

We now have all the elements that we need to prove Propositions \ref{p:set_5} and \ref{p:p1>0_final}.

\begin{proof}[Proof of Proposition~\ref{p:set_5}.]
Condition $p_1(P)\geq 1$ implies that there exist $i\in\gbra{1,...,5}$ and $(s,\alpha)\in J$ such that  $E_i(P)=\{(s,\alpha)\}$. By Lemma \ref{l:p1>0_1},  $v_{s,\alpha}$ is internal not central, and $V_{v_{s,\alpha}}=\{i,j\}$. More precisely, $v_{s,\alpha}=\lambda e_i\pm e_j$ where $\lambda\in\Z$.

Let us fix, without loss of generality, $\alpha=1$. From $I(P)<-1$ we obtain that $\sum_i a_i\leq 13$ and therefore $a_{s,1}=2$, forcing $\lambda=\pm 1$ (recall that $a_0\geq 3$). Then, up to replacing $P$ with $\Omega P$ where $\Omega\in\Upsilon_{P}$, we can assume $v_{s,1}=e_1-e_2$.

Since $n=5$ and $v_{s,1}$ is internal and different from $v_0$ we deduce $s=1$, $n_1=2$ and $n_2=n_3=1$. From Lemma \ref{l:p1>0_2} we have that $v_{2,1}$ or $v_0$ has square $-2$ and, since $v_0\cdot v_0\leq -3$, we conclude $v_{2,1}=e_2-e_3$ (up to replacing $P$ with $\Omega P$ where $\Omega\in\Upsilon_{P}$).

It is now straightforward to verify that, up to replacing $P$ with $\Omega P$ where $\Omega\in\Upsilon_{P}$,  $v_0=e_2+e_3-e_4$ $v_{1,2}=-e_4+e_5$ and $v_{1,3}=-e_4-e_5$. A direct computation gives $I(P)=-4$.
\end{proof}

\begin{proof}[Proof of Proposition~\ref{p:p1>0_final}.]
If $n=5$ the proposition follows from Proposition \ref{p:set_5}. Now, let us consider the case $n=n_1+n_2+n_3+1>5$.
By Remark~\ref{r:p1>0_coef.} we know that in $P_n$ there are two complementary legs, let them be $L_2$ and $L_3$. By Lemma~\ref{l:cl}\,$(1)$, there exists $j\in\{1,...,n\}$ such that the central vertex satisfies $v_0=\tilde v_0\pm e_j$, where $j\in V_{L_{2}\cup L_{3}}$ and $V_{\tilde v_{0}}\subseteq V_{L_{1}}$. If $n_2+n_3>2$, by Lemma~\ref{l:p1>0_3} there exists a subset $P_{n-1}$, which fulfills the working assumptions and $I(P_{n-1})=I(P_n)$. Observe that, as long as $n_2(P_{n-1})+n_3(P_{n-1})>2$, we can apply Lemma~\ref{l:p1>0_3}. Thus, applying Lemma~\ref{l:p1>0_3} $n_2(P_n)+n_3(P_n)-2$ times we obtain a sequence of contractions $P_n\searrow\cdots\searrow\ P_{n_1+3}$, where each set satisfies the working assumptions and $I(P_n)=\cdots=I(P_{n_1+3})$. Notice that, by construction, it also holds $c(P_n)=\cdots=c(P_{n_1+3})$.

On the other hand, by Remark~\ref{r:p1>0_coef.}\,$(2)$, the set $S_{n_1+1}:=L_1\cup\{\tilde v_0\}\subseteq\Z^{n_1+1}$ fulfills the hypothesis of Lemma~\ref{l:n.e}. We label the vectors in $S_{n_1+1}$ with the same indexes that they already have as elements of $P$. If $n_1>2$ there exists, by Lemma~\ref{l:n.e}\,$(3)$, and index  $h\in\{1,...,n\}$ and $(t,1),(r,1)\in\{0,(n_1,1)\}$ such that $E_h(S_{n_1+1})=\{0,(n_1,1)\}$, $a_{t,1}=2$ and $a_{r,1}>2$. If $r=0$ the set
$P_{n_1+2}:=(P_{n_1+3}\setminus\{v_0,v_{n_1,1}\})\cup\{\pi_{e_h}(v_0)\}\subseteq\Z^{n_1+2},$ 
obtained by contraction, satisfies the working assumptions and $(I(P_{n_1+2}),c(P_{n_1+2}))=(I(P_{n_1+3}),c(P_{n_1+3}))$. If $(r,1)=(n_1,1)$ then the set
$$P_{n_1+2}:=(P_{n_1+3}\setminus\{v_0,v_{n_1,1},v_{1,1}\})\cup\{\pi_{e_h}(v_{n_1,1})\}\cup\{v_{1,1}\pm e_j\}$$
is obtained by a contraction (recall that for a set with complementary legs the notion of contraction was extended to this operation). Again $(I(P_{n_1+2}),c(P_{n_1+2}))=(I(P_{n_1+3}),c(P_{n_1+3}))$ and the set $P_{n_1+2}$ satisfies the working assumptions. Starting with $P_{n_1+3}$ we can perform these contractions $n_1-2$, times obtaining the rest of the desired sequence $P_{n_1+3}\searrow\cdots\searrow\ P_5$, in which every set satisfies the working assumptions and it holds $I(P_{n_1+3})=\cdots=I(P_5)$ and $c(P_n)=\cdots=c(P_{n_1+3})$. Since by Proposition~\ref{p:set_5} $(I(P_5),c(P_5))=(-4,1)$ then, for every $k\in\{5,...,n\}$, $I(P_k)=-4$ and the set $P_k$ is standard.
\end{proof}

Notice that, if we impose that the vector $v_{s,\alpha}$ is not the central vertex, then we can relax  assumption \textbf{(w3)} and the non existence of bad components in \textbf{(w2)}. Indeed, the proofs of Lemmas \ref{l:p1>0_2} and \ref{l:p1>0_3} do not use these hypotheses, while in  Lemma \ref{l:p1>0_1} they are used only to guarantee that $v_{s,\alpha}\neq v_0$. Therefore, we readily obtain the following statement.

\begin{lem}\label{l:p1>0_conclusion}
Let $n\geq 5$ and $P_n\subseteq\Z^n$ be a good set with a trivalent vertex such that there exist $i\in\{1,...,n\}$ and $(s,\alpha)\in J$ that satisfy $E_i(P_n)=\{(s,\alpha)\}$. If $v_{s,\alpha}$ is not the central vertex, then
\begin{itemize}
\item $P_n$ has two complementary legs, $L_\beta$,$L_\gamma\neq L_\alpha$, and $V_{L_\beta\cup L_\gamma}\cap V_{L_\alpha}=\emptyset$.
\item There exists $j\in\{1,...,n\}$, $j\neq i$, such that the central vertex $v_0$ can be written $v_0=\tilde v_0+e_j$, where $e_j\in V_{L_\beta\cup L_\gamma}$ and $V_{\tilde v_0}\subseteq V_{L_\alpha}$.
\item The set $L_\alpha\cup\tilde v_0$ is a good linear set that satisfies the assumptions of Lemma~\ref{l:n.e}. 
\hfill\qed
\end{itemize}
\end{lem}

\subsection{Good sets with $p_1=0$}\label{s:p1=0,p2>0}
It follows from Lemma~\ref{l:lisca_1} that if a subset $P\subseteq\Z^n$ of cardinality $n$ satisfies $I(P)<0$ and $p_1(P)=0$, then necessarily $p_2(P)>0$. Having already dealt with the case $p_1(P)>0$, $I(P)<-1$, now we approach the more difficult case of a good subset with $p_1(P)=0$, $p_2(P)>0$ and $I(P)<-1$. The main result of this section is Lemma~\ref{l:p.ind}.

If a good subset $P=\{v_0,v_{1,1},...,v_{n_3,3}\}\subseteq\Z^n$ satisfies $p_2(P)>0$ then, for some $i\in\{1,...,n\}$ and $(s,\alpha),(t,\beta)\in J$, we must have $E_i(P)=\{(s,\alpha),(t,\beta)\}$. There are two possibilities: either $a_{s,\alpha}$ and $a_{t,\beta}$ are both greater than $2$, or at least one of them is equal to $2$. The next lemma deals with the latter possibility (assuming $P$ has neither linear bad components nor three-legged bad components), while the former possibility is considered in Lemma \ref{l:red}.

\begin{lem}\label{l:cases}
Suppose that $n=n_1+n_2+n_3+1\geq 5$, the subset of $\Z^n$, $P=\{v_0,v_{1,1}...,v_{n_3,3}\}$ is good, has neither linear bad components nor three-legged bad components, $p_1(P)=0$, $I(P)<-1$ and there exist $i\in\{1,...,n\}$ and $(s,\alpha),(t,\beta)\in J$ such that $E_i(P)=\{(s,\alpha),(t,\beta)\}$ and $a_{s,\alpha}=2$. Then, one of the following holds:
\begin{itemize}
\item[$(1)$] $v_{s,\alpha}\cdot v_{t,\beta}=0$, $a_{t,\beta}=2$, $\alpha\neq\beta$, $n_\alpha=n_\beta=1$ and $L_\alpha$ and $L_\beta$ are complementary legs. 
\item[$(2)$] $v_{s,\alpha}\cdot v_{t,\beta}=0$, $v_{s,\alpha}$ is internal and $a_{t,\beta}>2$.
\item[$(3)$] $v_{s,\alpha}\cdot v_{t,\beta}=0$, $v_{s,\alpha}$ is not internal, $|V_{v_{t,\beta}}|>2$, the set
$$P':=(P\setminus\{v_{s,\alpha},v_{t,\beta}\})\cup\{\pi_{e_i}(v_{t,\beta})\}\subseteq\Z^{n-1}$$ 
is a good set with no bad components of any type and $I(P')\leq I(P)$.
\item[$(4)$] $v_{s,\alpha}\cdot v_{t,\beta}=1$, $a_{t,\beta}>2$ and the set $P'$ defined in $(3)$ is again good with no bad components of any type and $I(P')\leq I(P)$. Moreover, for some $h\in\{1,...,n\}$ we have $|v_{t,\beta}\cdot e_h|>1$.
\end{itemize}
\end{lem}
\begin{proof}
Since $a_{s,\alpha}=2$, we have $V_{v_{s,\alpha}}=\{i,j\}$ for some $i,j\in\{1,...,n\}$.

\textbf{First case:} \emph{$v_{s,\alpha}\cdot v_{t,\beta}=0$ and $a_{t,\beta}=2$}. In this case $V_{v_{t,\beta}}=\{i,j\}$. Since $P$ is a good set, it is irreducible and being $n>3$ there must exist $v_{r,\gamma}\in P$ with $(r,\gamma)\not\in\{(s,\alpha),(t,\beta)\}$ linked to either $v_{s,\alpha}$ or $v_{t,\beta}$. Since $E_i(P)=\{(s,\alpha),(t,\beta)\}$, necessarily $j\in V_{v_{r,\gamma}}$; hence, $v_{s,\alpha}\cdot v_{r,\gamma}=v_{t,\beta}\cdot v_{r,\gamma}=1$ and we have $|V_{v_{r,\gamma}}|\geq 2$. If $|V_{v_{r,\gamma}}|=2$ it follows that $P$ is reducible. Therefore, $|V_{v_{r,\gamma}}|>2$ and, since $P$ has no bad components, we must have $(r,\gamma)=(0,\gamma)$; hence, $(1)$ holds.

\textbf{Second case:} \emph{$v_{s,\alpha}\cdot v_{t,\beta}=0$ and $a_{t,\beta}>2$}. We have $\{i,j\}\varsubsetneq V_{v_{t,\beta}}$, since $|V_{v_{t,\beta}}|=2$ contradicts the irreducibility of $P$. Thus $|V_{v_{t,\beta}}|\geq 3$ and if this inequality were strict, then the set $P_1:=(P\setminus\{v_{s,\alpha},v_{t,\beta}\})\cup\{\pi_{e_i}\circ\pi_{e_j}(v_{t,\beta})\}\subseteq\Z^{n-1}$
would be a good set consisting of $n-1$ linearly independent vectors in the span of the $n-2$ vectors $\langle e_1,...,\hat e_i,...,\hat e_j,...,e_n\rangle$. Therefore, there exists $h\in\{1,...,n\}$, $h\neq i,j$, such that $V_{v_{t,\beta}}=\{i,j,h\}$. 

If $v_{s,\alpha}$ is internal in $P$, then $(2)$ holds. Suppose now that $v_{s,\alpha}$ is isolated and notice that in this case we have $v_{t,\beta}\neq v_0$. In fact, suppose by contradiction that $v_{t,\beta}=v_0$. Then, since $E_i(P)=E_j(P)=\{(s,\alpha),0\}$, the set $P_2:=P\setminus\{v_{s,\alpha}, v_0\}\subseteq\Z^{n-2}$
is a good linear set: its incidence matrix obviously satisfies \eqref{e:Q_P_linear} and it is irreducible  because $|V_{v_0}|=3$. Hence, every $v\in P$ different from $v_{s,\alpha}$ and linked to $v_0=v_{t,\beta}$ is also linked to $v_{1,\delta}$ for each $\delta\in\{1,2,3\}$. Moreover, the set $P_2$ has no linear bad components, since $b(P)=0$ and $P$ has no three-legged bad components. Moreover, it satisfies $I(P_2)< 0$, but at the same time, it must have at least three connected components, which contradicts \cite[Lemma~4.9]{b:Li2}. Now, since $v_{t,\beta}\neq v_0$ and $|V_{v_{t,\beta}}|=3$ we can define
$P':=(P\setminus\{v_{s,\alpha},v_{t,\beta}\})\cup\{\pi_{e_i}(v_{t,\beta})\}$
whose incidence matrix $Q_{P'}$ is of the form \eqref{e:Q_P} or \eqref{e:Q_P_linear}. Since every $v\in P$ linked to $v_{s,\alpha}$ must satisfy $v\cdot e_j\neq 0$, $P'$ is irreducible. A straightforward computation gives $I(P')\leq I(P)$ and, since $a_{s,\alpha}=2$, by Lemma \ref{l:b_3}, $P'$ has no bad components of any type and therefore $(3)$ holds.

Now, we  analyze the case in which $v_{s,\alpha}$ is a final vector and let us suppose $v_{s-1,\alpha}\cdot v_{s,\alpha}=1$ and $v_{s,\alpha}\cdot v_{s+1,\alpha}=0$ (the other case is analogous). To begin with, let us further assume that $v_{t,\beta}\neq v_0$. In this case we can define the set
$P':=(P\setminus\{v_{s,\alpha},v_{t,\beta}\})\cup\{\pi_{e_i}(v_{t,\beta})\}$
that has all the desired properties, just like when we considered $v_{s,\alpha}$ isolated, and so $(3)$ holds. We are now left with the study under the assumption $v_{t,\beta}=v_0$. This time, since $v_0\cdot v_{1,\beta}=1$ for some $\beta\neq\alpha$ and $V_{v_0}=\{i,j,h\}$, we have $|v_0\cdot e_h|=1$. We analyze separately the following two possibilities:

\vspace{1.5mm}
\noindent $(i)$ $s=2$ and $V_{v_{2,\alpha}}\cap V_{v_{1,\alpha}}\cap V_{v_0}=j$. If furthermore $a_0>3$, we necessarily have that $|v_0\cdot e_i|=|v_0\cdot e_j|\geq 2$, and then the set
$P':=(P\setminus\{v_{2,\alpha},v_0\})\cup\{\pi_{e_i}(v_0)\}$
has, just like before, all the desired properties and so $(3)$ holds. Suppose now that $a_0=3$, then $P$ has a three-legged bad component, as explained in what follows. If $V_{v_{1,\alpha}}=\{j,k\}$ then $E_k(P)=\{(1,\alpha)\}$ and hence $p_1(P)>0$, which contradicts our assumptions. Therefore, $|V_{v_{1,\alpha}}|>2$ and so $a_{1,\alpha}>2$.  Consider the set
$P_3:=(P\setminus\{v_{2,\alpha},v_{1,\alpha},v_0\})\cup\{\pi_{e_i}(v_0)\}\subseteq\Z^{n-1}.$ The associated incidence matrix $Q_{P_3}$ is of the form \eqref{e:Q_P_linear} ($P_3$ is a linear set). Thus, $p_1(P_3)=1$ (since $E_j(P_3)=\{(\pi_{e_i}(v_0))\}$) and a direct calculation gives $I(P_3)<-1$. On the one hand, if $P_3$ is an irreducible set it is good, and therefore Lemma \ref{l:n.e} applies and $(4)$ gives that $P_3$ is obtained by final $(-2)$-vector expansions and therefore $P$ has a three-legged bad component. On the other hand, if $P_3$ is reducible, there exist $P_3^1$ and $P_3^2$ such that $P_3=P_3^1\cup P_3^2$ with $V_{P_3^1}\cap V_{P_3^2}=\emptyset$. Let us suppose, without loss of generality, that $\pi_{e_i}(v_0)\in P_3^1$. Since $P$ is irreducible we have that $P_3^1$ is a linear irreducible set, and therefore it is good. As before, we apply Lemma \ref{l:n.e} to the set $P_3^1$ and by $(4)$ we get that it is obtained by final $(-2)$-vector expansions. Hence, we conclude again that $P$ has a three-legged bad component, contradicting the assumption of the lemma.

\vspace{1.5mm}
\noindent $(ii)$ $E_j(P)=\{(s-1,\alpha),(s,\alpha),0\}$ and $v_{s-1,\alpha}\cdot v_0=0$. In this case we have that $|V_{v_{s-1,\alpha}}|>2$. In fact, since $v_{s-1,\alpha}\cdot v_0=0$ and $j\in V_{v_{s-1,\alpha}}\cap V_{v_0}$, we need $h\in V_{v_{s-1,\alpha}}$. Moreover, since for some $\beta\neq\alpha$ we have $h\in V_{v_{1,\beta}}$ and $v_{s-1,\alpha}\cdot v_{1,\beta}=0$, there must be some $k\in\{1,...,n\}$, $k\neq j,h$ such that $k\in V_{v_{s-1,\alpha}}$. Therefore, $|V_{v_{s-1,\alpha}}|\geq 3$ and the set
$P_4:=(P\setminus\{v_{s,\alpha},v_{s-1,\alpha},v_0\})\cup\{\pi_{e_j}(v_{s-1,\alpha})\}$
is a linear good set with $I(P_4)<-1$ and at least three connected components. Let us check that the set $P_4$ has no linear bad components. Indeed, arguing as in the proof of Lemma \ref{l:p1>0_1}, if $P_4$ has a linear bad component $C=\{v_{c_1},...,v_{c_k}\}$, then, since $P$ has no bad components, there must be an element of $C$ connected to the central vertex $v_0\in P$, let it be $v_{c_1}$. Notice that $h$ is the only possible element in the intersection $V_{v_{c_1}}\cap V_{v_0}$. By the definition of $C$, $h$ belongs to  $V_{v_{c_k}}$ and we obtain the contradiction $v_0\cdot v_{c_k}\neq 0$. Thus, the assumptions of \cite[Lemma~4.9]{b:Li2} are fulfilled and and it follows $c(P_4)\leq 2$, a contradiction.

 \textbf{Third case:} \emph{$v_{s,\alpha}\cdot v_{t,\beta}=1$ and $a_{t,\beta}=2$}. In this case $\beta=\alpha$ and by symmetry we can assume $t=s+1$ and so $V_{v_{s+1,\alpha}}=\{i,k\}$ for some $k\neq j$. Notice that, since $p_1(P)=0$, the vector $v_{s,\alpha}$ must be internal. Then, arguing as in the proof of Lemma \ref{l:p1>0_1}, one achieves a contradiction using the fact that $E_i(P)=\{(s,\alpha),(s+1,\alpha)\}$ by considering the largest $\ell,m\geq 1$ such that the set $\{v_{s-m,\alpha},...,v_{s+\ell,\alpha}\}$ has a connected plumbing graph. As in Lemma \ref{l:p1>0_1} it is easy to check that there exists a smallest $r\in\{s-m,...,s-1\}$ such that $a_{r,\alpha}>2$ or a smallest $q\in\{s+2,...,s+\ell\}$ such that $a_{q,\alpha}>2$. Suppose that only the latter happens (the other cases can be handled similarly). Then, for some $t\in\{1,...,n\}$, it holds that $V_{v_{q,\alpha}}\cap V_{v_{q-1,\alpha}}=\{e_t\}$ and $|v_{q,\alpha}\cdot e_t|=1.$
Since $$\left|\bigcup_{i=s-m}^{q-1}V_{v_{i,\alpha}}\right|=q+m-s+1,$$  by eliminating all the vectors $v_{s-m,\alpha},v_{s-m+1,\alpha},...,v_{q-1,\alpha}$ and replacing $v_{q,\alpha}$ with $\pi_{e_t}(v_{q,\alpha})$, we obtain the same contradiction, via rank counting, as in the proof of Lemma \ref{l:p1>0_1}. 

\textbf{Fourth case:} \emph{$v_{s,\alpha}\cdot v_{t,\beta}=1$ and $a_{t,\beta}>2$}. Again in this case we have 
$\beta=\alpha$ and by symmetry we may assume $t=s-1$. Since $p_1(P)=0$, the vector $v_{s,\alpha}$ is not final and this implies $j\in V_{v_{s-1,\alpha}}$. In fact, if $j\not\in V_{v_{s-1,\alpha}}$ we get a contradiction as in the previous case by considering the biggest $l\geq 0$ such that $\{v_{s,\alpha},...,v_{s+l,\alpha}\}$ has connected plumbing graph.

If $V_{v_{s-1,\alpha}}=\{i,j\}$ then, since $P$ is irreducible, $v_{s,\alpha}$ is not final. Therefore, $E_j(P)=\{(s-1,\alpha),(s,\alpha),(s+1,\alpha)\}$ and since $v_{s-1,\alpha}\cdot v_{s+1,\alpha}=0$ there exists some $k\in\{1,...,n\}$, $k\neq i,j$, such that $k\in V_{v_{s-1,\alpha}}\cap V_{v_{s+1,\alpha}}$, which contradicts the assumption $V_{v_{s-1,\alpha}}=\{i,j\}$.

Therefore we conclude that $\{i,j\}\varsubsetneq V_{v_{s-1,\alpha}}$. Observe that, since $V_{v_{s-1,\alpha}}\cap V_{v_{s,\alpha}}=\{i,j\}$, $a_{s,\alpha}=2$ and $v_{s-1,\alpha}\cdot v_{s,\alpha}=1$, we necessarily have $|v_{s-1,\alpha}\cdot e_j|>1$ or $|v_{s-1,\alpha}\cdot e_i|>1$. Therefore, the $h\in\{1,...,n\}$ in case $(4)$ in the statement of the lemma is $i$ or $j$. Hence, the set
$P':=(P\setminus\{v_{s,\alpha},v_{s-1,\alpha}\})\cup \{\pi_{e_i}(v_{s-1,\alpha})\}$
is a good set, because its incidence matrix has the form \eqref{e:Q_P} or \eqref{e:Q_P_linear} and it is irreducible (observe that if $v\in P$ is linked to $v_{s,\alpha}$ then $j\in V_v$). It is clear that $I(P')\leq I(P)$ and, since $a_{s,\alpha}=2$, by Lemma \ref{l:b_3} we know it has no bad components of any type and therefore $(4)$ holds.

Notice that in this case we do not need to worry about $v_{s-1,\alpha}$ being the central vertex. In fact, if this was the case, then $v_{s,\alpha}$ would be $v_{1,\alpha}$ and hence $P'$ would be a linear good set, for which $\pi_{e_i}(v_{s-1,\alpha})\cdot\pi_{e_i}(v_{s-1,\alpha})\leq -2$ is guaranteed from the above discussion.
\end{proof}

\begin{lem}\label{l:red}
Suppose that $n=n_1+n_2+n_3+1\geq 5$, the subset of $\Z^n$, $P=\{v_0,...,v_{n_3,3}\}$ is good, has neither linear nor three-legged bad components, $I(P)<-1$ and there exist $i\in\{1,...,n\}$ and $(s,\alpha),(t,\beta)\in J$ such that $E_i(P)=\{(s,\alpha),(t,\beta)\}$ with $a_{s,\alpha}, a_{t,\beta}>2$. Then, one of the following holds:
\begin{itemize}
\item[$(1)$] Up to interchanging the roles of $v_{s,\alpha}$ and $v_{t,\beta}$ in the following definition, the set 
$P':=(P\setminus\{v_{s,\alpha},v_{t,\beta}\})\cup\{\pi_{e_i}(v_{t,\beta})\}\subseteq\Z^{n-1},$
satisfies $I(P')\leq I(P)-1$, $c(P')\leq c(P)+1$ and it is a good set with no bad components of any type.
\item[$(2)$] There exist $k\neq i$ and $s'\in\{s-1,s+1\}$ such that
\begin{itemize}
\item[$(a)$] $E_k(P)=\{(s,\alpha),(s',\alpha)\}$
\item[$(b)$] $v_{s',\alpha}\cdot v_{s,\alpha}=1$
\item[$(c)$] $a_{s',\alpha}=2$
\end{itemize}
\item[$(3)$] The vector $v_{s,\alpha}$ is the central vector, i.e.\ $v_{s,\alpha}=v_0$, and in $P$ there are two complementary legs.
\end{itemize}
\end{lem}
\begin{proof}
If the set $P'$ of $(1)$ is good, since $a_{s,\alpha}>2$ it follows that $I(P')\leq I(P)-1$. By Lemma~\ref{l:b_3} we know that $P'$ can have at most one linear bad component, or one three-legged bad component. If $P'$ had a bad component then, by Lemma~\ref{l:bad}, interchanging the roles of $v_{s,\alpha}$ and $v_{t,\beta}$ in the definition of $P'$ we obtain a good set without bad components. By definition of contraction, the inequality $c(P')\leq c(P)+1$ holds if $v_{s,\alpha}$ is not the central vertex. In fact, $v_{s,\alpha}$ can never be the central vertex, for $v_{s,\alpha}=v_0$ would imply that $c(P')\geq 3$, which contradicts \cite[Lemma~4.9]{b:Li2}. Hence, $(1)$ holds.

Now suppose that the set 
$P'((s,\alpha),(t,\beta)):=(P\setminus\{v_{s,\alpha},v_{t,\beta}\})\cup\{\pi_{e_i}(v_{t,\beta})\}$
is not good because $\pi_{e_i}(v_{t,\beta})^2=-1$. In this case 
$v_{t,\beta}=\lambda e_i\pm e_j$ with $|\lambda|>1$ and, since $i\in V_{v_{s,\alpha}}$ and $v_{s,\alpha}\cdot v_{t,\beta}\in\{0,1\}$, we have $j\in V_{v_{s,\alpha}}$. If $\pi_{e_i}(v_{s,\alpha})^2=-1$ then $v_{s,\alpha}=\mu e_i\pm e_j$ with $|\mu|>1$, but this is impossible since it implies $|v_{s,\alpha}\cdot v_{t,\beta}|=|-\lambda\mu\pm 1|\geq 3$. Therefore, $\pi_{e_i}(v_{s,\alpha})^2\leq -2$.  If this last inequality is strict the incidence matrix of $P'((t,\beta),(s,\alpha))$ is of the form \eqref{e:Q_P} or \eqref{e:Q_P_linear}. The same holds if $\pi_{e_i}(v_{s,\alpha})^2= -2$ and $v_{s,\alpha}$ is not the central vertex. In the remaining case, namely $\pi_{e_i}(v_{s,\alpha})^2=-2$ and $v_{s,\alpha}=v_0$, since $E_i(P)=\{(s,\alpha),(t,\beta)\}$, we have necessarily $v_{s,\alpha}\cdot v_{t,\beta}=1$ and hence $P'((t,\beta),(s,\alpha))$ is a linear good set whose incidence matrix satisfies \eqref{e:Q_P_linear}. In all three cases, $P'((t,\beta),(s,\alpha))$ is irreducible because, since $\{i,j\}\subseteq V_{v_{s,\alpha}}$, there is no vector linked to $v_{t,\beta}$ but unlinked from $v_{s,\alpha}$. Moreover, $P'((t,\beta),(s,\alpha))$ has no bad components, since $|V_{v_{t,\beta}}|=2$ and Lemma~\ref{l:b_3}\,$(1)$. Therefore, after replacing $((s,\alpha),(t,\beta))$ with $((t,\beta),(s,\alpha))$, assertion $(1)$ holds.

We deal now with the last possibility, namely either the set $P'((s,\alpha),(t,\beta))$ or the set $P'((t,\beta),(s,\alpha))$ has an incidence matrix of the form \eqref{e:Q_P} or \eqref{e:Q_P_linear} but it is not good because it is reducible. We assume without loss of generality that $P'((s,\alpha),(t,\beta))$ has this property. In this case we can write $P'((s,\alpha),(t,\beta))=P_1'\cup P'_2$, where $P'_2$ is a maximal irreducible subset of $P'((s,\alpha),(t,\beta))$ that contains $\pi_{e_i}(v_{t,\beta})$ and $P'_1=P'((s,\alpha),(t,\beta))\setminus P'_2$. Define $P_i\subseteq P\setminus\{v_{s,\alpha}\}$, $i=1,2$, to be the preimage of $P'_i$ under the surjective map $\pi_{e_i}:P\setminus\{v_{s,\alpha}\}\rightarrow P'((s,\alpha),(t,\beta))$. The decomposition $P\setminus\{v_{s,\alpha}\}=P_1\cup P_2$ shows that $P\setminus\{v_{s,\alpha}\}$ is reducible since clearly $V_{P_1}\cap V_{P_2}=\emptyset$. Since $P$ is irreducible while $P\setminus\{v_{s,\alpha}\}$ is reducible, there exists a vector $v_{r,\gamma}\in P_1$ that is linked to $v_{s,\alpha}$, therefore $|V_{v_{s,\alpha}}\cap V_{P_1}|\geq 1$.
 
Now, if $|V_{v_{s,\alpha}}\cap V_{P_1}|>1$ and $v_{s,\alpha}=v_0$ then $\{i\}\varsubsetneq V_{v_0}\cap V_{P_2}$. In fact, since $V_{P_1}\cap V_{P_2}=\emptyset$, the equality $\{i\}=V_{v_0}\cap V_{P_2}$ would imply that the linear set $P'':=(P\setminus\{v_{s,\alpha},v_{t,\beta}\})\cup\pi_{e_i}(v_{s,\alpha})\cup \pi_{e_i}(v_{t,\beta})$ has an associated intersection matrix of the form \eqref{e:Q_P_linear}. Therefore, by Remark \ref{r:li}, $P''$ should be a set of $n$ linearly independent vectors, but they are contained in the span of $\{e_1,...,e_n\}\setminus\{e_i\}$. Hence, we have $|V_{v_0}\cap V_{P_2}|\geq 2$, which implies $|V_{v_{s,\alpha}}|\geq 4$. It follows that in this case also $P'((t,\beta),(s,\alpha))$ has an incidence matrix of the form \eqref{e:Q_P} or \eqref{e:Q_P_linear} and moreover, since $|V_{v_0}\cap V_{P_2}|\geq 2$, we have that it is irreducible. Therefore, if  $P'((t,\beta),(s,\alpha))$ has no bad components, then $(1)$ holds. In turn, if it has a bad component $C$, then $(3)$ holds. In fact, by Lemma~\ref{l:b_3}\,$(3)$, we know that $\pi_{e_i}(v_{s,\alpha})\in C$ and therefore $C$ is a three legged bad component, which, by definition, has two complementary legs.

On the other hand, if $|V_{v_{s,\alpha}}\cap V_{P_1}|>2$ or $|V_{v_{s,\alpha}}\cap V_{P_1}|>1$ and $v_{s,\alpha}\neq v_0$, then we could replace $v_{s,\alpha}$ with
$$\tilde v_{s,\alpha}:=-\sum_{h\in V_{v_{s,\alpha}}\cap V_{P_1}}(v_{s,\alpha}\cdot e_h)e_h$$
and $v_{t,\beta}$ with $\pi_{e_i}(v_{t,\beta})$. The $n$ vectors resulting from these replacements have an associated incidence matrix of the form \eqref{e:Q_P} or \eqref{e:Q_P_linear}. In fact, for every $v\in P_1$ we have $v\cdot\tilde v_{s,\alpha}=v\cdot v_{s,\alpha}\in\{0,1\}$ and for every $v\in P_2$, since $V_{P_1}\cap V_{P_2}=\emptyset$, we have $v\cdot\tilde v_{s,\alpha}=0$. Therefore, by Remark \ref{r:li} the $n$ vectors are linearly independent, but they are contained in the span of $\{e_1,...,e_n\}\setminus\{e_i\}$, giving a contradiction. Thus, there exists $k\in\{1,...,n\}$ such that $V_{v_{s,\alpha}}\cap V_{P_1}=\{k\}$.  

If $v_{r,\gamma}\in P_1$ is any vector linked to $v_{s,\alpha}$, then $V_{v_{s,\alpha}}\cap V_{v_{r,\gamma}}=\{k\}$ and so $v_{r,\gamma}\cdot v_{s,\alpha}=1$. This implies that $\{(s-1,\alpha),(s,\alpha)(s+1,\alpha)\}\supseteq E_k(P)$, if $s\neq 0$, or that $\{0,(1,1),(1,2),(1,3)\}\supseteq E_k(P)$, if $s=0$ (i.e.\ if $v_{s,\alpha}$ is the central vertex). We will analyze separately the following three cases.

\noindent $(i)$ \textit{If $v_{s,\alpha}$ is final}, for any $v_{r,\gamma}\in P_1$ as above, we have $(r,\gamma)\in\{(s-1,\alpha),(s+1,\alpha)\}$ and then $E_k(P)=\{(s-1,\alpha),(s,\alpha)\}$ or $E_k(P)=\{(s,\alpha),(s+1,\alpha)\}$. By symmetry, we can assume that the first case occurs. If $a_{s-1,\alpha}>2$ we can eliminate $v_{s,\alpha}$, replace $v_{s-1,\alpha}$ with $\pi_{e_k}(v_{s-1,\alpha})$ and $v_{t,\beta}$ with $\pi_{e_i}(v_{t,\beta})$. Notice that, since $|e_k\cdot v_{s-1,\alpha}|=1$, we have $|\pi_{e_k}(v_{s-1,\alpha})|^2\geq 2$  and it also holds that, if $s-1=0$, the set obtained from $P$ by these replacements is a linear set. Thus, we have obtained a set of $n-1$ vectors whose associated incidence matrix has the form \eqref{e:Q_P} or \eqref{e:Q_P_linear} and applying Remark~\ref{r:li} once again we get a contradiction because these vectors belong to the span of the $n-2$ vectors $\{e_1,...,e_n\}\setminus\{e_i,e_k\}$. Therefore, we conclude that $a_{s-1,\alpha}=2$ and hence $(2)$ holds.

\noindent $(ii)$ \textit{If $v_{s,\alpha}$ is internal and it is not the central vertex}, we have $v_{s-1,\alpha}\cdot v_{s,\alpha}=v_{s,\alpha}\cdot v_{s+1,\alpha}=1$. Hence $V_{v_{s,\alpha}}\cap V_{P_1}=\{k\}$ and so $E_k(P)=\{(s-1,\alpha),(s,\alpha),(s+1,\alpha)\}$ and $V_{v_{s-1,\alpha}}\cap V_{v_{s,\alpha}}=V_{v_{s,\alpha}}\cap V_{v_{s+1,\alpha}}=\{k\}$. Let us define the vector $v_{s,\alpha}':=-(v_{s,\alpha}\cdot e_k)e_k+e_i$ and consider the set $\widetilde P:=P_1\cup \{v_{s,\alpha}'\}$ which satisfies $p_1(\widetilde P)>0$ (recall that $E_i(\widetilde P)=\{(s,\alpha)\}$). Moreover, $\widetilde P$ is good: since $P=P_1\cup P_2\cup\{v_{s,\alpha}\}$ is irreducible and $V_{P_1}\cap V_{P_2}=\emptyset$, then $\widetilde P$ is irreducible  and, since $v_{s,\alpha}'$ is not the central vertex in $\widetilde P$, $Q_{\widetilde P}$ is of the form \eqref{e:Q_P} or \eqref{e:Q_P_linear}. On the one hand, if $\widetilde P$ is a linear set, since $p_1(\widetilde P)>0$, by Lemma \ref{l:n.e} it is connected and it is obtained by final $(-2)$-vector expansions; therefore $(P_1\cup \{v_{s,\alpha}\})\subseteq P$ is a linear bad component (recall the assumption $a_{s,\alpha}>2$), contradicting the hypothesis of the lemma. On the other hand, if $\widetilde P$ is a set with a trivalent vertex different from $v_{s,\alpha}'$ then, since $p_1(\widetilde P)>0$, we know, by Lemma \ref{l:p1>0_conclusion}, that $\widetilde P$ is connected with two complementary legs. Thus $P$ has a three-legged bad component which again contradicts the assumption of the lemma.

\noindent $(iii)$ \textit{If $v_{s,\alpha}$ is the central vertex $v_0$}, then $E_k(P)\subseteq\{0,(1,1),(1,2),(1,3)\}$. If $|E_k(P)|=2$ then we argue as in $(i)$. In fact, let us suppose without loss of generality that $E_k(P)=\{0,(1,1)\}$. First observe that, if $a_{1,1}>2$ we can eliminate $v_0$, replace $v_{1,1}$ with $\pi_{e_k}(v_{1,1})$ and $v_{t,\beta}$ with $\pi_{e_i}(v_{t,\beta})$, obtaining $n-1$ linearly independent vectors in the span of $n-2$ basis vectors. This contradiction forces $a_{1,1}=2$ and hence $(2)$ holds. On the other hand, if $|E_k(P)|=3$, we argue as in $(ii)$. Indeed, let us consider without loss of generality $E_k(P)=\{0,(1,1),(1,2)\}$. We can again define the vector $v'_0:=-(v_0\cdot e_k)e_k+e_i$ and the good set $\widetilde P:=P_1\cup \{v'_0\}$. Since $k\not\in V_{v_{1,3}}$, we have $v_{1,3}\in P_2$ and therefore $\widetilde P$ is a linear set. Since $p_1(\widetilde P)>0$, by Lemma \ref{l:n.e} it is connected and it is obtained by final $(-2)$-vector expansions. Therefore $L_1$ and $L_2$ are complementary legs in $P$ and $(3)$ holds. Finally we deal with $E_k(P)=\{0,(1,1),(1,2),(1,3)\}$. In this case, the set $P'((s,\alpha),(t,\beta))$ has no bad components. Indeed, by Lemma~ \ref{l:b_3}\,$(3)$ the only possible bad component would be one of the connected components that appear when we erase $v_0$, but we know $E_k(P)=\{0,(1,1),(1,2),(1,3)\}$ and this contradicts the definition of bad component. The good linear set $P'_1$ has at least three connected components and $b(P'_1)=0$. Since $I(P)<-1$, a straightforward computation gives 
$$I(P'_1)+I(P'_2)=I(P)-|e_i\cdot v_{t,\beta}|^2-a_0+3<2-|e_i\cdot v_{t,\beta}|^2-a_0\leq -2.$$ 
If $I(P'_1)<0$ we get a contradiction with \cite[Lemma~4.9]{b:Li2} and therefore $I(P'_2)\leq -3$. The set $P'_2$ is a linear good set with no bad components and therefore, by \cite[Corollary~5.4]{b:Li}, it holds $I(P'_2)=-3$. This forces $|e_i\cdot v_{t,\beta}|=1$ and $a_0=3$. In order to conclude we will show that $P_2\subseteq P$ is a linear bad component, contradicting the assumption of the lemma. In fact, we write $V_{v_0}=\{i,k,\ell\}$ and observe that, since $k\in V_{P'_1}$ and $V_{P'_1}\cap V_{P'_2}=\emptyset$, we have $|E_\ell(P'_2)|=1$. Moreover, by \cite[Corollary~3.5]{b:Li} we know that $P'_2$ is connected and it is obtained from the set in \cite[Lemma~2.4\,$(1)$]{b:Li} by final $(-2)$-vector expansions. The sets $P_2$ and $P'_2$ only differ in one vector, namely $v_{t,\beta}\in P_2$ which becomes the vector $\pi_{e_i}(v_{t,\beta})\in P'_2$. Therefore, $P_2$ is, as claimed, a linear bad component.
\end{proof}

\begin{rem}
The argument used in $(ii)$ in the proof of Lemma~\ref{l:red} can be used to fix a wrong claim in the proof of  \cite[Lemma~4.3]{b:Li}. More precisely, the vector $v_{s+1}$, considered in \cite[last line, p.\ 445, Lemma~4.3]{b:Li} is internal in $S''_l$, contrary to what is claimed in \cite{b:Li}, and hence the claimed contradiction is not achieved. The argument used in $(ii)$ provides the desired contradiction, since in the case \virg{$v_s$ is not final} \cite[line 20, p.\ 445, Lemma~4.3]{b:Li}, it implies that the set $S$ of the statement of \cite[Lemma~4.3]{b:Li} has a bad component, contrary to the assumptions.
\end{rem}

\begin{lem}\label{l:p.ind}
Suppose that $n=n_1+n_2+n_3+1\geq 5$, the subset of $\Z^n$, $P=\{v_0,...,v_{n_3,3}\}$ is good, has neither linear bad components nor three-legged bad components, $I(P)<-1$, $p_1(P)=0$ and $p_2(P)>0$. Furthermore, suppose that $P$ has no complementary legs. Then, there exist $i\in\{1,...,n\}$ and $(s,\alpha),(t,\beta)\in J$ such that the set
$P':=P\setminus\{v_{s,\alpha},v_{t,\beta}\}\cup\{\pi_{e_i}(v_{s,\alpha})\}\subseteq\Z^{n-1}$
is good, $I(P')\leq I(P)$ and $P'$ has no bad components of any type.
\end{lem}
\begin{proof}
Since $p_2(P)>0$, there exist $i\in\{1,...,n\}$ and $(s,\alpha),(t,\beta)\in J$ such that $E_i(P)=\{(s,\alpha)(t,\beta)\}$. If $a_{s,\alpha}>2$ and $a_{t,\beta}>2$, the hypotheses of Lemma \ref{l:red} are satisfied. Therefore, since has no complementary legs, the conclusions of Lemma \ref{l:red}\,$(1)$ or Lemma \ref{l:red}\,$(2)$ hold. In the first case the lemma follows immediately, in the second case, Lemma \ref{l:cases} applies and $(4)$ holds giving the desired result.

From now on we assume that for each $i\in\{1,...,n\}$ and each $(s,\alpha),(t,\beta)\in J$ such that $E_i(P)=\{(s,\alpha)(t,\beta)\}$, we have either $a_{s,\alpha}=2$ or $a_{t,\beta}=2$. Since $p_1(P)=0$, $p_2(P)>0$ and $I(P)<-1$, by\footnote{In \cite{b:Li} the statement of the Lemma refers to linear good sets, nevertheless the proof does not use the linearity of the set and it holds word for word in our case.} \cite[Lemma~4.4]{b:Li}, we have $E_i(P)=\{(s,\alpha),(t,\beta)\}$ and either $v_{s,\alpha}$ is not internal or $v_{s,\alpha}\cdot v_{t,\beta}=1$, for at least one choice of $i,(s,\alpha),(t,\beta)$. Now, since we are assuming that $P$ has no bad components of any type, Lemma \ref{l:cases} applies and, since $P$ has no complementary legs, either the conclusion of Lemma \ref{l:cases}\,$(3)$ or the conclusion of Lemma \ref{l:cases}\,$(4)$ holds. In both cases the lemma is proved.
\end{proof}

\subsection{The general case}\label{s:gc}
In this part we use all the work done in Sections \ref{s:p1>0} and \ref{s:p1=0,p2>0} in order to prove that any good set $P$ with no bad components of any type and $I(P)<-1$ has $I(P)\in\{-2,-3,-4\}$ and is obtained by a sequence of expansions from a subset of $\Z^k$, where $k\in\{3,5\}$. The main result is Proposition~\ref{p:clave}, however, a considerable part of its proof is developed  before in Proposition~\ref{p:coef.}. The proof of Proposition~\ref{p:clave} works by induction and the initial case is studied in the following lemma.

\begin{lem}\label{l:base}
Let $P\subseteq\Z^5=\langle e_1,e_2,e_3,e_4,e_5\rangle$ be a good set with a trivalent vertex and  $I(P)<-1$. Then, $P$ is, up to replacing $P$ with $\Omega P$ where $\Omega\in\Upsilon_{P}$, one of the following graphs:
\begin{flushleft}
\frag[n]{ii}{$(1)$}
\frag[s]{a-b}{$e_2-e_3$}
\frag[s]{c-a}{$e_1-e_2$}
\frag[s]{a+c-d}{$e_2+e_3+e_4$}
\frag[s]{d+e}{$-e_4+e_5$}
\frag[s]{d-e}{$-e_4-e_5$}
\includegraphics[scale=0.7]{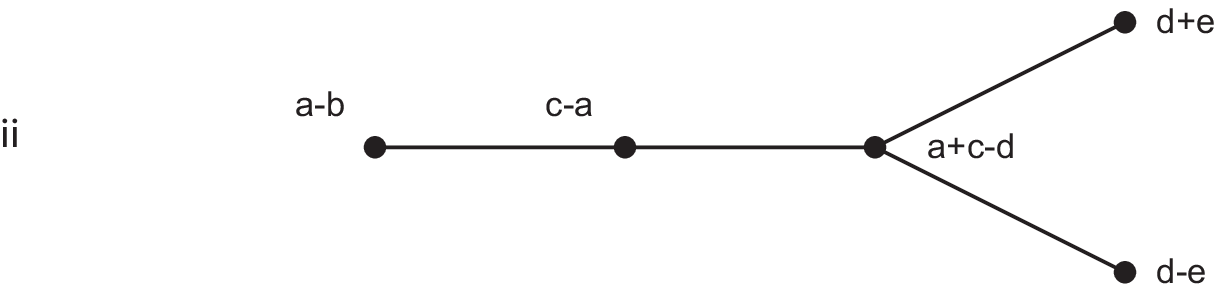}$\s\s\s$
\end{flushleft}
\begin{flushleft}
\frag[n]{ii}{$(2)$}
\frag[s]{a-b}{$e_2+e_1-e_3$}
\frag[s]{c-a}{$-e_2+e_1$}
\frag[s]{a+c-d}{$e_2+e_3+e_4$}
\frag[s]{d+e}{$-e_4+e_5$}
\frag[s]{d-e}{$-e_4-e_5$}
\includegraphics[scale=0.7]{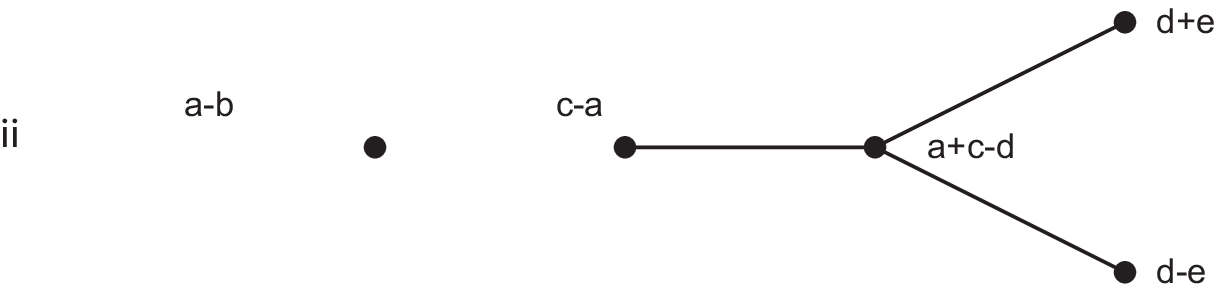}$\s\s\s$
\end{flushleft}
\begin{flushleft}
\frag[n]{ii}{$(3)$}
\frag[s]{a-b}{$e_1+e_2$}
\frag[s]{c-a}{$e_1-e_2+e_3$}
\frag[s]{a+c-d}{$-e_1+e_2+e_3+e_4$}
\frag[s]{d+e}{$-e_4+e_5$}
\frag[s]{d-e}{$-e_4-e_5$}\includegraphics[scale=0.7]{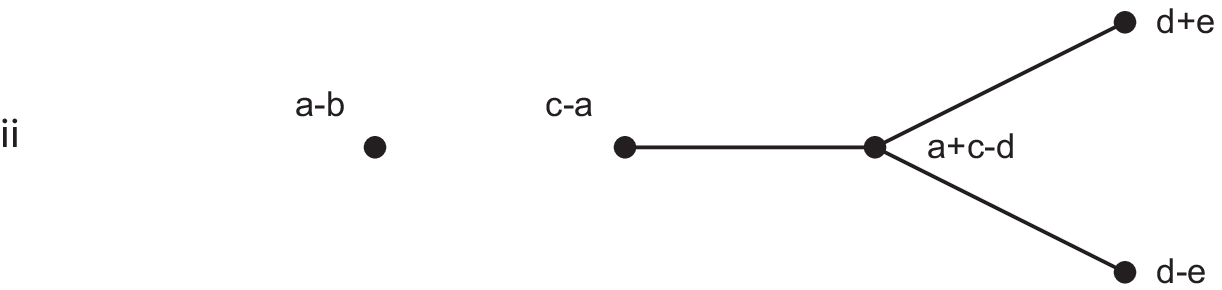}$\s\s\s$
\end{flushleft}
Moreover, $I(P)\in\{-4,-3,-2\}$.
\end{lem}
\begin{proof}
If $p_1(P)>0$, Proposition \ref{p:set_5} applies and $(1)$ holds. Therefore, from now on we assume that $p_1(P)=0$. Since $P\subseteq\Z^5$ is good with a trivalent vertex, by definition we have that $a_0\geq 3$. Let us consider the following possibilities.
\textit{First, $a_0=3$}. In this case, up to replacing $P$ with $\Omega P$ where $\Omega\in\Upsilon_{P}$, we have $v_0=e_2+e_3+e_4$. Since $I(P)<-1$, then $a_{s,\alpha}\leq 4$ for every $v_{s,\alpha}\in P$. Let us suppose first that there exists $v_{s,\alpha}\in P$ such that $a_{s,\alpha}=4$. Then, since $I(P)<-1$, we have $a_{t,\beta}=2$ for every $(t,\beta)\neq (s,\alpha),(0,\beta)$. On the one hand, if $v_{s,\alpha}\cdot v_0=1$, then $|V_{v_{s,\alpha}}\cap V_{v_0}|=3$ and, up to replacing $P$ with $\Omega P$ where $\Omega\in\Upsilon_{P}$, we have $v_{s,\alpha}=e_2-e_3-e_4+e_5$. Let us write, without loss of generality, $v_{s,\alpha}=v_{1,1}$. The other two vectors attached to $v_0$, namely $v_{1,2}$ and $v_{1,3}$, have square $-2$ and this implies $|V_{v_{1,2}}\cap V_{v_0}|=|V_{v_{1,3}}\cap V_{v_0}|=1$. Taking into account that $v_{1,1}\cdot v_{1,2}=v_{1,1}\cdot v_{1,3}=v_{1,2}\cdot v_{1,3}=0$, it is immediate to check that this configuration is impossible. On the other hand, if $v_{s,\alpha}\cdot v_0=0$, then $|V_{v_{s,\alpha}}\cap V_{v_0}|=2$ and, up to replacing $P$ with $\Omega P$ where $\Omega\in\Upsilon_{P}$, we have $v_{s,\alpha}=e_2-e_3+e_1+e_5$. This time we write $v_{s,\alpha}=v_{2,1}$. Then, the ($-2$)-vector $v_{1,1}$ satisfies $v_{1,1}\cdot v_{2,1}=0$. In fact, since $v_{1,1}\cdot v_0=1$, then $|V_{v_{1,1}}\cap V_{v_0}|=1$ and we cannot have at the same time $|V_{v_{1,1}}\cap V_{v_{2,1}}|=1$. Therefore, the three ($-2$)-vectors $v_{1,\beta}$, where $\beta=1,2,3$, must satisfy $|V_{v_{1,\beta}}\cap V_{v_0}|=1$ and $|V_{v_{1,\beta}}\cap V_{v_{2,1}}|=2$, which is not possible. Therefore, we have proved that if $a_0=3$, then $a_{s,\alpha}\leq 3$ for every $v_{s,\alpha}\in P$.

Next, since $I(P)<-1$ there are at least two vectors in $P$ with square equal to $-2$, and necessarily one of them is attached to the central vertex. Hence, we may write, without loss of generality and up to replacing $P$ with $\Omega P$ where $\Omega\in\Upsilon_{P}$, $v_{1,2}=-e_4+e_5$. Let us call $v_{2,\alpha}$ the only vector in $P$ that satisfies $v_0\cdot v_{2,\alpha}=0$. It is not possible to have $a_{2,\alpha}=2$. In fact, if $a_{2,\alpha}=2$, we have either $|V_{v_{2,\alpha}}\cap V_{v_0}|=2$ or $|V_{v_{2,\alpha}}\cap V_{v_0}|=0$. The last possibility yields, up to replacing $P$ with $\Omega P$ where $\Omega\in\Upsilon_{P}$, $v_{2,\alpha}=e_1-e_5$ and $\alpha=2$. Since $p_1(P)=0$, then $1\in V_{v_{1,1}}$ or $1\in V_{v_{1,3}}$. Let us assume that the first case occurs (the other possibility is analogous). Then, up to replacing $P$ with $\Omega P$ where $\Omega\in\Upsilon_{P}$, we have $v_{1,1}=-e_4-e_5-e_1$. Again since $p_1(P)=0$, we must have $\{2,3\}\subseteq V_{v_{1,3}}$, which contradicts $v_{1,3}\cdot v_0=1$. It remains to analyze the case $|V_{v_{2,\alpha}}\cap V_{v_0}|=2$. This time, if $4\in V_{v_{2,\alpha}}$, then, up to replacing $P$ with $\Omega P$ where $\Omega\in\Upsilon_{P}$, we have $v_{2,\alpha}=e_4-e_3$, $\alpha=2$ and $v_{1,2}\cdot v_{2,2}=1$. Since $p_1(P)=0$ we get $5\in V_{v_{1,1}}$ or $5\in V_{v_{1,3}}$. Observe that none of these possibilities gives a good subset $P\subseteq\Z^5$. On the other hand, if $4\not\in V_{v_{2,\alpha}}$, then, up to replacing $P$ with $\Omega P$ where $\Omega\in\Upsilon_{P}$, we have $v_{2,\alpha}=e_2-e_3$. Since $p_1(P)=0$, then $1\in V_{v_{1,1}}\cap V_{v_{1,3}}$ and we do not obtain a good set $P\subseteq\Z^5$. Therefore we conclude $a_{2,\alpha}=3$ and then, since $v_{2,\alpha}\cdot v_0=0$, we have $|V_{v_{2,\alpha}}\cap V_{v_0}|=2$. Hence, up to replacing $P$ with $\Omega P$ where $\Omega\in\Upsilon_{P}$, we have $v_{2,\alpha}=e_2+e_1-e_3$. Now, using the fact that $p_1(P)=0$ we are left with only one possibility, up to replacing $P$ with $\Omega P$ where $\Omega\in\Upsilon_{P}$, for the two remaining vectors. Namely  $v_{1,1}=-e_2+e_1$ and $v_{1,3}=-e_4-e_5$ and hence $(2)$ holds.

\textit{Second, $a_0=4$}. Since the arguments in this case follow closely the ones used in the previous one, we will omit most of the details. Up to replacing $P$ with $\Omega P$ where $\Omega\in\Upsilon_{P}$, we have $v_0=-e_1+e_2+e_3+e_4$. Since $I(P)<-1$ we have $a_{s,\alpha}\leq 3$ for every $(s,\alpha)\neq (0,\alpha)$. If $a_{1,\alpha}=2$ for every $\alpha=1,2,3$, then we must have $|V_{v_{1,\alpha}}\cap V_{v_0}|=1$ for every $\alpha=1,2,3$. Observe that this fact is incompatible with $v_{1,1}\cdot v_{1,2}=v_{1,1}\cdot v_{1,3}=v_{1,2}\cdot v_{1,3}=0$. Therefore, without loss of generality and up to replacing $P$ with $\Omega P$ where $\Omega\in\Upsilon_{P}$, we can write $v_{1,1}=e_1-e_2+e_3$. Since $I(P)<-1$, the three remaining vectors must have square equal to $-2$, and there is only one possibility, up to replacing $P$ with $\Omega P$ where $\Omega\in\Upsilon_{P}$, that yields a good set $P\subseteq\Z^5$. Namely, $v_{1,2}=-e_4+e_5$, $v_{1,3}=-e_4-e_5$ and $v_{2,1}={e_1+e_2}$. Hence, $(3)$ holds.

\textit{Third, $a_0\geq 5$}. Since $I(P)<-1$ and $n=5$, we have $a_0\leq 5$ and therefore we need only to consider the case where $a_0=5$ and all the other vectors have square equal to $-2$. These conditions are incompatible with the fact that $P\subseteq\Z^5$ is a good set. In fact, in this case for every $v\in P$ it holds $|v\cdot v_0|\neq 1$. 
\end{proof}

\begin{prop}\label{p:coef.}
Let $n\geq 5$ and $P\subseteq\Z^n$ be a good set with a trivalent vertex, no bad components of any type and such that $I(P)<-1$. Then, the following hold:
\begin{itemize}
\item[$(1)$] There exists a sequence of contractions $P_n:=P\searrow P_{n-1}\searrow\cdots\searrow P_k$, where $k\in\{3,5\}$, such that for all $i\in\{k,k+1,...,n\}$ the set $P_i\subseteq\Z^i$ is good with no bad components of any type and, moreover, $I(P_{i+1})\geq I(P_i)$ for $i=k,k+1,...,n-1$.
\item[$(2)$] $I(P)\in\{-4,-3,-2\}$.
\item[$(3)$] For every $j\in\{1,...,n\}$ and every $(s,\alpha)\in J$, we have $|v_{s,\alpha}\cdot e_j|\leq 1$. 
\end{itemize}
\end{prop}

\begin{proof}
We argue by induction on $n\geq 5$. If $n=5$ the whole statement is a straightforward consequence of Lemma~\ref{l:base}. Therefore, from now on we  assume $n>5$. Since $I(P)<-1$, by Lemma~\ref{l:lisca_1} inequality~\eqref{e:p_i} holds and therefore, either $p_1(P)>0$ or $p_2 (P)>0$. If $p_1(P)>0$ then $P$ satisfies the working assumptions of Section~\ref{s:p1>0} and then, $(1)$ and $(2)$ follow from Proposition~\ref{p:p1>0_final} (recall that by Lemma~\ref{l:st.sin.bad} a standard set has no bad components), while $(3)$ holds by Remark~\ref{r:p1>0_coef.}\,$(1)$. Thus, we can assume that $p_1(P)=0$ and $p_2(P)>0$ and so Lemma~\ref{l:cases} or Lemma~\ref{l:red} apply. If either Lemma~\ref{l:cases}\,$(1)$ or Lemma~\ref{l:red}\,$(3)$ holds, then in $P$ there are two complementary legs, the assumptions of Lemma~\ref{l:cl} are fulfilled and $(1)$ and $(3)$ follow. Suppose now that in $P$ there are no complementary legs. If such is the case, by Lemma~\ref{l:p.ind} there is a contraction of $P$ that gives a good set $P'\subset\Z^{n-1}$ with no bad components of any type and such that $-1>I(P)\geq I(P')$.

Next, starting with $P_n:=P$ we shall define a decreasing sequence of contractions of good sets without bad components of any type $P_n\searrow P_{n-1}\searrow\cdots\searrow P_k$ where $P_i\subseteq\Z^i$ for every $i\in\{k,k+1,...,n\}$, $k\in\{3,5\}$ and $I(P_{i+1})\geq I(P_i)$ for every $i=k,k+1,...,n-1$. In this way, we shall obtain $(1)$. We define $P_{n-1}:=P'$ and continue the sequence as follows:

If $P_{n-1}$ is a linear set and $n-1\geq 3$, then the assumptions of \cite[Corollary~5.4]{b:Li} are satisfied and we get the contractions $P_{n-1}\searrow P_{n-2}\searrow\cdots\searrow P_3$. All the sets involved are good with no bad components and moreover, it holds $I(P_{i+1})\geq I(P_i)$ for $i=3,4,...,n-1$.

If $P_{n-1}\subseteq\Z^{n-1}$ has a trivalent vertex then, in order to define $P_{n-2}$ and the rest of the sequence we must take into account:

\noindent $(i)$ If $n-1=5$ we stop and the sequence finishes with $P_{n-1}$, which by Lemma~\ref{l:base} satisfies $I(P_{n-1})\in\{-4,-3,-2\}$. 

\noindent $(ii)$ If $n-1\geq 6$ and $p_1(P_{n-1})>0$ then, by Proposition~\ref{p:p1>0_final}, $P_{n-1}$ is a standard set, $I(P_{n-1})=-4$ and we can finish the sequence with the contractions of standard sets $P_{n-1}\searrow P_{n-2}\searrow\cdots\searrow P_5$. Thus, we have a sequence of good sets with no bad components of any type (recall that by Lemma~\ref{l:st.sin.bad} a standard set has no bad components) and such that $I(P_{i+1})\geq I(P_i)$ for $i=5,6,...,n-1$.

\noindent $(iii)$ If $n-1\geq 6$ and $p_1(P_{n-1})=0$ then, by Lemma~\ref{l:lisca_1}, which applies because $I(P_{n-1})<-1<0$, we have $p_2(P_{n-1})>0$. Now, we distinguish two situations. 
	\begin{itemize}
	\item If $P_{n-1}$ has two complementary legs then, since it has no bad components of any type, by
	Lemma~\ref{l:cl}\,$(4)$ we can finish the sequence with the contractions $P_{n-1}\searrow
	P_{n-2}\searrow\cdots\searrow P_5$. In this way we have built the desired sequence of contractions of good sets
	without bad components of any type and such that $I(P_{i+1})\geq I(P_i)$ for $i=n-1,n-2,...,5$.
	\item If $P_{n-1}$ has no complementary legs then, since it has no bad components of any type, by
	Lemma~\ref{l:p.ind} there is a good subset with no bad components $P_{n-2}\subseteq\Z^{n-2}$ such
	that $I(P_{n-1})\geq I(P_{n-2})$. In order to define the set $P_{n-3}\subseteq\Z^{n-3}$ and the rest of the
	sequence we make $P_{n-2}$ play the role of $P_{n-1}$ in the above argument.
	\end{itemize}

Observe that, in the sequence of good sets without bad components $P_n\searrow P_{n-1}\searrow\cdots\searrow P_k$ there are two possibilities for $k$. If all the sets $P_i$ have a trivalent vertex, then $k=5$ and thus, by Lemma~\ref{l:base}, $I(P_5)\geq -4$. If, on the contrary, there is a contraction $P_{i+1}\searrow P_i$ such that $P_i$ is a linear set, then  $k=3$ and thus, by \cite[Corollary~5.4]{b:Li}, $I(P_3)\geq -3$. In either case, the following inequality holds,
\begin{align}\label{e:1}
\sum_{i=k}^{n-1}(I(P_{i+1})-I(P_i))=-I(P_k)+I(P_n)\leq 4-2=2.
\end{align}
Notice that, by construction, the sets in the sequence satisfy $I(P_{i+1})\geq I(P_i)$ for each $i=k,k+1,...,n-1$. Therefore, on the one hand, since by assumption $I(P=P_n)<-1$ and we have shown that $I(P_k)\geq -4$, we have $(2)$, i.e.\  $I(P)\in\{-4,-3,-2\}$. On the other hand, $I(P_{i+1})-I(P_i)\geq 0$ for every $i$, and inequality \eqref{e:1} implies $0\leq I(P_{i+1})-I(P_i)\leq 2$ for every $i$, and in particular $I(P_n)-I(P_{n-1})\leq 2$. 

Recall that the set $P_{n-1}$ was built from $P_n=P$ using Lemma~\ref{l:p.ind} which gives a good set
$P'=P_{n-1}=(P\setminus\{v_{s,\alpha},v_{t,\beta}\})\cup\{\pi_{e_h}(v_{s,\alpha})\}\subseteq\Z^{n-1},$
for some $h\in\{1,...,n\}$ and $(s,\alpha),(t,\beta)\in J$. Therefore, by a simple calculation one easily sees that the inequality $I(P)-I(P')\leq 2$ is equivalent to
$$a_{s,\alpha}+|v_{t,\beta}\cdot e_h|^2\leq 5.$$
Since the set $P$ is good, we have $a_{s,\alpha}\geq 2$ and then $|v_{t,\beta}\cdot e_h|^2=1$. Therefore, $a_{s,\alpha}\leq 4$ and since $|V_{v_{s,\alpha}}|\geq 2$, we conclude that $|v_{s,\alpha}\cdot e_j|\leq 1$ for every $j\in\{1,...,n\}$, which readily implies $(3)$.
\end{proof}

The following Lemma~\ref{l:cl2} is an important consequence of Proposition~\ref{p:coef.}\,$(2)$. It allows us to divide three-legged good subsets $P\subseteq\Z^n$ with $I(P)<-1$ into two subclasses closed under contraction, namely those with complementary legs and those without them. In fact, the following lemma guarantees that complementary legs are \virg{invariant} under contractions: we already know that they can be contracted to length-one complementary legs and we are about to prove that if in a set $P$ there are no complementary legs they will not appear after a contraction. 

\begin{lem}\label{l:cl2}
Let $P\subseteq\Z^n$ be a good subset with no bad components of any type and $I(P)<-1$. Furthermore, suppose that $P$ has no complementary legs. Then, for every contraction $P\searrow P'$ with $P'$ good with no bad components of any type, we have that $P'$ has no complementary legs.
\end{lem}
\begin{proof}
First of all, note that, in any contraction $P\searrow P'$, the vertices of $P'$ inherit naturally the indexes from $P$. Now, suppose by contradiction that the good set
$$P'=(P\setminus\{v_{s,\alpha},v_{t,\beta}\})\cup\{\pi_{e_i}(v_{t,\beta})\}\subseteq\Z^{n-1}$$ 
obtained by a contraction has two complementary legs $L'_2$ and $L'_3$. Let us call $L_2$ and $L_3$ the corresponding legs in $P$ and analogously, $L_1$ and $L'_1$ the third leg respectively in $P$ and in $P'$. Throughout the proof the reader must keep in mind that, by Proposition~\ref{p:coef.}\,$(3)$, for every $k\in\{1,...,n\}$ and every $(r,\gamma)\in J(P)$, we have $|v_{r,\gamma}\cdot e_k|\leq 1$ and that, by Lemma~\ref{l:cl}\,$(1)$, there exists $j\in\{1,...,n\}$ such that the central vertex $v_0\in P'$ is equal to $\tilde v_0\pm e_j$, where $j\in V_{L'_2\cup L'_3}$ and $V_{\tilde v_0}\subseteq V_{L'_1}$. Moreover, it holds $V_{L'_1}\cap V_{L'_2\cup L'_3}=\emptyset$ and by definition of complementary legs we have $|V_{L'_2\cup L'_3}|=|L'_2\cup L'_3|$. 

If $v_{s,\alpha}=v_0$, then $P'$ is a linear set, and the statement follows trivially. In case the vector $v_{s,\alpha}$ is isolated, we will consider, without loss of generality, that it belongs to $L_1$. By definition of complementary legs, $L'_2$ and $L'_3$ are connected and therefore, if $v_{s,\alpha}\not\in L_1$, it must be final in $L_2$ or $L_3$. By symmetry we may suppose that the first case occurs.

We start dealing with the possibility $v_{s,\alpha}$ final in $L_2$ and $v_{t,\beta}\not\in L_2\cup L_3$. It follows that there exists a vector $v\in L_2$ such that $v\cdot v_{s,\alpha}=1$ and therefore $|V_{v_{s,\alpha}}\cap V_{L'_2}|\geq 1$. By definition of complementary legs and since $v_{t,\beta}\not\in L_2\cup L_3$, it holds that $|V_{v_{s,\alpha}}\cap V_{L'_2\cup L'_3}|\geq 2$. Consider now an auxiliary vector $e_{\aux}$ with $e_{\aux}\cdot e_{\aux}=-1$ and define 
$$\bar v_{s,\alpha}:=-\sum_{k\in V_{L'_2\cup L'_3}}(v_{s,\alpha}\cdot e_k)e_k\s\s\mbox{and}\s\s\bar v_0:=\pm e_j+e_{\aux}.$$
Notice that the set $S:=\{\bar v_0\}\cup\{\bar v_{s,\alpha}\}\cup L'_2\cup L'_3$ is a standard linear set which by construction consists of $|L'_2\cup L'_3|+2$ vectors in the span of $|L'_2\cup L'_3|+1$ basis vectors, contradicting Lemma~\ref{l:ind}.

We deal now with the case $v_{s,\alpha}$ final in $L_2$ and $v_{t,\beta}\in L_2\cup L_3$. This time we have that $L_1\cup\{\tilde v_0\}=L'_1\cup\{\tilde v_0\}$ is a linear good set (see Lemma~\ref{l:cl}\,$(2)$) which implies that $V_{v_{s,\alpha}}\cap V_{L_1}=\emptyset$. In fact, if $|V_{v_{s,\alpha}}\cap V_{L_1}|\geq 1$ then, since $v_{s,\alpha}$ is orthogonal to $L_1$, we must have $|V_{v_{s,\alpha}}\cap V_{L_1}|\geq 2$. If this were the case, we could define the vector
$$\hat v_{s,\alpha}:=-\sum_{k\in V_{L_1}}(v_{s,\alpha}\cdot e_k)e_k$$
and consider the good set $L_1\cup\{\tilde v_0\}\cup\{\hat v_{s,\alpha}\}$ whose $|L_1|+2$ vectors belong to the span of $|L_1|+1$ basis vectors, contradicting Lemma~\ref{l:ind}. Hence, $V_{v_{s,\alpha}}\cap V_{L_1}=\emptyset$. It follows that $\widetilde S:=\{\bar v_0\}\cup L_2\cup L_3\subseteq\Z^{n_2+n_3+1}$ is a linear standard set, where $\bar v_0$ is defined as in the preceding case. The set $\widetilde S$ satisfies by construction $p_1(\widetilde S)>0$ and therefore, by Lemma~\ref{l:n.e}\,$(4)$, it is obtained by final $(-2)$-vector expansions. Thus,  in $P$ the legs $L_2$ and $L_3$ are complementary, contradicting the assumption of the lemma.

We are left with the last possibility, namely $v_{s,\alpha}\in L_1$. This implies, since in $P$ there are no complementary legs while in $P'$ there are, that $v_{t,\beta}\in L_2\cup L_3$. From the equality $v_{s,\alpha}\cdot v_{t,\beta}=0$, we deduce that $|V_{v_{s,\alpha}}\cap V_{L_2\cup L_3}|\geq 2$. Let us consider the vector 
$$\tilde v_{s,\alpha}:=-\sum_{k\in V_{L_2\cup L_3}}(v_{s,\alpha}\cdot e_k)e_k$$ 
and the set $\bar S:=\{\bar v_0\}\cup L_2\cup L_3\cup \{\tilde v_{s,\alpha}\}$, where $\bar v_0$ is defined above. Note that in $\bar S$ the vertex  $\tilde v_{s,\alpha}$ is isolated. The set $\bar S$ is a good subset of $\Z^{n_2+n_3+2}$. In fact, its incidence matrix has the form \eqref{e:Q_P_linear} and it is irreducible, since $\{\bar v_0\}\cup L_2\cup L_3$ is connected and $\tilde v_{s,\alpha}$ is linked to $v_{t,\beta}\in L_2$. By construction, it holds $p_1(\bar S)>0$ and therefore by Lemma~\ref{l:n.e}, the set $\bar S$ is standard and thus connected. This contradiction finishes the proof.
\end{proof}

The following Proposition~\ref{p:clave} shows that good subsets with no bad components of any type, possibly disconnected intersection graphs and sufficiently negative quantity $I(P)$ can be contracted to subsets having the same properties. This is the main result of the section and will be used in the proof of Theorem~\ref{t:todo}.

\begin{prop}\label{p:clave}
Suppose that $n\geq 5$, and let $P_n\subseteq\Z^n$ be a good set with a trivalent vertex, no bad components of any type and such that $I(P)<-1$. Then, there exists a sequence of contractions
$P_n\searrow P_{n-1}\searrow\cdots\searrow P_k,$
where either $k=3$ or $k=5$, such that, for each $i=k,k+1,...,n$, the set $P_i$ is good, has no bad components of any type and either
$(I(P_{i+1}),c(P_{i+1}))=(I(P_i),c(P_i))$
or
$I(P_{i+1})-1\geq I(P_i)\ \ \mbox{and}\ \ c(P_{i+1})+1\geq c(P_i).$
Moreover,
\begin{itemize}
\item[$(1)$]If $p_1(P_n)>0$ then $I(P_n)=-4$, $k=5$, $P_n$ is standard and one can choose the above sequence in such a way that $I(P_i)=-4$ and $P_i$ is standard for every $i=5,...,n-1$.
\item[$(2)$]If $I(P_n)+c(P_n)\leq -1$ and $k=5$, then $P_5$ is given, up to replacing $P_5$ with $\Omega P_5$ where $\Omega\in\Upsilon_{P_5}$, by either $(1)$ or $(2)$ in Lemma~\ref{l:base}. Furthermore, if $I(P_n)+c(P_n)<-1$ then necessarily $k=5$ and $P_5$ is given by Lemma~\ref{l:base}\,$(1)$.
\end{itemize}
\end{prop}
\begin{proof}
We take as sequence of contractions, $P_n\searrow P_{n-1}\searrow\cdots\searrow P_k$, of good sets with no bad components of any type, the one given in Proposition~\ref{p:coef.}, which satisfies by construction $I(P_{i+1})\geq I(P_i)$, for every $i=k,k+1,...,n-1$. Given a contraction $P_{i+1}\searrow P_i$ in the sequence, we analyze the numbers $c(P_{i+1})$ and $c(P_i)$. Since $-1>I(P_n)\geq I(P_{i+1})$, then by Lemma~\ref{l:lisca_1}, inequality \eqref{e:p_i} holds and we deduce that either $p_1(P_{i+1})>0$ or $p_2(P_{i+1})>0$. If $p_1(P_{i+1})>0$, then we use Proposition~\ref{p:p1>0_final} to build $P_i$ and the equality $(I(P_{i+1}),c(P_{i+1}))=(I(P_i),c(P_i))$  as well as $(1)$ follow. Moreover, in this case $(2)$ follows from $(1)$ because $P_5$ is given, up to replacing $P_5$ with $\Omega P_5$ where $\Omega\in\Upsilon_{P_5}$, by Lemma~\ref{l:base}\,$(1)$. If $P_{i+1}$ has two complementary legs, then $P_i$ is built using Lemma~\ref{l:cl} and Lemma~\ref{l:cl}\,$(4)$ gives the claim. 

From now on we shall assume that $p_1(P_{i+1})=0$, $p_2(P_{i+1})>0$ and that in $P_{i+1}$ there are no complementary legs. Under these assumptions there exist $j\in\{1,...,n\}$ and $(s,\alpha),(t,\beta)\in J$ such that $E_j(P_{i+1})=\{(s,\alpha),(t,\beta)\}$. If $a_{s,\alpha},a_{t,\beta}>2$ the hypothesis of Lemma~\ref{l:red} are satisfied and, since in $P_{i+1}$ there are no complementary legs, either Lemma~\ref{l:red}\,$(1)$ or Lemma~\ref{l:red}\,$(2)$ hold. If Lemma~\ref{l:red}\,$(2)$ holds, then so does Lemma~\ref{l:cases}\,$(4)$. But this is impossible since, by Proposition~\ref{p:coef.}, we have $|v\cdot e_h|\leq 1$ for every $v\in P_{i+1}$ and every $h\in\{1,...,n\}$. On the other hand, if the conclusion of Lemma~\ref{l:red}\,$(1)$ holds then the good set without bad components
$P_i=(P_{i+1}\setminus\{v_{s,\alpha},v_{t,\beta}\})\cup\{\pi_{e_j}(v_{t,\beta})\}$
clearly satisfies $I(P_{i+1})-1\geq I(P_i)$ and $c(P_{i+1})+1\geq c(P_i)$. 

Notice that there is only one possibility left, namely that for each $j\in\{1,...,n\}$ and each $(s,\alpha),(t,\beta)\in J$ such that $E_j(P_{i+1})=\{(s,\alpha),(t,\beta)\}$ we have either $a_{s,\alpha}=2$ or $a_{s,\beta}=2$. By \cite[Lemma~4.4]{b:Li}, for at least one choice of $j$, $(s,\alpha)$, $(t,\beta)$ we have $a_{s,\alpha}=2$ and either $v_{s,\alpha}$ is not internal or $v_{s,\alpha}\cdot v_{t,\beta}=1$. Therefore, since $P_{i+1}$ has no bad components and no complementary legs, the conclusion of either Lemma~\ref{l:cases}\,$(3)$ or Lemma~\ref{l:cases}\,$(4)$ holds. But, as we pointed out above, Lemma~\ref{l:cases}\,$(4)$ contradicts Proposition~\ref{p:coef.}\,$(3)$, therefore Lemma~\ref{l:cases}\,$(3)$ must hold. Thus, since $|v_{t,\beta}\cdot e_j|=1$ and $v_{s,\alpha}$ is not internal, the good set with no bad components $P_i$ satisfies $(I(P_{i+1}),c(P_{i+1}))=(I(P_i),c(P_i))$.

In order to conclude we must prove $(2)$. If $I(P_n)+c(P_n)\leq -1$ we have the following inequalities
\begin{equation}\label{a}
\begin{split}
I(P_k)+c(P_k)\leq I(P_{k+1})+c(P_{k+1})\leq\cdots\leq I(P_n)+c(P_n)\leq -1.
\end{split}
\end{equation}
If $k=5$ it follows from Lemma~\ref{l:base} that, up to replacing $P_5$ with $\Omega P_5$ where $\Omega\in\Upsilon_{P_5}$, $P_5$ must be either of type Lemma~\ref{l:base}\,$(1)$ or Lemma~\ref{l:base}\,$(2)$. Inequalities \eqref{a} imply that if $I(P_n)+c(P_n)< -1$ then $I(P_k)+c(P_k)\leq -2$. Hence, $P_5$ is given, up to replacing $P_5$ with $\Omega P_5$ where $\Omega\in\Upsilon_{P_5}$, by Lemma~\ref{l:base}\,$(1)$.

Finally, let us assume by contradiction $I(P_n)+c(P_n)<-1$ and $k=3$. It follows from \cite[Lemma~2.4]{b:Li} and inequalities \eqref{a} that, up to replacing $P_3$ with $\Omega P_3$ where $\Omega\in\Upsilon_{P_3}$, $P_3$ is given by \cite[Lemma~2.4(1)]{b:Li}, which satisfies $(I(P_3),c(P_3))=(-3,1)$ and therefore, $(I(P_n),c(P_n))=(-3,1)$. Since $P_n$ has a trivalent vertex while $P_3$ is a linear set, there must be an index $i\in\{3,...,n-1\}$ such that in the contraction $P_{i+1}\searrow P_i$, where
$$P_i=(P_{i+1}\setminus\{v_{1,\alpha},v_{t,\beta}\})\cup\{\pi_{e_j}(v_{t,\beta})\},\s j\in\{1,...,n\},\ \ (1,\alpha),(t,\beta)\in J(P_{i+1}),$$
the set $P_{i+1}$ is three-legged while $P_i$ is a linear set. Since $I(P_i)=-3$, then by \cite[Proposition~6.1]{b:Li} we have $p_1(P_i)=1$ and since $I(P_{i+1})\neq -4$, by Proposition~\ref{p:p1>0_final} we know $p_1(P_{i+1})=0$.  Let us denote $\ell$ the only index in $\{1,...,n\}$ that satisfies $|E_\ell(P_i)|=1$ and clearly $|E_\ell(P_{i+1})|=2$. If $v_{t,\beta}\cdot v_{1,\alpha}=0$ then $E_\ell(P_{i+1})=\{(1,\alpha),(t,\beta)\}$ and this contradicts $v_{1,\alpha}\cdot v_0=1$. On the other hand, if $v_{t,\beta}\cdot v_{1,\alpha}=1$, and therefore $v_{t,\beta}=v_0$, then we have necessarily $|E_\ell(P_{i+1})|=1$ which is a contradiction. Therefore, we conclude that if $I(P_n)+c(P_n)< -1$, the sequence of contractions $P_n\searrow\cdots\searrow P_k$ must end with $k=5$.
\end{proof} 

\begin{rem}\label{r:out}
In the last paragraph of the proof of Proposition~\ref{p:clave} we have proved that we cannot have a contraction of standard sets $P_{i+1}\searrow P_i$ such that $P_{i+1}$ has a trivalent vertex, $I(P_{i+1})\neq -4$ and $P_i$ is a linear set with $I(P_i)=-3$.
\end{rem}

\section{Proof of Theorems~\ref{l:strings1} and \ref{l:2}}\label{s:ss}

In this section we specialize the analysis done in Sections~\ref{s:contractions} to \ref{s:clas} to the case of standard subsets with $I<-1$ and we finally prove Theorems~\ref{l:strings1} and \ref{l:2}.

The graph of a standard set with a trivalent vertex has at least $4$ vertices. The following lemma justifies the fact that in most of the statements of Section~\ref{s:clas} we have assumed $n\geq 5$. 

\begin{lem}\label{l:no4}
In $\Z^4=\abra{e_1,e_2,e_3,e_4}$ there are no standard subsets $P$ with a trivalent vertex and $I(P)<-1$. 
\end{lem}
\begin{proof}
Suppose by contradiction that there exists a standard subset $P\subseteq Z^4$ with a trivalent vertex and such that $I(P)<-1$. Since $P$ is standard, we have $a_0\geq 3$ and the condition $I(P)<-1$ forces $a_0\leq 4$. Assume first that $a_0=4$, which implies, since $I(P)<-1$, that $a_{1,1}=a_{1,2}=a_{1,3}=2$. It follows that, in order to have $v_{1,i}\cdot v_0=1$ for $i\in\{1,2,3\}$, both $|V_{v_{i,1}}\cap V_{v_0}|=1$ and $V_{v_0}=\{1,2,3,4\}$ must hold, a contradiction.   

We deal now with the case $a_0=3$, which together with the condition $I(P)<-1$ implies that at least two among $a_{1,1},a_{1,2}$ and $a_{1,3}$ are equal to $2$. Assume, without loss of generality, $a_{1,1}=a_{1,2}=2$ and $V_{v_0}=\{1,2,3\}$. It follows that $V_{v_{1,1}}=V_{v_{1,2}}$ and that $4\in V_{v_{1,1}}\setminus V_{v_0}$. Since $v_{1,3}\cdot v_{1,1}=v_{1,3}\cdot v_{1,2}=0$, we have $V_{v_{1,3}}\cap V_{v_{1,1}}=\emptyset$ and therefore $|V_{v_{1,3}}|\leq 2$. This last inequality is incompatible with $v_{1,3}\cdot v_0=1$ and $a_{1,3}\geq 2$. Thus, the statement follows.
\end{proof}

\begin{thm}\label{t:todo}
Let $n\geq 5$ and let $P_n\subseteq\Z^n$ be a standard subset such that $I(P_n)<-1$. Then, $I(P_n)\in\{-4,-3,-2\}$ and there is a sequence of contractions
$$P_n\searrow\cdots\searrow P_{k+1}\searrow P_k,$$
where either $k=3$ or $k=5$, such that, for every $i=k,k+1,...,n-1$, the set $P_i$ is standard and $I(P_{i+1})\geq I(P_i)$.
\end{thm}
\begin{proof}
We argue by induction on $n\geq 5$. If $n=5$ the theorem follows immediately from Lemma~\ref{l:base}, so let us assume that $n>5$ and that the statement holds true for sets of cardinality between $5$ and $n-1$. By Proposition~\ref{p:clave} we have $I(P_n)\in\{-4,-3,-2\}$ and there is a sequence of contractions
$P_n\searrow\cdots\searrow P_{k+1}\searrow P_k$, with $k\in\{3,5\}$, such that for every $i=k,...,n-1$, $P_i$ is good, it has no bad components of any type, and we have either
\begin{align}\label{e:igual}
(I(P_{i+1}),c(P_{i+1}))=(I(P_i),c(P_i)),\quad\mbox{or}
\end{align}
\begin{align}\label{e:inec}
I(P_{i+1})-1\geq I(P_i)\ \ \mbox{and}\ \ c(P_{i+1})+1\geq c(P_i).
\end{align}

If $I(P_n)=-4$ or $I(P_n)=-3$, since by assumption $c(P_n)=1$, we have $I(P_n)+c(P_n)<-1$ and hence, by Proposition~\ref{p:clave}\,$(2)$, the sequence of contractions finishes with a set $P_5$ which satisfies $(I(P_5),c(P_5))=(-4,1)$. Then, \eqref{e:igual} and \eqref{e:inec} force $c(P_i)=1$ for every $i=5,...,n-1$ and the statement follows in this case.

Now assume $I(P_n)=-2$. If the sequence of contractions ends with $P_5$, then, by Lemma~\ref{l:base}, we know that in $P_5$ there are two complementary legs and therefore, by Lemma~\ref{l:cl2}, $P_n$ has two complementary legs too. Since $P_n$ is standard we obtain the claim by Lemma~\ref{l:cl}\,$(5)$. Therefore we may assume that the sequence of contractions ends with $P_3$. By \eqref{e:igual} and \eqref{e:inec} we have $c(P_{n-1})\leq 2$. If $c(P_{n-1})=1$ we can apply the induction hypothesis and immediately obtain the result. Therefore, we may assume $(I(P_{n-1}),c(P_{n-1}))\in\{(-3,2),(-4,2)\}$. On the one hand, if $(I(P_{n-1}),c(P_{n-1}))=(-4,2)$ then, by \eqref{e:igual} and \eqref{e:inec}, it follows that $(I(P_3),c(P_3))=(-4,2)$, which contradicts \cite[Lemma~2.4]{b:Li}. On the other hand, if $I(P_{n-1}),c(P_{n-1}))=(-3,2)$ it follows, again by \eqref{e:igual} and \eqref{e:inec}, that 
$(I(P_3),c(P_3))\in\{(-4,2),(-4,1),(-4,3),(-3,2)\}.$
All these possibilities contradict  \cite[Lemma~2.4]{b:Li} and therefore we must have $c(P_{n-1})=1$ and the theorem is proved.
\end{proof}

In order to prove Theorems~\ref{l:strings1} and \ref{l:2} we will apply Theorem~\ref{t:todo} to identify the numbers $\{a_0,...,a_{n_3,3}\}$ corresponding to standard subsets $P\subseteq\Z^n$ with a trivalent vertex and $I(P)<-1$.

\begin{proof}[Proof of Theorem~\ref{l:strings1}]
Since $L_2$ and $L_3$ are complementary legs they have associated strings, $(a_{1,2},...,a_{n_2,2})$ and $(a_{1,3},...,a_{n_3,3})$, related to each other by Riemenschneider's point rule.

By Lemma~\ref{l:cl}\,$(2)$ we know that in the standard set $P$ we can define a linear standard subset as $S_{n_1+1}:=L_1\cup\{\tilde v_0\}\subseteq\Z^{n_1+1}$. Taking into account the definition of complementary legs, a straightforward computation gives $I(P)+1=I(S_{n_1+1})$. Since $P$ is a standard subset satisfying $I(P)<-1$, it follows from Theorem~\ref{t:todo} that $I(P)\in\{-4,-3,-2\}$ and hence $I(S_{n_1+1})\in\{-3,-2,-1\}$. The strings associated to standard linear subsets with $I(P)$ equal to $-3,-2$ or $-1$ are respectively described in Remark~\ref{r:lisca}(\ref{p:-3}), (\ref{p:-2}) and (\ref{p:-1}). Since $\tilde v_0\cdot\tilde v_0=v_0\cdot v_0+1$ a direct case by case analysis gives the lists in the statement.
\end{proof}

\begin{rem}\label{r:ult}
Observe that in Theorem~\ref{l:strings1} we have considered all possible standard subsets $P\subseteq\Z^n$ with a trivalent vertex and $I(P)\in\{-4,-3\}$. In fact, in these cases it holds $I(P)+c(P)<-1$ and therefore, by Proposition~\ref{p:clave}\,$(2)$, we know that there is a sequence of standard sets starting with $P$ and ending with the set in Lemma~\ref{l:base}\,$(1)$. This last set has two complementary legs and hence, by Lemma~\ref{l:cl2}, the set $P$ has, necessarily, two complementary legs too. Notice that, since $P\subseteq\Z^n$ is standard, Lemma~\ref{l:no4} guarantees $n\geq 5$ and hence the assumptions of Proposition~\ref{p:clave} are fulfilled.
\end{rem}

\begin{proof}[Proof of Theorem~\ref{l:2}]
Since $P_n$ is standard, by Lemma~\ref{l:no4}, we have $n\geq 5$. The assumption $P_n$ has no complementary legs implies, by Remark~\ref{r:ult},  that $I(P_n)=-2$. By Theorem~\ref{t:todo} there is a sequence of contractions of standard sets starting with $P_n$ and ending with $P_k$ where $k\in\{3,5\}$. If $k=5$, by Lemma~\ref{l:base}, we know that $P_5$ has two complementary legs. Moreover, Lemma~\ref{l:cl2} implies that in $P_n$ there are two complementary legs too, against the assumption in the statement. Therefore, $k=3$. By Theorem~\ref{t:todo} and \cite[Lemma~2.4]{b:Li} there is a sequence of contractions of standard sets, 
$P_n\searrow\cdots\searrow P_i\searrow P_{i-1}\cdots\searrow P_3$
with $I(P_3)=-3$. Therefore, for some $n\geq i>3$, we have $I(P_n)=\cdots=I(P_i)=-2$ and $I(P_{i-1})=\cdots=I(P_3)=-3$. By Remark~\ref{r:out}, we know that $P_i$ must be a linear set. Moreover, since $P_n$ has a trivalent vertex while $P_3$ is a linear set, there must be some $j\in\{i+1,...,n\}$ such that in the contraction
$P_{j-1}=(P_j\setminus\{v_{s,\alpha},v_{t,\beta}\})\cup\{\pi_{e_h}(v_{t,\beta})\}$
the set $P_j$ has a trivalent vertex and $P_{j-1}$ is a linear set. This implies that $s=1$ and $L_\alpha=\{v_{1,\alpha}\}$. Since $I(P_j)=I(P_{j-1})$, it holds $a_{1,\alpha}=2$. Let $V_{v_{1,\alpha}}=\{h,k\}$ and let $v_0\in P_j$ denote, as usual, the central vertex. Since $v_0\cdot v_{1,\alpha}=1$ then $|\{h,k\}\cap V_{v_0}|=1$. Since $I(P_j)\neq -4$ then, by Proposition~\ref{p:p1>0_final}, we have that $|E_k(P_j)|\geq 2$ and by definition of contraction we have that $E_h(P_j)=\{(1,\alpha),(t,\beta)\}$. These two conditions let us conclude that
$k\in V_{v_0}$ and $h\not\in V_{v_0}$. By definition of contraction we know $|E_h(P_j)|=2$ and therefore we necessarily have $E_k(P_j)=\{0,(1,\alpha),(t,\beta)\}$. Even if the subindexes have no meaning in the linear set $P_{j-1}$, let us call $\tilde v_0$ and $\tilde v_{t,\beta}$ the vectors in $P_{j-1}$ that correspond to the vectors $v_0$ and $v_{t,\beta}$ in $P_j$. 

Since $I(P_n)=\cdots=I(P_j)=-2$ and all these sets are standard, it follows that the set $P_n$ is obtained by final $(-2)$-vector expansions of the set $P_j$. Therefore, in order to establish the list of graphs in the statement we have to point out the standard linear graphs with $I=-2$ which can play the role of $P_{j-1}$. Once we determine $P_{j}$ we perform on it the $-2$ vector expansions. Recall that these expansions produce strings of numbers related to one another by Riemenschneider's point rule.

The possible strings associated to standard linear sets $P_{j-1}$ with $I(P_{j-1})=-2$ are listed in Remark~\ref{r:lisca}\,(\ref{p:-2}) above. Notice that in it, when considering $t=s=0$ for cases $(1)$ and $(2)$ and $k=1, b_{1}=2$ in case $(3)$ it yields the same graph $\bar P\subseteq\Z^{4}$, which, up to replacing $\bar P$ with $\Omega\bar P$ where $\Omega\in\Upsilon_{\bar P}$, is the following.
\begin{center}
\frag{e_1-e_2}{$e_1-e_{2}$}
\frag{e_2-e_3}{$e_2-e_{3}$}
\frag{-e_2-e_1+e_4}{$-e_2-e_{1}+e_{4}$}
\frag{e_2+e_3+e_4}{$e_2+e_{3}+e_{4}$}
\includegraphics{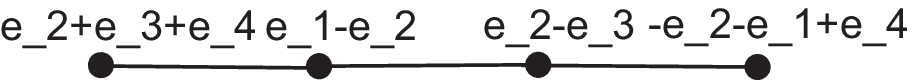}
\end{center}

If we expand $\bar P$ by final $-2$ vectors we obtain $(3)$ in Remark~\ref{r:lisca}\,(\ref{p:-2}). It is immediate to check that no graph in $(3)$ can play the role of $P_{j-1}$ described above. This is so because there is no pair of vectors $\tilde v_0,\tilde v_{t,\beta}$, such that $\tilde v_0$ is an internal vector with $a_{\tilde v_0}\geq 3$ and such that there exists $k$ with $E_k(P_{j-1})=\{0,(t,\beta)\}$.

If we expand $\bar P$ as described in $(1)$ in Remark~\ref{r:lisca}\,(\ref{p:-2}),  there are two possible pairs of vectors $\tilde v_0,\tilde v_{t,\beta}$ such that $\tilde v_0$ is an internal vector with $a_{\tilde v_0}\geq 3$ and such that there exists $k$ with $E_k(P_{j-1})=\{0,(t,\beta)\}$. In both pairs $\tilde v_{t,\beta}$ is final. By symmetry, these two pairs yield graph $(a)$ in the statement. The condition $s>0$ in the statement comes from the requirement $a_{\tilde v_0}\geq 3$ with $\tilde v_{0}\in P_{j-1}$.

Finally, if we expand $\bar P$ as described in $(2)$ in Remark~\ref{r:lisca}\,(\ref{p:-2}), there are  different pairs of vectors $\tilde v_0,\tilde v_{t,\beta}$ such that $\tilde v_0$ is an internal vector with $a_{\tilde v_0}\geq 3$ and such that there exists $k$ with $E_k(P_{j-1})=\{0,(t,\beta)\}$. Now, it is not difficult to check that these pairs yield to the graphs $(b),(c),(d)$ and $(e)$ in the statement.
\end{proof}


\newcommand{\paper}[8]
{\bibitem[#1]{#2} \textsc{#3}, \emph{#4}, #5 \textbf{#6} (#7), #8.}

\newcommand{\preprint}[6]
{\bibitem[#1]{#2} \textsc{#3}, \emph{#4}, preprint (#5), \texttt{#6}.}

\newcommand{\toappear}[6]
{\bibitem[#1]{#2} \textsc{#3}, \emph{#4},  preprint (#5), to appear in #6.}

\newcommand{\book}[5]
{\bibitem[#1]{#2} \textsc{#3}, \textbf{#4}, #5.}

\newcommand{\bookseries}[7]
{\bibitem[#1]{#2} \textsc{#3}, \textbf{#4}, #5 \textbf{#6}, #7.}

\newcommand{\chapt}[7]
{\bibitem[#1]{#2} \textsc{#3}, \emph{#4}, #5 (#6), #7.}

\vspace{1cm}

\end{document}